\date{today}

\documentclass[12pt,a4paper]{amsart}

\input xy
\xyoption{all}

\DeclareMathAlphabet{\mathpzc}{OT1}{pzc}{m}{it}

\usepackage{amsfonts,amsmath,amssymb,indentfirst,mathrsfs,amscd}
\usepackage[mathscr]{eucal}
\usepackage[active]{srcltx}
\usepackage{graphicx}

\usepackage[dvips]{color}
\usepackage{color}

\author[M. Logares]{Marina Logares}
\address{Instituto de Ciencias Matem\'aticas (CSIC-UAM-UC3M-UCM),
C/ Nicolas Cabrera 15, 28049 Madrid, Spain}
\email{marina.logares@icmat.es}

\author[V. Mu\~{n}oz]{Vicente Mu\~{n}oz}
\address{Facultad de Matem\'aticas, Universidad Complutense de Madrid,
Plaza Ciencias 3, 28040 Madrid Spain}
\email{vicente.munoz@mat.ucm.es}

\author[P. Newstead]{P. E. Newstead}
\address{Department of Mathematical Sciences,
University of Liverpool, Peach Street,
Liverpool L69 7ZL, UK}
\email{newstead@liverpool.ac.uk}

\title[Hodge polynomials of character varieties]{Hodge polynomials
of $\SL(2,\CC)$-character varieties for curves of small genus}

\subjclass[2010]{Primary: 14C30. Secondary: 14D20, 14L24, 32J25}

\keywords{Moduli spaces, E-polynomial, character varieties, surface group}
\thanks{The first author is supported by an i-Math Future contract and partially supported by FCT(Portugal) through project PTDC/MAT/099275/2008 and (Spain) project  MTM2010-17717. The second author is supported by (Spanish MICINN) research project MTM2010-17389. The first and third authors would like to thank the Isaac Newton Institute, where this work was completed during the Moduli Spaces programme. }
\date{\today}

\DeclareMathOperator{\Id}{Id\,}             
\DeclareMathOperator{\tr}{Tr\,}             

\DeclareMathOperator{\GL}{GL}

\DeclareMathOperator{\SL}{SL}

\DeclareMathOperator{\PGL}{PGL}

\DeclareMathOperator{\stab}{Stab}
\DeclareMathOperator{\Gr}{Gr}

\DeclareMathOperator{\Tr}{Tr\,}       
\newcommand{\G}{\Gamma}         

\begin{document}

\newtheorem{thm}{Theorem}[section]
\newtheorem{prop}[thm]{Proposition}
\newtheorem{lem}[thm]{Lemma}
\newtheorem{cor}[thm]{Corollary}
\newtheorem{conjecture}{Conjecture}

\theoremstyle{definition}
\newtheorem{defn}[thm]{Definition}
\newtheorem{ex}[thm]{Example}
\newtheorem{as}{Assumption}

\theoremstyle{remark}
\newtheorem{rmk}[thm]{Remark}

\theoremstyle{remark}
\newtheorem*{prf}{Proof}

\newcommand{\iacute}{\'{\i}} 
\newcommand{\norm}[1]{\lVert#1\rVert} 

\newcommand{\lto}{\longrightarrow}
\newcommand{\hra}{\hookrightarrow}

\newcommand{\suchthat}{\;\;|\;\;}
\newcommand{\dbar}{\overline{\partial}}

\newcommand{\cC}{\mathcal{C}}
\newcommand{\cD}{\mathcal{D}}
\newcommand{\cG}{\mathcal{G}} 
\newcommand{\cO}{\mathcal{O}} 
\newcommand{\cM}{\mathcal{M}} 
\newcommand{\cN}{\mathcal{N}} 
\newcommand{\cP}{\mathcal{P}} 
\newcommand{\cS}{\mathcal{S}} 
\newcommand{\cU}{\mathcal{U}} 
\newcommand{\cX}{\mathcal{X}}
\newcommand{\cT}{\mathcal{T}}
\newcommand{\cV}{\mathcal{V}}
\newcommand{\cB}{\mathcal{B}}
\newcommand{\cR}{\mathcal{R}}
\newcommand{\cH}{\mathcal{H}}

\newcommand{\ext}{\mathrm{ext}} 
\newcommand{\x}{\times}

\newcommand{\mM}{\mathscr{M}} 

\newcommand{\CC}{\mathbb{C}} 
\newcommand{\QQ}{\mathbb{Q}} 
\newcommand{\PP}{\mathbb{P}} 
\newcommand{\HH}{\mathbb{H}} 
\newcommand{\RR}{\mathbb{R}} 
\newcommand{\ZZ}{\mathbb{Z}} 

\renewcommand{\lg}{\mathfrak{g}} 
\newcommand{\lh}{\mathfrak{h}} 
\newcommand{\lu}{\mathfrak{u}} 
\newcommand{\la}{\mathfrak{a}} 
\newcommand{\lb}{\mathfrak{b}} 
\newcommand{\lm}{\mathfrak{m}} 
\newcommand{\lgl}{\mathfrak{gl}} 
\newcommand{\too}{\longrightarrow}
\newcommand{\imat}{\sqrt{-1}} 

\hyphenation{mul-ti-pli-ci-ty}

\hyphenation{mo-du-li}

\begin{abstract}
We compute the E-polynomials of the moduli spaces of representations of the
fundamental group of a complex curve into $\SL(2,\CC)$, for the case of small genus
$g$, and allowing the holonomy around a fixed point to be any matrix of $\SL(2,\CC)$, that is
$\Id, -\Id$, diagonalisable, or of either of the two Jordan types.

For this, we introduce a new geometric technique, based on stratifying the space of
representations, and on the analysis of the behaviour of the E-polynomial under
fibrations.
\end{abstract}

\maketitle


\section{Introduction}\label{sec:introduction}

Let $X$ be a smooth complex projective curve of genus $g\geq 1$, and let $G$ be a complex reductive group.
The $G$-character variety  of $X$ is defined as the moduli space of semisimple representations of
$\pi_{1}(X)$ into $G$. The topology of $G$-character varieties has been studied extensively in the past ten years.
The main purpose was to prove certain
conjectures on the mirror symmetry phenomena exhibited in the non-abelian Hodge theory of a curve.

The  $G$-character variety is the space
 $$
  \cM (G)= \{(A_{1},B_{1},\ldots,A_{g},B_{g}) \in G^{2g} \,| \prod_{i=1}^{g}[A_{i},B_{i}]=\Id\}/ / G.
 $$
For complex linear groups $G=\GL(n,\CC), \SL(n,\CC)$, the representations of $\pi_1(X)$ into $G$ can be understood
as $G$-local systems $E\to X$, hence defining
a flat bundle $E$ which has $\deg E=0$. A natural generalization consists of allowing bundles $E$ of
non-zero degree $d$. The $G$-local systems on $X$ correspond to representations $\rho:\pi_1(X\setminus \{p_0\}) \to G$,
where $p_0\in X$ is a fixed point, and $\rho(\gamma)= -\frac{d}{n} \Id$, $\gamma$ a loop around $p_0$, giving
rise to the moduli space
 $$
  \cM^d (G)= \{(A_{1},B_{1},\ldots,A_{g},B_{g}) \in G^{2g} \,| \prod_{i=1}^{g}[A_{i},B_{i}]=e^{2\pi i d/n} \Id \}/ / G.
 $$

The space $\cM^d(G)$ is known in the literature as the Betti moduli space.
This space is closely related to two other spaces: the De Rham and Dolbeault moduli spaces. The De Rham moduli space
$\cD^d(G)$ is the moduli space parameterizing flat bundles, i.e., $(E,\nabla)$, where $E\to X$ is an algebraic bundle of
degree $d$ and rank $n$ (and fixed determinant in the case $G=\SL(n,\CC)$), and
$\nabla$ is an algebraic connection on $X\setminus \{p_0\}$ with a logarithmic pole at
$p_0$ with residue $-\frac{d}{n} \Id$. The Riemann-Hilbert correspondence
\cite{deligne:1970,simpson:1995} gives an analytic isomorphism (but not an algebraic isomorphism)
$\cD^d(G) \cong \cM^d(G)$.

The Dolbeault moduli space $\cH^d(G)$ is the moduli space of $G$-Higgs
bundles $(E,\Phi)$, consisting of an algebraic bundle
$E\to X$ of degree $d$ and rank $n$ (also with fixed determinant in the case $G=\SL(n,\CC)$),
and a homomorphism $\Phi:E\to E\otimes K_X$ (known as the Higgs field - in the case $G=\SL(n,\CC)$, the Higgs field has trace $0$). In this situation the
theory of harmonic bundles \cite{corlette:1988,simpson:1992} gives a homeomorphism $\cM^d(G)\simeq \cH^d(G)$.

When $\gcd(n,d)=1$ these moduli spaces are smooth and the underlying differential manifold is
hyperk\"{a}hler, but the complex structures do not correspond under the previous isomorphisms.
The cohomology of these moduli spaces has been computed in several particular cases, but mostly
from the point of view of the Dolbeault moduli space $\cH^{d}(G)$.
Poincar\'e polynomials for $G=\SL(2,\CC)$ were computed by Hitchin in his seminal paper on Higgs
bundles \cite{hitchin:1987} and for $G=\SL(3,\CC)$ by Gothen in \cite{gothen:1994}.
More recently the techniques involved in the computations by Gothen and Hitchin have been improved
to compute the classes in the completion of the Grothendieck ring for these varieties in the $G=\GL(4,\CC)$ case, and from their computations it is also possible to deduce the compactly supported Hodge polynomials \cite[Section 1.2]{garciaprada-heinloth-schmitt:2011}.

Hausel and Thaddeus \cite{hausel-thaddeus:2003} gave a new perspective for the topological study of these
varieties giving the first non-trivial example of the Strominger-Yau-Zaslow Mirror Symmetry using the so called
Hitchin system \cite{hitchin:1987-2} for the Dolbeault moduli space. They conjectured also (and checked for
$G=\SL(2,\CC)$ and $G=\SL(3,\CC)$ using the previous results by Hitchin and Gothen) that a version of the
topological Mirror Symmetry holds, i.e., some Hodge numbers $h^{p,q}$ of $\cH^d(G)$ and $\cH^d(G^{L})$, for $G$ and
Langlands dual $G^{L}$, agree. It was noticed that while the mixed Hodge structures of $\cH^{d}(G)$
and $\cD^{d}(G)$ agree \cite[Theorem 6.2]{hausel-thaddeus:2003} and are pure \cite{MMehta:2001,hausel:2004}
this is not the case for $\cM^{d}(G)$. This fact motivates the
study of E-polynomials of the $G$-character varieties.

Hausel and Rodriguez-Villegas started the computation of the E-polynomials of $G$-character
varieties focusing on $G=\GL(n,\CC), \SL(n,\CC)$ and $\PGL(n,\CC)$ using arithmetic methods inspired on the Weil
conjectures. In \cite{hausel-rvillegas:2007} they obtained the E-polynomials of $\cM^{d}(G)$
for $G=\GL(n,\CC)$ in terms of a simple generating function.
Following the methods of Hausel and Rodriguez-Villegas, Mereb \cite{mereb:2010} studied the case of
$\SL(n,\CC)$ giving
an explicit formula for the E-polynomial in the case $G=\SL(2,\CC)$, while for $\SL(n,\CC)$ these polynomials are given in terms of a generating function. He proved also that the E-polynomials for $\SL(n,\CC)$ are palindromic and monic (so the $\SL(n,\CC)$-character varieties are connected).

Another direction of interest is the moduli spaces of parabolic bundles. Let $X$ be a complex curve with $s$ marked
points $p_1,\ldots,p_s$. For $i\le s$, fix conjugacy classes $\cC_i\subset G$ given by semisimple elements. The corresponding Betti moduli space of parabolic
representations (or \textit{parabolic character variety}) is
 \begin{eqnarray*}
 P\cM^{\cC_1,\ldots, \cC_s}(G) &:=& \{ (A_{1},B_{1},\ldots, A_{g},B_{g},C_{1}, \ldots, C_{s})\in G^{2g+s} \
 |  \\
 & & \ \ \prod_{i=1}^{g}[A_{i},B_{i}]\prod_{j=1}^s C_j = \Id , \ C_j \in \cC_j, \ j=1,\ldots,s \} / /G\, .
\end{eqnarray*}

In \cite{simpson:1990} Simpson proved that this space is analytically isomorphic to the moduli space of flat
logarithmic $G$-connections 
and homeomorphic to a moduli space of Higgs bundles with parabolic structures at $p_1,\ldots,p_s$.
Very recently, formulae for the E-polynomials of the parabolic character varieties for $G=\GL(n,\CC)$ and
generic semisimple $\cC_1,\ldots,\cC_s$ have been obtained by Hausel, Letellier and Rodriguez-Villegas
\cite{hausel-letellier-rvillegas:2011} using arithmetic methods.
In addition, the Poincar\'e polynomials have been
computed in the context of parabolic Higgs bundles for $G=\SL(2,\CC)$ in \cite{by} and for
$G=\SL(3,\CC)$ and $\GL(3,\CC)$ in \cite{garciaprada-gothen-munoz:2004}.

In this paper we consider certain character varieties for the group $G=\SL(2,\CC)$. For any element $\xi$ belonging to a conjugacy class $\cC\subset \SL(2,\CC)$, we define
 \begin{eqnarray*}
  \cM_\xi &=& \{(A_{1},B_{1},\ldots,A_{g},B_{g}) \in \SL(2,\CC)^{2g} \,| \prod_{i=1}^{g}[A_{i},B_{i}]=\xi\}/ / \stab(\xi) \\
 &=& \{(A_{1},B_{1},\ldots,A_{g},B_{g}) \in \SL(2,\CC)^{2g} \,| \prod_{i=1}^{g}[A_{i},B_{i}] \in \cC\}/ / G \, .
 \end{eqnarray*}
This space has the structure of an algebraic variety independent of the complex structure on $X$.
When $\cC$ is semisimple, $\cM_\xi$ is a parabolic character variety as defined earlier, but we do not
make such restriction on $\cC$. Also we do not assume that $\cC$ satisfies a genericity condition.

Our basic aim is to compute the E-polynomial (or Hodge-Deligne polynomial) of $\cM_\xi$.
For this purpose, we propose a geometric technique,
which allows us to deal with the cases in which $\xi$ is presented in Jordan form. Our method depends
on certain basic properties of E-polynomials, in particular the additive property (Proposition \ref{thm:basics}(i)),
which allow us to compute these polynomials using stratifications. We also require methods for handling fibrations
which are locally trivial in the analytic topology but not in the Zariski topology (Propositions \ref{basics2},
\ref{prop:fibration} and \ref{prop:e-total-space}). Our main results are explicit formulae for the E-polynomials
for $g=1$ and $g=2$ for the five
different moduli spaces  corresponding to $\xi$ being $\Id$, $-\Id$, diagonal with different
eigenvalues $\xi_\lambda=\left( \begin{array}{cc} \lambda & 0\\ 0 & \lambda^{-1} \end{array}\right)$, and of
either of the two Jordan types $J_+=\left( \begin{array}{cc} 1 & 1\\ 0 & 1\end{array}\right)$,
$J_-=\left( \begin{array}{cc} -1 & 1\\ 0 & -1\end{array}\right)$. These results can be summarised as follows:

\begin{thm}\label{thm11}
  Let $X$ be a complex curve of genus $g=1$. Then the E-polynomials of $\cM_\xi$ are as follows:
  \begin{eqnarray*}
    e(\cM_{\Id}) &=& q^2+1\,, \\
    e(\cM_{-\Id}) &=& 1\, ,\\
    e(\cM_{J_+}) &=&q^2-2q-3 \, , \\
    e(\cM_{J_-}) &=& q^2 +3q\, ,\\
    e(\cM_{\xi_\lambda}) &=& q^2+4q+1\, ,
  \end{eqnarray*}
where $q=\, uv$, $e(\cM_\xi)\in \ZZ[u,v]$.
\end{thm}

\begin{thm}\label{thm12}
  Let $X$ be a complex curve of genus $g=2$. Then
  \begin{eqnarray*}
    e(\cM_{\Id}) &=& q^6+17q^4 + q^2+1\, ,\\
    e(\cM_{-\Id}) &=& q^6-2q^4-30q^3-2q^2+1\, ,\\
    e(\cM_{J_+}) &=& q^8  -3 q^6 -4 q^5 - 39 q^4-4 q^3 -15 q^2 \, ,\\
    e(\cM_{J_-}) &=& q^8-3 q^6+15 q^5+6 q^4+45 q^3 \, , \\
    e(\cM_{\xi_\lambda}) &=& q^8 + q^7- 2q^6 + 13q^5-26q^4+13q^3 -2 q^2+q+1 \, .
  \end{eqnarray*}
\end{thm}

The layout of the paper is as follows. Section  \ref{sec:method} consists of a review of the basic properties of E-polynomials, including a description of the methods  which we use for computing these polynomials. This is followed in section \ref{sec:SL-and-GL} by a computation of the E-polynomials for the groups $\GL(2,\CC)$, $\SL(2,\CC)$ and $\PGL(2,\CC)$ and also for the stratification of $\SL(2,\CC)$ by conjugacy class types. In section \ref{sl2}, we stratify $\SL(2,\CC)^2$ and compute the E-polynomials for each stratum. In section \ref{charvarg=1}, we use these computations to prove Theorem \ref{thm11}. It turns out that we can do slightly better and compute all Hodge numbers of $\cM_\xi$ (except in the case $\xi=J_-$).

Sections \ref{sec:representation} and \ref{sec:representation2} are key sections in which we introduce the \emph{Hodge monodromy representation}. For this, we define first
$$\overline{X}_{4}:=\left\{(A,B,\lambda)\, | \, AB=\left(\begin{array}{cc}
      \lambda & 0 \\
      0 & \lambda^{-1}\\
    \end{array}
  \right) BA, \text{ for some } \lambda\ne0, \pm1,\, A,\,B \in \SL(2,\CC)\,\right\}.$$
The map $(A,B,\lambda)\mapsto\lambda$ fibres this space over $\CC\setminus\{0,\pm1\}$, but the fibration is only analytically locally trivial. The idea of the Hodge monodromy representation is that it encodes in a convenient form all the information on the E-polynomials of both
the fibre and the total space. We compute this for $\overline{X}_4$ in section \ref{sec:representation} and extend it to $\overline{X}_4/\ZZ_2$ in section \ref{sec:representation2} (here $\ZZ_2$ acts on $\overline{X}_4$ by interchange of basis vectors). In sections \ref{sec:g=2}--\ref{sec:J-}, we prove Theorem \ref{thm12}.

Our arguments can certainly be extended to higher genus, but the computations become rapidly more complicated as
$g$ increases. Another advantage of a geometrical method of this type is that it may allow
one to compute the motives of the character
varieties. This will be pursued in future work.

\noindent\textbf{Acknowledgements}
We thank Tam\'as Hausel and Richard Thomas for helpful comments.
In particular, several conversations with T.\ Hausel
have been invaluable to confirm the correctness of our polynomials. We also thank the referees for their careful reading of the paper.

\section{E-polynomials}\label{sec:method}

In this section we recall some properties of E-polynomials, which include methods
for computing them using stratifications and fibrations.
We start by reviewing the Hodge theory of algebraic varieties over $\CC$ (for details see \cite{Deligne2, Deligne3}).

A {\em pure Hodge structure of weight $k$} consists of a finite dimensional complex vector space
$H$ with a real structure, and a decomposition
 $$
 H=\bigoplus_{k=p+q} H^{p,q}
 $$
such that $H^{q,p}=\overline{H^{p,q}}$, the bar meaning complex conjugation on $H$.
A Hodge structure of weight $k$ gives rise to the so-called {\em Hodge filtration} which is a descending filtration
$F^{p}=\bigoplus_{s\ge p}H^{s,k-s}$. We define $\Gr^{p}_{F}(H):=F^{p}/ F^{p+1}=H^{p,k-p}$.

A {\em mixed Hodge structure} consists of a finite dimensional complex vector space $H$ with a real structure,
an ascending (weight) filtration $\ldots \subset W_{k-1}\subset W_k \subset \ldots \subset H$
(defined over $\RR$) and a descending (Hodge) filtration $F$ such that $F$ induces a pure Hodge structure
of weight $k$ on each $\Gr^{W}_{k}(H)=W_{k}/W_{k-1}$. We define
 $$
 H^{p,q}:= \Gr^{p}_{F}\Gr^{W}_{p+q}(H)
 $$
and write $h^{p,q}$ for the {\em Hodge number}
 $$
 h^{p,q} :=\dim H^{p,q}.
 $$

\medskip

Let $Z$ be any quasi-projective algebraic variety (maybe non-smooth or non-compact). 
The cohomology groups $H^k(Z):=H^k(Z;\CC)$ and the cohomology groups with compact support  $H^k_c(Z)$ are endowed with mixed Hodge structures. Moreover, these
Hodge structures are pure of weight $k$ if $Z$ is smooth and projective.
We define the {\em Hodge numbers}  by
 \begin{eqnarray*}
 h^{k,p,q}(Z)&=&h^{p,q}(H^k(Z))=\dim \Gr^{p}_{F}\Gr^{W}_{p+q}H^{k}(Z) ,\\
 h^{k,p,q}_{c}(Z)&=&h^{p,q}(H_{c}^k(Z))=\dim \Gr^{p}_{F}\Gr^{W}_{p+q}H^{k}_{c}(Z) .
 \end{eqnarray*}
When $Z$ is smooth of dimension $n$, Poincar\'e duality
provides an equality $h^{k,p,q}=h^{2n-k,n-p,n-q}_c$.

We consider the Euler characteristic
 $$
 \chi^{p,q}_{c}(Z)=\sum_{k}(-1)^{k}h^{k,p,q}_{c}(Z)
 $$
and define the E-polynomial by
 $$
 e(Z)=e(Z)(u,v):=\sum _{p,q} \chi_{c}^{p,q}(Z)u^{p}v^{q}.
 $$

Hence, when $Z$ is smooth and projective of dimension $n$, Poincar\'e duality implies that
$$
e(Z)(u,v)=(uv)^n e(Z)\left(\, \frac{1}{u},\frac{1}{v}\, \right).
$$

\begin{rmk}\label{rmk:q=uv}
Note that, when $h_c^{k,p,q}=0$ for $p\neq q$, the polynomial $e(Z)$ depends only on the product $uv$.
This will happen in all the cases that we shall investigate here. In this situation, it is
conventional to use the variable
 $$
  q:=uv.
  $$
If this happens, we shall say that the variety is {\it of balanced type}.
\end{rmk}

\begin{rmk}
In the literature $e(Z)(-u,-v)$ is sometimes referred to as the Hodge-Deligne polynomial of $Z$. The polynomial $H_c(Z)(u,v,t):=\sum_{p,q,k}h_c^{p,q,k}u^pv^qt^k$, which encodes all the Hodge numbers, is known as the (compactly supported) \emph{mixed Hodge polynomial}. (In the case of balanced type, we write also $H_c(Z)(q,t)$.) This generalises both the E-polynomial and the Poincar\'e polynomial. The E-polynomial is especially amenable to computation precisely because it is defined by an Euler characteristic. The preprint \cite{hausel-letellier-rvillegas:2011} is built around a conjectural formula for $H_c(P\cM^{\cC_1,\ldots, \cC_s}(\GL(n,\CC)))$ in the case of generic semisimple conjugacy classes.
\end{rmk}

The key property of E-polynomials that permits their calculation is that they are additive for
stratifications of $Z$. This and some other properties that we shall need are
summarised in the following propositions.

\begin{prop} \label{thm:basics}
Let $Z$ be a complex algebraic variety.
\begin{itemize}
\item[(i)] If $Z=\bigsqcup_{i=1}^{n}Z_{i}$, where all $Z_i$ are locally closed in $Z$, then $e(Z)=\sum_{i=1}^{n}e(Z_{i})$,
\item[(ii)] $e(\CC^n)=q^n$,
\item[(iii)] $e(\PP^{n-1})=1+q +q^{2}+\cdots +q^{n-1}=\frac{1-q^{n}}{1-q}$,
\item[(iv)] $e(\CC^n\setminus\{0\})=q^n-1$.
\end{itemize}
\end{prop}

\begin{proof}
(i) This basic fact follows from \cite[(8.3.9]{Deligne3} (see also \cite[Theorem 2.2]{mov1}).

(ii) and (iii) follow from the known cohomological structure of $\CC^n$ and $\PP^{n-1}$.

(iv) follows from (i) and (ii).
\end{proof}

\begin{prop}\label{basics2}
Suppose that $B$ is connected and $\pi:Z\to B$ is an algebraic fibre bundle with fibre $F$ (not necessarily locally trivial in the Zariski topology) and that the action of $\pi_1(B)$ on $H_c^*(F)$ is trivial. 
Suppose that $Z,F,B$ are smooth. Then $e(Z)=e(F)e(B)$.
\end{prop}

\begin{proof}
The Leray spectral
sequence for cohomology has $E_2$-term equal to $E_2^{l,m}=H^l(B)\otimes H^m(F)$ and abuts to $H^*(Z)$.
By \cite{a}, this spectral sequence has a mixed Hodge structure (actually, the 
result of \cite{a} is for proper maps, but the proof works verbatim for non-proper
maps which are topologically fiber bundles). Let us compute this mixed Hodge structure.

We apply functoriality to the map between fibrations $F\to F\to \{pt\}$ to $F\to E\to B$,
to get that the isomorphism $E_2^{0,m}\to H^m(F)$ preserves mixed Hodge structures.
We apply now functoriality to the map from $F\to E\to B$ to $\{pt\} \to B\to B$,
to get that the isomorphism $H^l(B) \to E_2^{l,0}$ preserves mixed Hodge structures.
Finally, recall that $E_2^{*,*}$ has a natural $H^*(B)$-module structure. This is induced by the
maps of fibrations $F \to E\to B$ to $F \to E\x B \to B\x B$, where the map $B\to B\x B$ is the
diagonal. The induced map $E_2^{*,*} \otimes  H^*(B)\too E_2^{*,*}$ preserves then the mixed Hodge
structures. Putting all together, $E_2^{*,*} \cong H^*(B)\otimes H^*(F)$ is an isomorphism of mixed
Hodge structures.

Taking now Euler characteristics, and we have that $e(H^*(F)) e(H^*(B))=e(H^*(Z))$.
As the spaces are smooth, we can go to compactly supported cohomology and get $e(F)e(B)=e(Z)$.
\end{proof}

\begin{rmk}\label{rem:new}
The hypotheses of Proposition \ref{basics2} hold in particular in the following cases:
\begin{itemize}
\item $B$ is irreducible and $\pi$ is locally trivial in the Zariski topology. (The key point here is that there is an open covering $\bigcup U_k$ of $B$ such that the monodromy action of each $\pi_1(U_k)$ is trivial and $\bigcap U_k$ is connected, allowing us to apply van Kampen's theorem.)
\item $\pi:Z\to B$ is a morphism between quasi-projective varieties which is a locally trivial fibre bundle for the analytic topology and $F$ is a projective space $\PP^{N}$ \cite[Lemma 2.4]{mov2}.
\item $\pi$ is a principal $G$-bundle with $G$ a connected algebraic group. (In this case, any loop in $B$ based at $x_0$ lifts to an automorphism of the fibre $F_{x_0}$ induced by the action of an element of $G$; since $G$ is connected, this isomorphism is homotopic to the identity.)
\end{itemize}
\end{rmk}

We will also require a method of calculating $e(Z/\ZZ_2)$ when $\ZZ_2$ acts on $Z$.
In these circumstances, we denote by $H^*_c(Z)^\pm$ the $\pm$-invariant part of $H^*_c(Z)$,
with a corresponding notation
 $$
 e(Z)^\pm:= e(H^*_c(Z)^\pm).
 $$
Note that
 \begin{eqnarray*}
  e(Z)^+ &=&e(Z/\ZZ_2), \\
  e(Z)^- &=& e(Z)-e(Z)^+\, .
 \end{eqnarray*}
For example, if $Z=\CC^*$ with the action $\lambda\mapsto \lambda^{-1}$, then $Z/\ZZ_2\cong \CC$,
and
\begin{equation}\label{eqcc}
e(\CC^*)^+=q,\ \ e(\CC^*)^-=-1.
\end{equation}

\begin{prop}\label{prop:fibration}
Let
 $$
 \xymatrix{
 F\ar[d]\ar[r]^{=} & F \ar[d]\\
 {Z} \ar[d]^{\widetilde{\pi}}\ar[r]& Z/\ZZ_2 \ar[d]^{\pi}\\
 B\ar[r]^{\hspace{-10pt}2:1} &B/\ZZ_{2}}
 $$
be a diagram of fibrations, where $B$ is smooth, irreducible, $\widetilde{\pi}$ and $\pi$ are smooth morphisms,
$\widetilde{\pi}$ is a locally trivial fibration in the analytic topology and the monodromy action
of $\pi_1(B)$ on $H^*_c(F)$ is trivial.
Then
 \begin{equation}\label{eqn:HP}
 e(Z/\ZZ_2)=e(F)^{+}e(B)^{+}+e(F)^{-}e(B)^{-},
 \end{equation}
where $e(F)^\pm$ are defined by an action of $\ZZ_2$ on $H^*_c(F)$, which is compatible with the Hodge structure.
\end{prop}

\begin{proof}
By hypothesis, the monodromy action of $\pi_1(B/\ZZ_2)$ on $H^*_c(F)$
factors through $\ZZ_2$ and therefore induces a splitting $H^*_c(F)=H^*_c(F)^+\oplus H^*_c(F)^-$.
This corresponds to a splitting $H^*(F)=H^*(F)^+\oplus H^*(F)^-$
through Poincar\'e duality.
We consider the spectral sequences of the fibrations $\pi$ and $\widetilde{\pi}$ and the restriction map between their $E_2$-terms:
 $$
 H^i(B/\ZZ_2,H^j(F))\longrightarrow H^i(B,H^j(F))=H^i(B)\otimes H^j(F).
 $$
By \cite{a} the $E_2$-terms have mixed Hodge structures, and the map is compatible with them.
This map induces an isomorphism
  $$
  H^i(B/\ZZ_2,H^j(F))\stackrel{\cong}{\longrightarrow} (H^i(B)\otimes H^j(F))^+=
  (H^i(B)^+\otimes H^j(F)^+)\oplus (H^i(B)^-\otimes H^j(F)^-).
  $$
The first spectral sequence abuts to $H^*(Z/\ZZ_2)$ and the isomorphism is compatible with mixed Hodge structures.
Since all differentials in the spectral sequence go between two groups for which $i+j$ has opposite parity,
it follows that any degeneration in the spectral sequence leads to the cancellation of equal terms in the Hodge
polynomial. So $e(H^*(Z/\ZZ_2))=e(H^*(F)^{+})e(H^*(B)^{+})+e(H^*(F)^{-})e(H^*(B)^{-})$. Going to compactly
supported cohomology, the result follows.
\end{proof}

\begin{rmk} \label{rmk:e(Z/Z_2)-situations}
In the situation of Proposition \ref{prop:fibration}, we have an action of $\ZZ_2$ on $Z$, and
we can write
 \begin{eqnarray*}
  e(Z)^+ &= & e(Z/\ZZ_2)=e(F)^{+}e(B)^{+}+e(F)^{-}e(B)^{-} \\
 e(Z)^- &=& e(Z)-e(Z/\ZZ_2)=e(F)^{+}e(B)^{-}+e(F)^{-}e(B)^{+}.
 \end{eqnarray*}
\end{rmk}

For the next proposition, recall that a variety is
{of balanced type} if all Hodge numbers $h_c^{k,p.q}$ with $p\ne q$ are zero.

\begin{prop} \label{prop:pure-type}
\begin{itemize}
\item[(i)]
  Let $Y=Z\sqcup U$, where $Z\subset Y$ is closed. Then, if
  two of the spaces $Y,Z,U$ are of balanced type, so is the third.
  \item[(ii)] If $\pi:Z\to B$ is as in Proposition \ref{basics2} and $B$ and $F$ are of balanced type, so is $Z$.
  \item[(iii)] If $\ZZ_2$ acts on $Z$ and $Z$ is of balanced type, so is $Z/\ZZ_2$. 
  \end{itemize}
\end{prop}

\begin{proof}
(i) Consider the long exact sequence of mixed Hodge structures
  $$\ldots \to H^k_c(U) \to H^k_c(Y) \to H^k_c(Z) \to H^{k+1}_c(U) \to \ldots$$
and take the $(p,q)$-components. Using the fact that the functor $H\mapsto H^{p,q}$ is exact \cite{Deligne2},
we get an exact sequence
  $$\ldots \to H^{k,p,q}_c(U) \to H^{k,p,q}_c(Y) \to H_c^{k,p,q}(Z) \to H^{k+1,p,q}_c(U) \to \ldots$$
The result follows.

(ii) This follows from the fact that the Leray spectral sequence respects mixed Hodge structures
(see the proof of Proposition \ref{basics2}).

(iii) is obvious since $h_c^{k,p,q}(Z/\ZZ_2)=h_c^{k,p,q}(Z)^+$. 
\end{proof}

\begin{rmk}
The effect of this proposition is that all the spaces we consider will be of balanced type.
\end{rmk}

We shall also deal with fibrations which are locally trivial in the analytic topology, for
which we want to compute the E-polynomial of the total space. Consider a fibration
  \begin{equation}\label{eqn:fibration}
  F \too Z \stackrel{\pi}{\too} B=\CC\setminus \{p_1,\ldots, p_\ell\},
  \end{equation}
with basis the complex line minus a finite set of points, and with fibre $F$ which
is a smooth variety of balanced type. 
The fibration defines a local system
 $$
 \cH^k_c \to B
 $$
whose fibres are the cohomology groups $H^k_c(F_b)$, $b\in B$, $F_b:=\pi^{-1}(b)$. These fibres
possess mixed Hodge structures, thus defining a variation of Hodge structures.
The subspaces $W_t (H^k_c(F_b))$, $b\in B$, define a subbundle. As this is defined over
the rational numbers, it is locally constant, i.e., a local subsystem.
Therefore, the holonomy preserves $W_t$, and hence it induces a variation of Hodge structures on
the pure Hodge structure $\Gr^W_t H^k_c(F_b)$. Now we use the assumption that $F$ is of balanced type.
Hence we must have $t=2p$ and $\Gr^W_t H^k_c(F_b)=H^{k,p,p}_c(F_b)$.

Fix a base point $b_0\in B$, and write $F=F_{b_0}$.
Associated to the fibration, there is a \textit{monodromy representation}
  $$
  \rho: \pi_1(B,b_0) \to \GL(H^{k,p,p}_c(F))\, .
  $$
We shall assume that $\rho$ factors through the homology $H_1(B)$ (in other words, the monodromy is abelian). Write
 $$
 \Gamma := H_1(B)= \ZZ\langle \gamma_1,\ldots,\gamma_\ell \rangle\, ,
 $$
where $\gamma_i$ is a loop around $b_i$. Then the monodromy representation is given by
 \begin{equation}\label{eqn:exxtra}
 \rho: H_1(B)\to \GL(H^{k,p,p}_c(F))\, .
 \end{equation}

Consider the representation ring $R(\Gamma)$. Then the $H_c^{k,p,p}(F)$ are modules over $R( \Gamma)$.
In particular, there is a well defined element
  \begin{equation}\label{eqn:Hodge-mon-rep}
  R(Z) := \sum (-1)^k H_c^{k,p,p}(F)\,  q^p \in R(\Gamma)[q] \, ,
  \end{equation}
which we shall call the {\em Hodge monodromy representation} of $\pi$.

We want to recover $e(Z)$ from the information of the Hodge monodromy representation.
We introduce the notation $e(F)^{inv}$
to denote the E-polynomial of the invariant part of the monodromy representation.

Suppose that the monodromy representation (\ref{eqn:exxtra}) has finite image,
so $\rho$ has finite order around each $p_i$. Then there is a finite covering
 $$
  \psi: B' \to B=\CC\setminus \{p_1,\ldots, p_\ell\}
 $$
from a complex curve $B'$ so that the pull-back fibration
 \begin{equation}\label{eqn:fibration-bis}
 \begin{array}{ccc}
 Z' & \too & Z \\
 \pi'\downarrow \,\,\,\,\, & & \pi\downarrow\,\,\,\, \\
 B' & \too & B
 \end{array}
 \end{equation}
has trivial monodromy. 

\begin{prop} \label{prop:e-total-space}
 Suppose that \eqref{eqn:fibration} is a fibration with $F$ smooth of balanced type and that the monodromy is abelian with finite image. Suppose further that $B'$ is a curve of balanced type (i.e.,\ it is a rational curve). 
 Then $Z$ is of balanced type and its E-polynomial is given by
 $$
 e(Z) = (q-1)\, e(F)^{inv} - (\ell-1) \,  e(F) \, .
 $$
\end{prop}

\begin{proof}
 The Leray spectral sequence of the fibration (\ref{eqn:fibration})
 has $E_2$-term
  \begin{equation}\label{eqn:E2-term}
   E_2^{l,m}(Z)=H^l(B,H^m(F))
  \end{equation}
 and abuts to $H^k(Z)$, where $k=l+m$. We have
 non-zero terms only for $l=0,1$.
To compute (\ref{eqn:E2-term}), recall that the
 cohomology of the local system $H^m(F)$ is given by the cohomology of the complex
 \begin{eqnarray*}
 d: H^{m}(F) & \too & \bigoplus_{i=1}^{\ell} \gamma_i \otimes H^{m}(F) \\
   x &\mapsto & \sum \gamma_i \otimes ( \rho(\gamma_i)(x) - x) \, .
 \end{eqnarray*}
 Therefore
  \begin{align}
   & E_2^{0,m}(Z)=H^0(B) \otimes H^m(F)^{inv}, \label{eqn:dddde}\\
   & E_2^{1,m}(Z)= (H^1(B) \otimes H^m(F)) / d(H^m(F)^{non-inv}), \label{eqn:dddde2}
  \end{align}
where $H^m(F)^{non-inv}$ denotes a complement of $H^m(F)^{inv}$.
  
Now let us compute the mixed Hodge structure of $E_2^{l,m}(Z)$. The result of \cite{a} asserts
its existence, and we use the functorial property for computing it. By the proof of Proposition 
\ref{basics2}, the mixed Hodge structure associated to $\pi'$ in (\ref{eqn:fibration-bis}) is
the product mixed Hodge structure $E_2^{1,m}(Z')=H^l(B')\otimes H^m(F)$. By our
assumption, $B'$ and $F$ are of balanced type, so $Z'$ is of balanced type. There is
a map $E_2^{l,m}(Z) \to E_2^{l,m}(Z')$, which preserves the mixed Hodge structures.
This map is easily seen to be injective. In particular, $Z$ is of balanced type. 

$E_2^{0,m}(Z')=H^0(B') \otimes H^m(F)$ has a component of weight $(p,p)$, for a component of
weight $(p,p)$ in $H^m(F)$. This gives a component of weight $(p,p)$ in $E_2^{0,m}(Z)$.
So (\ref{eqn:dddde}) is an equality of mixed Hodge structures.
The line (\ref{eqn:dddde2}) is worked out analogously: by our assumption $H^1(B')$ has weight $(1,1)$.
The map $H^1(B) \hookrightarrow H^1(B')$ respects mixed Hodge structures, so the mixed Hodge
structure of $E_2^{l,m}(Z)$ is induced by putting weight $(1,1)$ to $H^1(B)$ in (\ref{eqn:dddde2}).

The E-polynomial of $E_2^{*,*}(Z)$ is thus $e(H^*(F)^{inv}) - q ((\ell-1) e(H^*(F))+e(H^*(F)^{inv}))$. 
Going back to compactly supported cohomology, we get
 $$
  e(Z) =(q-1)\, e(F)^{inv} - (\ell-1) e(F) \, .
 $$
\end{proof}

\section{E-polynomials of $\GL(2,\CC)$ and $\SL(2,\CC)$}\label{sec:SL-and-GL}

We start with the following very simple computation.

\begin{lem}\label{lem:HPgroups}
The E-polynomials for the algebraic groups $\GL(2,\CC)$, $\SL(2,\CC)$ and $\PGL(2,\CC)$ are
\begin{itemize}
\item[(i)] $e(\GL(2,\CC))=q(q+1)(q-1)^{2}=q^4-q^3-q^2+q$,
\item[(ii)] $e(\PGL(2,\CC))=q(q+1)(q-1)=q^3-q$,
\item[(iii)] $e(\SL(2,\CC))=q(q+1)(q-1)=q^3-q$.
\end{itemize}
\end{lem}

\begin{proof}
(i) Consider the following locally trivial fibration in the Zariski topology
 \begin{eqnarray*}
 \CC^{2}\setminus\CC\longrightarrow \GL(2,\CC) &\longrightarrow& \CC^{2}\setminus\{0\}\\
 \left(
     \begin{array}{cc}
       a & b \\
       c & d \\
     \end{array}
   \right)&\mapsto&(a,b)
 \end{eqnarray*}
It follows from Propositions \ref{thm:basics} and \ref{basics2} that
$$e(\GL(2,\CC))=
e(\CC^{2}\setminus\CC)e(\CC^{2}\backslash\{0\})=(q^{2}-q)(q^{2}-1).$$

(ii) Notice that
 $$
 \CC^{\ast}\longrightarrow \GL(2,\CC)\longrightarrow \PGL(2,\CC)
 $$
satisfies the hypotheses of Proposition \ref{basics2}. So $$e(\PGL(2,\CC))=e(\GL(2,\CC))/e(\CC^*)=q(q^2-1).$$

(iii) As in (i), we have a fibration
 \begin{eqnarray*}
 \CC\longrightarrow \SL(2,\CC) &\longrightarrow& \CC^{2}\setminus\{0\}\\
 \left(
     \begin{array}{cc}
       a & b \\
       c & d \\
     \end{array}
   \right)&\mapsto&(a,b)
 \end{eqnarray*}
which is locally trivial in the Zariski topology.
Hence
 $
  e(\SL(2,\CC))= q(q^2-1).
 $
\end{proof}

\subsection{E-polynomial of coset spaces of $\GL(2,\CC)$} Let
 \begin{eqnarray*}
   D&:=&\{\mbox{diagonal matrices in }\GL(2,\CC)\}\cong\CC^*\times\CC^*, \\
   U &:=& \left\{\left(\begin{array}{cc}\mu&\nu\\0&\mu\\
   \end{array}\right)\in\GL(2,\CC)\right\}\cong\CC^*\times\CC.
 \end{eqnarray*}
Note that $\ZZ_2$ acts on $\GL(2,\CC)$ by switching the rows, i.e., multiplication on the left by
$\left(\begin{array}{cc}0&1\\1&0\\\end{array}\right)$, and that this action descends to the left coset space $\GL(2,\CC)/D$.

\begin{prop}\label{propGL}
\begin{itemize}
\item[(i)] $e(\GL(2,\CC)/D)=q(q+1)=q^2+q$,
\item[(ii)] $e(\GL(2,\CC)/D)^+=q^2$ and $e(\GL(2,\CC)/D)^- = q$,
\item[(iii)] $e(\GL(2,\CC)/U)=q^2-1$.
\end{itemize}
\end{prop}

\begin{proof}
(i)
The map $\left(\begin{array}{cc}a&b\\c&d\\\end{array}\right)\mapsto ([a:b],[c:d])$ defines an isomorphism $\GL(2,\CC)/D\cong(\PP^1\times\PP^1)\setminus\Delta$,
where $\Delta$ denotes the diagonal.
The E-polynomial is therefore $e(\PP^1)^2-e(\PP^1)=q(q+1)$.

(ii) By the above,
$$(\GL(2, \CC)/D)/\ZZ_{2}\cong ((\PP^1\times\PP^1)\setminus\Delta)/\ZZ_2=
\PP^{2}\setminus C,$$ where $C$ is a conic. So
$$e(\GL(2,\CC)/D)^+=e((\GL(2,\CC)/D)/\ZZ_2)=e(\PP^2)-e(C)=e(\PP^2)-e(\PP^1)=q^2.$$

(iii) We have a fibration
 $$
 \CC^*\longrightarrow\GL(2,\CC)/U\longrightarrow\PP^1,
 $$
given by $\left(\begin{array}{cc} a&b \\c&d \end{array}\right)\mapsto[c:d]$.
This fibration is locally trivial in the Zariski topology, so
 $$e(\GL(2,\CC)/U)=e(\CC^*)e(\PP^1)=(q-1)(q+1).$$
\end{proof}

The action of $\ZZ_2$ descends also to $\PGL(2,\CC)=\GL(2,\CC)/\CC^*$ and we have a diagram of fibrations

\begin{equation}\label{eqn:fibrationPGL}
\xymatrix{
 \CC^\ast \ar[d]\ar[r]^{=} & \CC^\ast \ar[d]\\
 \PGL(2,\CC) \ar[d]^{\widetilde{\pi}}\ar[r]& \PGL(2,\CC)/\ZZ_2\ar[d]^{\pi}\\
 \GL(2,\CC)/D \ar[r]^{\hspace{-20pt}2:1} &(\GL(2,\CC)/D)/\ZZ_2.}
 \end{equation}

\begin{prop} \label{prop:rem:extra}
For the action of $\ZZ_2$ on $\PGL(2,\CC)$,
$$
 e(\PGL(2,\CC))^+ = q^3-q, \qquad e(\PGL(2,\CC))^-=0 .
 $$
 Moreover, the monodromy of \eqref{eqn:fibrationPGL} is given by the morphism $\CC^*\to\CC^*$, $\mu\mapsto\mu^{-1}$.
\end{prop}

\begin{proof}
The action of $\ZZ_2$ on $\PGL(2,\CC)$ extends to an action of $\GL(2,\CC)$ (by left multiplication). Since $\GL(2,\CC)$ is connected, it follows that $\ZZ_2$ acts trivially on cohomology, giving the first statement. For the monodromy, note that $\GL(2,\CC)$ acts on the left hand fibration of \eqref{eqn:fibrationPGL} by right multiplication; a path in $\GL(2,\CC)$ connecting the identity to $\left(\begin{array}{cc} 0  & 1 \\ 1 & 0 \end{array}\right)$ descends to a loop representing the generator of $\ZZ_2$. Since the fibre of \eqref{eqn:fibrationPGL} over the identity is embedded by $\mu \mapsto \left(\begin{array}{cc} \mu  & 0 \\ 0 & 1 \end{array}\right)$, the embedding after the action of $\ZZ_2$ is given (up to homotopy) by
 $$
\mu \mapsto \left(\begin{array}{cc} 0  & 1 \\ \mu & 0 \end{array}\right)\left(\begin{array}{cc} 0  & 1 \\ 1 & 0 \end{array}\right)=\left(\begin{array}{cc} 1 & 0 \\ 0 & \mu \end{array}\right) \sim
\left(\begin{array}{cc} \mu^{-1}  & 0 \\ 0 & 1 \end{array}\right).
 $$
This gives the second assertion.
\end{proof}

\subsection{E-polynomials for a stratification of $\SL(2,\CC)$} We use the $\GL(2,\CC)$-action on $\SL(2,\CC)$
by conjugation to decompose the space $W=\SL(2,\CC)$ into strata whose E-polynomials we compute. The strata we use
are the different conjugacy classes: 
\begin{itemize}
\item $W_{0}:=$ conjugacy class of $\left(
             \begin{array}{cc}
               1 & 0 \\
               0 & 1
             \end{array}
           \right)$.
\item $W_{1}:=$ conjugacy class of  $\left(
             \begin{array}{cc}
               -1 & 0 \\
               0 & -1
             \end{array}
           \right)$.
\item $W_{2}:=$ conjugacy class of $\left(
              \begin{array}{cc}
                1 & 1 \\
                0 & 1
              \end{array}
            \right)$.
\item $W_{3}:=$ conjugacy class of $\left(
              \begin{array}{cc}
                -1 & 1 \\
                0 & -1
              \end{array}
            \right)$.
\item $W_{4}:=\{A\in\SL(2,\CC)\, | \, \Tr(A)\ne\pm2\}$. \end{itemize}
The stratum $W_4$ can also be described as the union of  the conjugacy classes of $\left(
              \begin{array}{cc}
                \lambda & 0 \\
                0 & \lambda^{-1}
              \end{array}
            \right)$
for $\lambda \in \CC\backslash\{0,\pm 1\}$.

\subsubsection{$W_0$ and $W_1$.}  These strata are single points, so
$$e(W_0)=e(W_1)=1.$$

\subsubsection{$W_2$ and $W_3$.} The stabiliser of $\left(
    \begin{array}{cc}
      \lambda & 1 \\
      0 & \lambda
    \end{array}
  \right)$, where $\lambda=\pm 1$, is the subgroup $U$ of Proposition \ref{propGL}. So
 $$
 e(W_{2})=e(W_{3})=q^{2}-1
 $$
by Proposition \ref{propGL}(iii).

\subsubsection{$W_4$.} \label{3.2.3}
The stabiliser of
$\left(
    \begin{array}{cc}
      \lambda & 0 \\
      0 & \lambda^{-1}
    \end{array}
  \right)$, where $\lambda\ne\pm 1$, is the diagonal group $D$. We therefore have a diagram
 $$
 \xymatrix{
 \GL(2,\CC)/D\ar[d]\ar[r]^= & \GL(2,\CC)/D\ar[d]\\
 \widetilde{W}_4\ar[d]\ar[r]&W_{4}\ar[d]\\
 \CC\setminus\{0,\pm 1\}\ar[r]^{2:1}& \CC\setminus\{\pm 2\}.
 }
 $$
where $\widetilde{W}_4:=\{(\lambda,P)|\lambda\in\CC\setminus\{0,\pm1\},P\in\GL(2,\CC)/D\}$.
The middle map is given by $(\lambda,P)\mapsto P\left(\begin{array}{cc}\lambda&0\\0&\lambda^{-1}\\ \end{array}\right)P^{-1}$,
the bottom map is given by $\lambda\mapsto\lambda+\lambda^{-1}$ and the map $W_4\to\CC\setminus\{\pm2\}$ is the trace map, $\Tr$.
Writing $F=\GL(2,\CC)/D$ and $B=\CC\setminus\{0,\pm1\}$, we have $e(B)^+=
q-2$, $e(B)^-=-1$ and, from Proposition \ref{propGL}, $e(F)^+=q^2$, $e(F)^-=q$. So, by Proposition \ref{prop:fibration},
 $$
 e(W_{4})=e(F)^{+} e(B)^{+} +e(F)^{-} e(B)^{-} = q^{2}(q-2)-q=q^{3}-2q^{2}-q.
 $$
One may note also that $e(\widetilde{W}_4)=q(q+1)(q-3)=q^3 - 2 q^2-3 q$.

\begin{rmk}
Note that
 $$
 e(W_{0})+e(W_{1})+e(W_{2})+e(W_{3})+e(W_{4})=q^3-q=e(\SL(2,\CC))\, ,
 $$
as expected.
\end{rmk}

\section{E-polynomials of  $\SL(2,\CC)^2$ using orbit spaces}\label{sl2}

Let us consider the map
\begin{eqnarray*}
f:\SL(2,\CC)^{2}&\longrightarrow & \SL(2,\CC)\\
(A,B)&\mapsto& [A,B]=ABA^{-1}B^{-1}
\end{eqnarray*}

Note that $\PGL(2,\CC)$ acts on both spaces and $f$ is equivariant. We stratify $X=\SL(2,\CC)^{2}$ as follows
\begin{itemize}
\item $X_{0}:=f^{-1}(W_{0})$
\item $X_{1}:=f^{-1}(W_{1})$
\item $X_{2}:=f^{-1}(W_{2})$
\item $X_{3}:=f^{-1}(W_{3})$
\item $X_{4}:=f^{-1}(W_{4})$.
\end{itemize}

Hence $X=\bigsqcup_{i=0}^{4} X_{i}$, so again we will test that
 $$
e(X)=e(\SL(2,\CC))^2= \sum_{i}e(X_{i})=q^{2}(q+1)^{2}(q-1)^{2}=q^{6}-2q^{4}+q^{2}.
 $$
 It will be convenient to study the varieties $f^{-1}(\xi)$ for fixed $\xi$ and we define accordingly
 \begin{itemize}
\item $\overline{X}_{0}:=f^{-1}(\Id)=X_0$,
\item $\overline{X}_{1}:=f^{-1}(-\Id)=X_1$,
\item $\overline{X}_{2}:=f^{-1}\left(\left(\begin{array}{cc}1&1\\0&1\end{array}\right)\right)$,
\item $\overline{X}_{3}:=f^{-1}\left(\left(\begin{array}{cc}-1&1\\0&-1\end{array}\right)\right)$,
\item $\overline{X}_{4,\lambda}:=f^{-1}\left(\left(\begin{array}{cc}\lambda&0\\0&\lambda^{-1}
  \end{array}\right)\right)$, where $\lambda\ne0,\pm1$.
\end{itemize}
It will also be convenient to define
\begin{itemize}
\item $\overline{X}_{4}:=\left\{(A,B,\lambda)\, | \, AB=\left(\begin{array}{cc}
      \lambda & 0 \\
      0 & \lambda^{-1}\\
    \end{array}
  \right) BA, \text{ for some } \lambda\ne0, \pm1,\, A,\,B \in \SL(2,\CC)\right\}$.
\end{itemize}

\subsection{E-polynomial for $X_{0}$}
We decompose
 $$
 X_{0}=(\SL(2,\CC)\times\{\pm \Id\})\cup (\{\pm\Id\}\times \SL(2,\CC))\cup X_0'\, ,
 $$
where
 $$
 X'_{0}=\{(A,B)\in X \ | \ AB=BA;\; A,B\neq \pm\Id\}.
 $$
Hence
 $$
 e(X_{0})=4\, e(\SL(2,\CC))-4 +e(X'_{0}).
 $$
We subtract $4$ because $(\SL(2,\CC)\times\{\pm \Id\})\cap (\{\pm\Id\}\times \SL(2,\CC))$
consists of $4$ points, $\{\pm \Id\}\times \{\pm\Id\}$, which we have counted twice in
$(\SL(2,\CC)\times\{\pm \Id\})\cup (\{\pm\Id\}\times \SL(2,\CC))$.

Furthermore we split $X'_{0}$ into the following subsets $X'_{0}=X^{'a}_{0}\sqcup X^{'b}_{0}$
\begin{itemize}
\item  $X^{'a}_{0}:=\{(A,B)\in X_0' | AB=BA;\; A,B\neq \pm\Id, \tr A\neq \pm2\}$. This
set consists of pairs of matrices $(A,B)\sim\left(\left(
    \begin{array}{cc}
      \lambda & 0 \\
      0 & \lambda^{-1} \\
    \end{array}
     \right),\left(
    \begin{array}{cc}
      \mu & 0 \\
      0 & \mu^{-1}\\
    \end{array}
  \right)\right)$,
  where $\lambda, \mu\in \CC\backslash\{0,\pm 1\} $.
\item $
X^{'b}_{0}:=\{(A,B)\in X_0' | AB=BA;\; A,B\neq \pm\Id, \tr A=\pm2\}$. This set
consists of pairs of matrices $(A,B)\sim\left(\left(
    \begin{array}{cc}
      \lambda & b \\
      0 & \lambda^{-1} \\
    \end{array}
  \right),\left(
    \begin{array}{cc}
      \mu & y \\
      0 & \mu^{-1} \\
    \end{array}
  \right)\right)$,
  where $\lambda=\pm 1, \mu= \pm 1$ and $y, b\in \CC^{\ast}$.
\end{itemize}

We use the following diagram to compute $e(X^{'a}_{0})$:
 $$
\xymatrix{
F=\CC\backslash\{0,\pm 1\}\ar[r]^{=}\ar[d] &\CC\backslash\{0,\pm1\}\ar[d]\\
\widetilde{X}^{'a}_{0}\ar[r]\ar[d]^{\widetilde{\pi}} & X^{'a}_{0}\ar[d]^{\pi}\\
B=\widetilde{W}_4\ar[r]^{2:1} & W_4
}
 $$
where $W_4$, $\widetilde{W}_4$ are as in subsection \ref{3.2.3}, the map $\pi$ is given by
$(A,B)\mapsto A$ and $\widetilde{X}^{'a}_0 = (\CC \setminus \{0,\pm 1\}) \times\dfrac{}{} \widetilde{W}_4$.
The bottom map is $(\lambda, P) \mapsto P\left(\begin{array}{cc} \lambda & 0\\ 0& \lambda^{-1}
 \end{array} \right) P^{-1}$ and the middle map is
 $$(\mu,\lambda, P) \mapsto
 \left(P\left(\begin{array}{cc} \lambda & 0\\ 0& \lambda^{-1}
 \end{array} \right) P^{-1}, P\left(\begin{array}{cc} \mu & 0\\ 0& \mu^{-1}
 \end{array} \right) P^{-1}\right).$$
 The maps are $2:1$ and reflect the action of $\ZZ_{2}$ consisting of
 interchanging the eigenvalues and eigenvectors.
 The E-polynomials of $W_4$ and $\widetilde{W}_4$ have already been computed in subsection \ref{3.2.3},
so we have $e(B)^+=q^3-2q^2-q$, $e(B)^-=-2q$, while $e(F)^+=q-2$, $e(F)^-=-1$ since the action on $F$ is given by $\lambda\mapsto\lambda^{-1}$. So
 $$
e(X^{'a}_{0})=(q^{3}-2q^{2}-q)(q-2)+(-2q)(-1)=q^{4}-4q^{3}+3q^{2}+4q.
$$

Now we compute $e(X^{'b}_{0})$. Note that $X^{'b}_0$ has $4$ components,
according to the signs of $\tr A$ and $\tr B$. We analyse one of these components, say the one given by $\lambda=\mu=1$,
the others being analogous. Note that the stabiliser of $A$ (which is the same as
the stabiliser of $B$) is the subgroup $U$ of Proposition \ref{propGL}. It is easy to check that each orbit contains exactly one element of the form
 $$
 (A,B)=\left(\left(
    \begin{array}{cc}
      1 & 1 \\
      0 & 1
    \end{array}
  \right), \left(
    \begin{array}{cc}
       1 & y \\
      0 &  1
    \end{array}
  \right)\right),
 $$
 so
 $$
 e(X^{'b}_{0})=4\, e( \GL(2,\CC)/U)e(\CC^\ast)=4(q^{2}-1)(q-1)= 4q^{3}-4q^{2}-4q+4 \, .
 $$
Hence
$$
e(X'_{0})=e(X^{'a}_{0})+e(X^{'b}_{0})= q^{4}-q^2+4
$$
and
 \begin{eqnarray*}
 e(X_{0})&=& 4 \, e(\SL(2,\CC))-4 +e(X'_{0}) \\
  &=& 4q^{3}-4q-4+q^{4}-q^2+4=q^{4}+4q^{3}-q^2-4q\, .
 \end{eqnarray*}

\subsection{E-polynomial for $X_{1}$}

$$
X_{1}=\{(A,B) \in X \ | \ AB=-BA \}.
$$

For $(A,B)\in X_{1}$,
$$\tr B=\tr (ABA^{-1}) =\tr(-B)=-\tr B$$ by the properties of the trace.
Hence $\tr B=0$ and analogously $\tr A=0$. The characteristic equations of these matrices are then $x^{2}+1=0$, therefore $A=\left(
    \begin{array}{cc}
      i & 0 \\
      0 & -i
    \end{array}
  \right)$ in some basis. The equation $AB=-BA$ implies that, in the same basis, $B=\left(
    \begin{array}{cc}
      0 & y \\
      -y^{-1} & 0
    \end{array}
  \right)$ where $y\neq 0$. Conjugating by a diagonal matrix, we can arrange that $y=1$.
Finally note that $\stab (A,B)=\CC^\ast$. So
  $$
 e(X_1)=e( \PGL(2,\CC))=q(q+1)(q-1)=q^{3}-q.
 $$

\subsection{E-polynomials for $X_{2}$, $\overline{X}_2$} \label{4.3}

\begin{eqnarray*}
 X_{2} &=& \{(A,B)\in X\ | \ AB\sim\xi BA\}, \\
 \overline{X}_{2} &=& \{(A,B)\in X \ | \ AB=\xi BA\},
\end{eqnarray*}
 where $\xi=\left(
    \begin{array}{cc}
      1 & 1 \\
      0 & 1
    \end{array}
  \right)$.

Firstly, $AB=\xi BA$ implies that $\tr B=\tr (ABA^{-1})=\tr \xi B$. Together with $\det B=1$, this gives us
 $$
 B=\left(
    \begin{array}{cc}
      x & y \\
      0 & x^{-1}
    \end{array}
  \right).
  $$
 Similarly
 $$
 A=\left(
    \begin{array}{cc}
      a & b \\
      0 & a^{-1}
    \end{array}
  \right).
$$

Now the condition $AB=\xi BA$ may be rewritten as
 $$
\left(
    \begin{array}{cc}
      ax & ay+bx^{-1} \\
      0 & a^{-1}x^{-1}
    \end{array}
  \right)=
  \left(
    \begin{array}{cc}
      ax & bx+a^{-1}y+a^{-1}x^{-1} \\
      0 & a^{-1}x^{-1}
    \end{array}
  \right)
$$
or
\begin{equation}\label{eqn:abxiba}
y(a-a^{-1})=b(x-x^{-1})+a^{-1}x^{-1},
\end{equation}
which can be written as
 $$
 y(x(a^2-1)) - b (a (x^2-1)) =1 \, .
 $$
So we have a family of lines over the open set $\CC^* \x \CC^* \setminus \{(\pm 1,\pm 1)\}$ in the plane parameterised
by $(a,x)$.
The stabiliser of $\xi$ in $\PGL(2,\CC)$ is $\CC$, and it acts on each line by translations:
 $$
 (y,b)\mapsto (y + \lambda x^{-1}(1-x^2), b + \lambda a^{-1}(1-a^2)) .
 $$
So
 $$
 e(\overline{X}_2)= q\big( (q-1)^2-4 \big) = q^3-2q^2-3 q
 $$
 and
  \begin{eqnarray*}
  e(X_{2}) &=& e(\GL(2,\CC)/U) e(\overline{X}_2) \\
  &=& (q^2-1) (q^3-2q^2-3q) \\
  &=& q^{5}-2q^{4}-4q^{3}+2q^{2}+3q.
 \end{eqnarray*}

\subsection{E-polynomials for $X_{3}$, $\overline{X}_3$}  \label{4.4}

\begin{eqnarray*}
 X_{3} &=& \{(A,B)\in X \ | \ AB\sim\xi BA\}, \\
  \overline{X}_{3} &=& \{(A,B)\in X \ | \ AB=\xi BA\},
\end{eqnarray*}
where $\xi = \left(
    \begin{array}{cc}
      -1 & 1\\
      0 & -1
    \end{array}
  \right)$.

For $(A,B)\in \overline{X}_3$, write
 $$
 A=\left(
    \begin{array}{cc}
      a & b \\
      c & d\\
    \end{array}
  \right), \quad
 B=\left(
    \begin{array}{cc}
      x & y \\
      z & w\\
    \end{array}
  \right).
  $$
Then $\tr B=\tr \xi B$ implies
\begin{equation}
z=2(x+w),\label{eq:x3-tr-1}
\end{equation}
while $\tr A =\tr \xi^{-1}A$ implies
 \begin{equation}
 c=-2(a+d).\label{eq:x3-tr-2}
 \end{equation}
and the relation $AB=\xi BA$ is translated into
\begin{eqnarray}
  && cy+2dw+bz=0 \label{eq:x3}
\end{eqnarray}

 We consider the following cases
 \begin{itemize}
 \item $\overline{X}{}'_{3}, X_3'$, when $x+w=0$,
 \item $\overline{X}{}''_{3}, X_3''$, when $x+w\neq 0$.
 \end{itemize}

\subsubsection{}
$\overline{X}{}'_{3}:=\{(A,B)\in \overline{X}_{3}|\, x+w=0 \}$
and $X'_{3}:=\{(A,B)\in {X}_{3}|\, x+w=0 \}$.
In this case \eqref{eq:x3-tr-1} implies $xw=1$; together with the equation $w=-x$, this gives $x=\pm i$, $w=\mp i$. The first vector of
our basis is an eigenvector for $B$. We choose the second vector to be an eigenvector of $B$ for the other eigenvalue,
so that
$$
B=\left(
    \begin{array}{cc}
      \pm i & 0 \\
      0 & \mp i
    \end{array}
  \right) .
 $$
Note that $\stab(\xi,B)=\CC^\ast$. {}From equation (\ref{eq:x3}), we get that $d=0$. So $A$ is of the form
$$
A=\left(
    \begin{array}{cc}
      a & \frac{1}{2a} \\
      - 2a & 0
    \end{array}
  \right)
$$
with $a\in \CC^{\ast}$. These $(A,B)$ form a slice for the action of the stabiliser of $\xi$
in $\PGL(2,\CC)$ on $\overline{X}{}_3'$, and this stabiliser is isomorphic to $\CC$. So
 $$
 \overline{X}{}_3'\cong\CC\times\CC^*\times\{\pm i\},\ \ X_3'\cong\PGL(2,\CC)\times\CC^*\times\{\pm i\}
 $$
 and
 $$
 e(\overline{X}{}_3')=2q(q-1)=2 q^2 -2 q,\ \
 e(X'_{3})=2q(q^2-1)(q-1)=2 q^4 - 2 q^3- 2 q^2+ 2 q.
 $$

\subsubsection{}
$\overline{X}{}''_{3}:=\{(A,B)\in \overline{X}_{3}|\, x+w\neq 0 \}$
and $X''_{3}:=\{(A,B)\in {X}_{3}|\, x+w\neq 0 \}$.
Changing the second vector in the basis corresponds to conjugating by the matrix
$\left( \begin{array}{cc}
1 & \alpha\\
0& 1
\end{array}\right)$. But
$$\left( \begin{array}{cc}
1 & \alpha\\
0& 1
\end{array}\right) \left( \begin{array}{cc}
x &y\\
z& w
\end{array}\right)
\left( \begin{array}{cc}
1 & -\alpha\\
0& 1
\end{array}\right)=\left( \begin{array}{cc}
x+\alpha z &*\\
z& w -\alpha z
\end{array}\right).$$ So we can arrange to have $w=0$, and hence by \eqref{eq:x3-tr-1},
 $$
 B=\left( \begin{array}{cc}
 x &-\frac{1}{2x} \\
 2x& 0
 \end{array}\right),
 $$
where $x\neq 0$. Again these $(A,B)$ form a slice, which we shall denote by $S$.
The equation \eqref{eq:x3} in this situation is $cy+bz=0$, i.e., $b=\frac{1}{4x^2} c$.
 The determinant is $ad-bc=1$; using \eqref{eq:x3-tr-2}, we get
  \begin{equation} \label{eqn:adx}
  ad - \frac{(a+d)^2}{x^2}=1\, .
  \end{equation}

 Let us compute the E-polynomial of this variety $S$. We have the contributions:
 \begin{itemize}
 \item $S^a:=\{(A,B) \in S \ | \  a+d=0\}$. Then $a=-d=\pm i$, $x\neq 0$. So $e(S^{a})=2(q-1)$.
 \item The case $a+d\neq 0$ has several contributions. If $a+d\neq 0$, take $t=(a+d)/x$. Then $t\in \CC^\ast$,
  and the equation (\ref{eqn:adx}) is $ad=1+t^2$. This is a family of conics over $\CC^\ast$.
  The first contribution is
   $$
  S^b:= \{(A,B) \in S \ | \  a+d\neq 0, t=\pm i \}\, .
  $$
  This corresponds to degenerate conics. For $t=\pm i$, the conic is $ad=0$. As $a+d\neq 0$, we have that
  $(a,d)\neq (0,0)$. So we get $e(S^b)=4(q-1)$.
  \item $S^c:=\{(A,B) \in S\ | \  a+d\neq 0, t\neq 0,\pm i \}$.
  For $t\in \CC\setminus\{0,\pm i\}$, we get a smooth conic $ad=1+t^2$. Consider the family
   of projective conics obtained by completing these affine conics. This gives a conic bundle, i.e., a fibration
  $$
   \PP^1 \to S^{c,1} \to \CC\setminus\{0,\pm i\}.
  $$
  Now:
  \begin{enumerate}
  \item By Remark \ref{rem:new},  we get $e(S^{c,1})=(q+1)(q-3)$.
  \item $S^{c,2}$ corresponds to the points at infinity of the conics in the conic bundle.
  These are given by $ad=0$, where $t\in \CC\setminus\{0,\pm i\}$. So $e(S^{c,2})=2(q-3)$.
  \item $S^{c,3}$ corresponds to the points in the conic bundle with $a+d=0$. So these points are given by
  $ad=1+t^2$, $a+d=0$, which gives the affine hyperbola $-a^2=1+t^2$, from which we must remove
  $t=0,\pm i$. The contribution is $e(S^{c,3})=q-5$.
 \end{enumerate}
  Thus $e(S^c)=e(S^{c,1})- e(S^{c,2})- e(S^{c,3})=(q+1)(q-3)-2(q-3)-(q-5)=q^2-5q+8$.
\end{itemize}
Finally we get
  $$
  e(S)= e(S^a)+ e(S^b)+ e(S^c)= 2(q-1)+4(q-1)+(q^2-5q+8)= q^2+q+2
 $$
and (using again the fact that the stabiliser of $\xi$ in $\PGL(2,\CC)$ is isomorphic to $\CC$)
 $$
 e(\overline{X}{}''_{3})=q(q^2+q+2),\ \ e(X_3'')=q(q^2-1)(q^2+q+2)= q^5 + q^4 + q^3- q^2  -2 q.
 $$
Thus
$$e(\overline{X}_3)=e(\overline{X}{}_3')+e(\overline{X}{}_3'')=q(q^2+3q)=q^3 +3 q^2$$
and
 $$
 e(X_{3})=e(X'_{3})+e(X''_{3})= q(q^2-1)(q^2+3q)=q^5  + 3 q^4 - q^3 -3 q^2.
 $$

\subsection{E-polynomials for $X_{4}$, $\overline{X}_{4,\lambda}$ and $\overline{X}_4$}\label{4.5}

Finally consider
$$X_{4}=\left\{(A,B)\in X \ | \ AB\sim\left(\begin{array}{cc}
      \lambda & 0 \\
      0 & \lambda^{-1}
    \end{array}
  \right) BA, \lambda\ne 0,\pm1\right\},$$
  $$\overline{X}_{4,\lambda}=\left\{(A,B)\in X\ | \ AB=\left(\begin{array}{cc}
      \lambda & 0 \\
      0 & \lambda^{-1}
    \end{array}
  \right) BA\right\}, \ \ \lambda\in \CC\setminus \{0,\pm1\} $$
  and
  $$\overline{X}_{4}=\left\{(A,B,\lambda)\ | \ AB=\left(\begin{array}{cc}
      \lambda & 0 \\
      0 & \lambda^{-1}
    \end{array}
  \right) BA,\lambda\ne0,\pm1, \,A,\,B \in \SL(2,\CC)\right\}.$$
  Write $\xi=\left(\begin{array}{cc}\lambda&0\\0&\lambda^{-1}\end{array}\right)$. For
$(A,B)\in \overline{X}_{4,\lambda}$, we have $\tr B=\tr \xi B$ and $\tr A=\tr \xi^{-1}A$, therefore
 $$
 A=\left(
    \begin{array}{cc}
      a & b \\
      c & \lambda^{-1}a
    \end{array}
  \right) \quad \text{and}  \quad B=\left(
    \begin{array}{cc}
      x & y \\
      z & \lambda x
    \end{array}
  \right).
$$

Now $AB=\xi BA$ means
$$
\left(
    \begin{array}{cc}
      ax+bz & ay+\lambda bx \\
      cx+\lambda^{-1}az & cy+ax
    \end{array}
  \right)=
  \left(
    \begin{array}{cc}
      \lambda(ax+cy) & \lambda bx+ay \\
      \lambda^{-1}az+cx & \lambda^{-1}(ax+bz)
    \end{array}
  \right).
$$
Combining this with the equations obtained from the determinants we have the following equations for $X_{4}$:
 \begin{eqnarray}
 \lambda^{-1}a^{2}-bc&=&1\label{eq:x4-1}\\
 \lambda x^{2}-yz&=&1\label{eq:x4-2}\\
 ax+bz&=&\lambda(ax+cy).\label{eq:x4-3}
 \end{eqnarray}
It cannot happen that $b=0$, $y=0$ simultaneously, since in this case $a\ne0$, $x\neq 0$, but then
$\lambda=1$ by \eqref{eq:x4-3}, which is a contradiction.

We shall split $X_{4}$ into strata $X_4^i$ (with corresponding strata for $\overline{X}_{4,\lambda}$ and $\overline{X}_4$) and construct a slice $S^i$ for each stratum $\overline{X}{}_{4}^i$. We will then have (noting that the stabiliser of $\xi$ in $\PGL(2,\CC)$ is isomorphic to $\CC^*$ and writing $S^i_\lambda$ for $S^i\cap\overline{X}{}^i_{4,\lambda}$)
 $$
 \overline{X}{}_{4,\lambda}^i\cong\CC^*\times S^i_\lambda,\ \ \overline{X}{}_4^i\cong\CC^*\times S^i,
 $$
with corresponding formulae for the E-polynomials. There is an action of $\ZZ_2$ on $\overline{X}_4$ given by $(A,B,\lambda)\mapsto(P_0AP_0^{-1},P_0BP_0^{-1},\lambda^{-1})$, where $P_0=\left(\begin{array}{cc}0&1\\1&0\end{array}\right)$. Provided $S^i$ is invariant under this action, we have
$$X_4^i\cong(\PGL(2,\CC)\times S^i)/\ZZ_2.$$
Since $e(\PGL(2,\CC))^-=0$ by Proposition \ref{prop:rem:extra}, it follows that
 \begin{equation}\label{eqn:x4}
 e(X_4^i)=e(\PGL(2,\CC))e(S^i/\ZZ_2)=(q^3-q)e(S_i)^+\, .
 \end{equation}

In describing the strata, we shall always assume that $AB=\xi BA$ and equations \eqref{eq:x4-1}, \eqref{eq:x4-2} and \eqref{eq:x4-3} hold.
\subsubsection{$X_4^0:=\{(A,B) \in X_4 \ | \  b=c=0\}$}
Equations (\ref{eq:x4-1}), (\ref{eq:x4-2}), (\ref{eq:x4-3}) specialize to $a^{2}=\lambda$,
$x=0$ and $yz=-1$. Rescaling the two basis vectors, we can assume that
$$
A=
  \left(
    \begin{array}{cc}
      a & 0\\
      0 & a^{-1}\\
    \end{array}
  \right),\ \ B= \left(
    \begin{array}{cc}
     0 & i\\
    i & 0\\
    \end{array}
  \right).
 $$
Note that $a\in \CC\setminus\{0,\pm 1,\pm i\}$. This gives our slice $S^0$, and $S^0_\lambda$ consists of the two points given by $a^2=\lambda$. So
  \begin{eqnarray*}
  e(\overline{X}{}_{4,\lambda}^0) &= &2(q-1)=2q-2,\\ e(\overline{X}{}^0_4) &=& (q-1)(q-5)=q^2 - 6 q +5
  \end{eqnarray*}
and, by \eqref{eqn:x4},
 $$
e(X_{4}^{0})=(q^3-q)e((\CC\setminus\{0,\pm 1,\pm i\})/\ZZ_2)=
(q^3-q)(q-3)= q^4-3q^3-q^2+3q \, .
 $$

\subsubsection{$X_4^1:=\{(A,B) \in X_4 \ | \  y=z=0\}$}
Analogously
\begin{eqnarray*}
 e(\overline{X}{}_{4,\lambda}^1) &=& 2(q-1),\\
  e(\overline{X}{}^1_4)&=&(q-1)(q-5),\\
  e(X_{4}^{1})&=& q^4-3q^3-q^2+3q \, .
 \end{eqnarray*}

\subsubsection{$X_4^2:=\{(A,B) \in X_4 \ | \  b=0, c\ne0 \} \cup \{(A,B) \in X_4 \ | \ c=0, b\ne0\}$}
In the case $b=0$, rescaling fixes $c=1$.
Equations (\ref{eq:x4-1}), (\ref{eq:x4-2}), (\ref{eq:x4-3}) in this situation are
\begin{eqnarray}
 a^{2} &=&\lambda\label{eq:x4'-1}\\
a^{2}x^{2}-yz &=&1\label{eq:x4'-2}\\
(1-\lambda)ax &=&\lambda y\label{eq:x4'-3}.
\end{eqnarray}
Equation (\ref{eq:x4'-1}) implies that $a\neq 0,\pm 1, \pm i$.
Equation (\ref{eq:x4'-3}) allows us to obtain $y$ uniquely. Note that
(\ref{eq:x4'-2}) and (\ref{eq:x4'-3}) imply that $x\ne0$ and $y\neq 0$.
Equation (\ref{eq:x4'-2}) uniquely determines $z$. So
we are reduced to considering $x\in \CC^\ast$ and
$a \in \CC\setminus\{0,\pm 1, \pm i\}$, which defines our slice $S^2$.
Moreover, for $S^2_\lambda$, we have again $a^2=\lambda$, so we get two copies of $\CC^*$.
Since we have a second component given by $c=0$, $b\neq 0$, we must double everything and we get
 \begin{eqnarray*}
 e(\overline{X}{}_{4,\lambda}^2)&=&4(q-1)^2=4 q^2 - 8 q +4,\\
 e(\overline{X}{}^2_4)&=& 2(q-1)^2(q-5)=2 q^3 - 14 q^2+ 22 q-10.
 \end{eqnarray*}

The $\ZZ_2$-action interchanges the two components, so we simply consider one of them and
get $X_4^2\cong\PGL(2,\CC)\times\CC^*\times (\CC\setminus\{0,\pm 1,\pm i\})$. Hence
 $$
 e(X^{2}_{4})=(q^3-q) (q-1)(q-5)=q^{5}-6q^{4}+4q^{3}+6q^{2}-5q.
 $$

\subsubsection{$X_4^3:=\{(A,B) \in X_4 \ | \  y=0,z\ne0 \}\cup \{(A,B) \in X_4 \ | \  z=0,y\ne0\}$}
This runs similarly to $X^{2}_{4}$, so
\begin{eqnarray*}
 e(\overline{X}{}_{4,\lambda}^3) &=& 4(q-1)^2=4 q^2 - 8 q +4,\\
 e(\overline{X}{}^3_4)&=& 2(q-1)^2(q-5)=2 q^3 - 14 q^2 + 22 q -10,\\
 e(X^{3}_{4})&=&q^{5}-6q^{4}+4q^{3}+6q^{2}-5q.
 \end{eqnarray*}

\subsubsection{$X_4^4:=\{(A,B) \in X_4 \ | \  b=z=0\} \cup \{(A,B) \in X_4 \ | \  c=y=0\}$}
Note that it is impossible that $b=0$ and $y=0$ (or $c=0$ and $z=0$) simultaneously. Therefore,
this gives exactly the intersection of the previous two strata
  $$
  X_4^4 = X_4^2 \cap X_4^3\, .
  $$
For the first component, we scale so that $c=1$.
The equations (\ref{eq:x4-1}), (\ref{eq:x4-2}), (\ref{eq:x4-3}) in this situation
give $a^2=\lambda$, $x^2=\lambda^{-1}$ and $(1-\lambda) ax =  \lambda y$.
Therefore $a\in \CC\backslash\{0,\pm 1,\pm i\}$, $x=\pm \frac{1}{a}$ and
$y=\pm \frac{1-a^{2}}{a^{2}}$. The slice $S^4$ is parameterised by $a$ and $x$,
while $S^4_\lambda$ consists of $4$ points. Doubling up to take into consideration the
two components, we get
\begin{eqnarray*}
 e(\overline{X}{}_{4,\lambda}^4) &=& 8(q-1)=8q-8,\\
 e(\overline{X}{}^4_4)&=&4(q-1)(q-5)= 4 q^2- 24 q +20 .
\end{eqnarray*}
Moreover
 $$
 e(X^{4}_{4})  =2(q^3-q)(q-5)= 2 q^4- 10 q^3- 2 q^2 +10 q\, .
 $$
Note that
 $$
 e(X^2_4 \cup X^3_4)= e(X^2_4)+ e(X^3_4)- e(X^4_4) \, .
 $$

\subsubsection{The case where $b,c,y,z$ are all $\neq 0$}
There will be several contributions in this case.
We begin by scaling so that $b=1$.
The equations (\ref{eq:x4-1}), (\ref{eq:x4-2}), (\ref{eq:x4-3}) are now
\begin{eqnarray}
c&=&\lambda^{-1}a^{2}-1\label{eq:x4''-1}\\
1&=&\lambda x^{2}-yz\label{eq:x4''-2}\\
z&=&(\lambda-1)ax +\lambda cy\label{eq:x4''-3}
\end{eqnarray}
and the condition $c\neq 0$ translates into $a^2\neq \lambda$.

Using (\ref{eq:x4''-1}) in (\ref{eq:x4''-3}), we get
$$
z=(\lambda-1)ax+ a^{2}y-\lambda y
$$
and then (\ref{eq:x4''-2}) gives
\begin{equation}\label{eq:x4''-conics}
 \lambda x^2 +a  (1-\lambda) x y + (\lambda -a^2) y^2=1.
 \end{equation}
This is a family of conics parameterised by $(x,y)$ over the plane $(a,\lambda)$, $a^2\neq \lambda$.
We continue stratifying the space $X_{4}$
with respect to the possibilities of these conics being degenerate or non-degenerate.

The discriminant is
\begin{equation}\label{eq:X4''-dis}
D:=4\lambda (a^{2}-\lambda)+(\lambda -1)^{2}a^{2}=((\lambda+1)a-2\lambda)((\lambda+1)a+2\lambda).
\end{equation}
We consider below
\begin{itemize}
\item $X^{5}_{4}$ corresponding to $D=0$ (degenerate conics) and
\item $X^{6}_{4}$, $X^{7}_{4}$ and $X^{8}_{4}$ for $D\neq 0$ (non-degenerate conics).
\end{itemize}

We shall also need to make explicit the $\ZZ_2$-action. This acts on coordinates as follows:
 \begin{eqnarray}\label{eqn:Z2-action-here}
  a &\mapsto & \lambda^{-1} a \nonumber\\ c &\mapsto & b=1 \ \mapsto \ c \nonumber\\
  b= 1 &\mapsto & c \ \mapsto \ 1 \nonumber\\
  x &\mapsto & \lambda x \\ y &\mapsto & z \ \mapsto \ \frac{z}{c}= \frac{(\lambda-1)ax}{c}
 +\lambda y =\lambda y + \frac{\lambda (\lambda-1) a}{a^2-\lambda} x \nonumber\\
 z &\mapsto & y \ \mapsto \ yc\, . \nonumber
 \end{eqnarray}
Recall that the $\ZZ_2$-action is given by conjugation with $\left(\begin{array}{cc} 0 & 1 \\ 1 & 0 \end{array}\right)$ together with the slice fixing condition which makes $b=1$.

\subsubsection{$X^{5}_{4}:=\{(A,B) \in X_4 \ | \  c,y,z \neq 0,\, b=1,\, D=0\}$} \label{4.5.7}
The discriminant of the conics in (\ref{eq:x4''-conics}) is zero over two lines given by
\begin{equation*}
a=\pm\frac{2\lambda}{\lambda+1}.
\end{equation*}
Note that $c=0 \iff a^2=\lambda
\iff (\lambda-1)^2=0$, which does not hold. So $c\neq 0$ automatically.

Suppose
 \begin{equation} \label{eqn:REF}
 a=\frac{2\lambda}{\lambda+1};
 \end{equation}
the other situation $a=-\frac{2\lambda}{\lambda+1}$ is similar.
The equation of the family of conics in (\ref{eq:x4''-conics}) is then
 $$
 \lambda x^2 +\frac{2\lambda}{\lambda+1} x y (1-\lambda) + \left(\lambda -\left(\frac{2\lambda}{\lambda+1}\right)^{2} \right)y^2=1,
 $$
that is
 \begin{equation*}
 \lambda x^{2}-\frac{2\lambda(\lambda-1)}{\lambda+1}xy+\lambda\frac{(\lambda-1)^{2}}{(\lambda+1)^{2}}y^{2}=1.
 \end{equation*}
Hence
 \begin{equation}\label{eq:x4''-conic0}
 \left(x- \frac{\lambda-1}{\lambda+1}y\right)^2=\frac{1}{\lambda}.
 \end{equation}
These are two parallel lines. Recall that we have to impose the condition $yz\neq 0$,
i.e., $\lambda x^2\neq 1$. This
means that $x\neq \pm \frac{1}{\sqrt{\lambda}}$.

Set $\mu:=x-\frac{\lambda-1}{\lambda+1} y$, so that $\mu^{2}=\lambda^{-1}$ by \eqref{eq:x4''-conic0} and
$\mu\in \CC\backslash\{ 0,\pm 1,\pm i\}$.
The condition $\lambda x^2\ne1$ becomes $x\ne\pm\mu$ and the map $(x,\mu) \mapsto (x,y,\lambda)$ is an isomorphism.
This gives our slice $S^5$ with $e(S^5)=(q-2)(q-5)$, while $e(S^5_\lambda)=2(q-2)$. So (recalling that we have two
components of this type),
\begin{eqnarray*}
e(\overline{X}{}_{4,\lambda}^5)&=&4(q-1)(q-2)=4 q^2 - 12 q +8,\\
e(\overline{X}{}^5_4)&=&2(q-1)(q-2)(q-5)= 2 q^3 - 16 q^2 + 34 q -20.
\end{eqnarray*}

The $\ZZ_2$-action leaves each of the two components invariant. We can therefore consider the action
of $\ZZ_2$ on the slice parameterised by $(x,\mu)$. The action maps $x\mapsto \mu^{-2}x$ and
 \begin{eqnarray*}
  \mu=x-\frac{\lambda-1}{\lambda+1}y & \mapsto & \lambda x -\frac{\lambda^{-1}-1}{\lambda^{-1}+1}
  \left( \lambda y + \frac{\lambda (\lambda-1)a}{a^2-\lambda} x\right) \\
   & & = \lambda x + \lambda \left( \frac{\lambda-1}{\lambda+1}y - 2 x\right) \\
    & & = -\lambda \left(x-\frac{\lambda-1}{\lambda+1}y \right) \\
    & & = -\lambda \mu^2=-\mu^{-1}\, ,
 \end{eqnarray*}
using (\ref{eqn:REF}).

We make a change of variable $X=x/\mu$ so that $X\neq \pm 1$ and $\ZZ_2$ acts as $X\mapsto -X$.
Using the fibration
 $$
 F=\CC\backslash\{\pm 1\} \longrightarrow E \longrightarrow \CC\backslash\{ 0,\pm 1,\pm i\}=B,
 $$
where $E=F\times B=\{(X,\mu)|X\ne\pm1,\mu\ne0,\pm1,\pm i\}$, we have
$e(F)^+=q-1$, $e(F)^-=-1$, $e(B)^+=q-3$, $e(B)^-=-2$.
So, taking into account that we have two components corresponding to the two possibilities for $a$, and
using \eqref{eqn:x4},
 \begin{eqnarray*}
 e(X^{5}_{4})&=&2 (q^3-q)\big((q-1)(q-3)+2 \big)=2(q^3-q)(q^2-4q+5)\\&=& 2q^5- 8 q^4 + 8 q^3 + 8 q^2 -10 q.
\end{eqnarray*}

\subsubsection{$X_4^6$.} Now consider the case $D\ne0$.
{}From equation (\ref{eq:X4''-dis}) this is equivalent to
 \begin{equation}\label{eq:X4-nocurve}
 a\neq \pm \frac{2\lambda}{\lambda+1}.
 \end{equation}
Recall the equation (\ref{eq:x4''-conics}) defining our family of conics over an open subset
of the plane parameterised by $(a,\lambda)$, $a^2\neq \lambda$. We complete each conic to a projective conic in
$\PP^{2}$ with coordinates $[u:x:y]$,
 $$
 \lambda x^2 +a (1-\lambda) x y + (\lambda -a^2) y^2=u^2.
 $$
We define $X_4^6, \overline{X}{}_4^6$ and $\overline{X}{}_{4,\lambda}^6$ as the corresponding $\PP^1$-fibrations.
In particular, we have a fibration
 \begin{equation}\label{eqn:S6}
 \PP^1 \to S^6 \to \{(a,\lambda) \ | \ \lambda\neq 0,\pm 1; \;
 a^2\neq \lambda; \; a\neq \pm 2\lambda/(\lambda+1) \}\, .
 \end{equation}
The plane $(a,\lambda)$ with $\lambda\neq 0,\pm 1$ and satisfying (\ref{eq:X4-nocurve}) has
E-polynomial $(q-2)(q-3)$. We remove
the subset $a^2=\lambda$, with $a\neq 0,\pm 1,\pm i$, hence getting
$$(q-2)(q-3)-(q-5)=q^2-6q+11.$$
The resulting variety is our slice $S^6$ and $e(S^6)=(q+1)(q^2-6q+11)$.
Moreover $S_\lambda^6$ is a conic bundle over $\CC\setminus\{4 \text{ points}\}$,
so $e(S^6_\lambda)=(q+1)(q-4)$. Thus
\begin{eqnarray*}
 e(\overline{X}{}_{4,\lambda}^6) &=& (q-1)(q+1)(q-4)=q^3 - 4 q^2 - q + 4,\\
 e(\overline{X}{}^6_4)&=&(q-1)(q+1)(q^2-6q+11)=q^4 - 6 q^3 + 10 q^2+ 6 q -11 .
 \end{eqnarray*}

The action of $\ZZ_2$ takes $(a,\lambda) \mapsto (\lambda^{-1}a,\lambda^{-1})$, and takes
each fibre of the conic bundle to another fibre, so defining a conic bundle over the plane
quotiented by $\ZZ_2$. To compute the latter,
set $\mu^{2}:=\lambda$ and $A:=\mu^{-1}a$, so the action by $\ZZ_{2}$ sends
$\mu\mapsto \mu^{-1}$ and $A$ remains invariant.
The equation in (\ref{eq:X4-nocurve}) is then
 $$
 A\neq \pm \frac{2}{\mu+\mu^{-1}}\, .
 $$
Also $a^2\neq \lambda$ means $A\neq \pm 1$.
We now change variables, writing $s=\mu+\mu^{-1}$, to take
account of the $\ZZ_2$-action. We still have to quotient by $(A,s)\mapsto (-A,-s)$, due
to the sign indeterminacy of $\mu$.
Let $B=As$ and $S=s^{2}$, so that the variables $(B,S)$
define the required quotient, with the conditions
 $$
 A\neq \pm \frac{2}{s},\pm 1 \Longleftrightarrow B\neq \pm 2, B^2 \neq S
 $$
and
 $$
 \lambda\neq 0,\pm 1 \iff \mu\neq 0,\pm 1, \pm i \iff s\neq 0,\pm 2 \iff S\neq 0, 4.
 $$

Thus $S^6/\ZZ_2$ consists of a family of conics over $(\CC\backslash\{\pm 2\})\times (\CC\setminus\{0,4\})$ minus
the set $S=B^2$, which is isomorphic to $\CC\setminus\{0,\pm 2\}$.
The E-polynomial for the base is $(q-2)^2-(q-3)=q^2-5q+7$. Finally the E-polynomial is
 $$
 e(X^{6}_{4})=(q^3-q)(q+1)(q^{2}-5q+7) = q^6-4q^5+q^4+11q^3-2q^2-7q \,  .
 $$

\subsubsection{$X^{7}_{4}$}
In $X^{6}_{4}$, we have added points when considering the projective conics which should not be included in the computation of $e(X_4)$. We now remove these points. Consider the sets $X_4^7$,
$\overline{X}{}_4^7$, $\overline{X}{}_{4,\lambda}^7$, $S^7$
corresponding to the points in each of the conics over the plane (\ref{eqn:S6}) satisfying the condition
$yz=0\Leftrightarrow\lambda x^2=1$. This means the slice $S^7$ is parameterised by $(a,\lambda,x,y)$ subject to
$\lambda\neq 0,\pm 1$, $a\neq \pm \frac{2\lambda}{1+\lambda}$, $a^2\neq \lambda$,
$\lambda x^2=1$ and
$a (1-\lambda) x y + (\lambda -a^2) y^2=0$. Therefore
$\lambda=x^{-2}$, $x\neq 0,\pm 1,\pm i$, and $y=0$ or $y=\frac{a(1-\lambda)}{a^2 -\lambda} x$.

We have two contributions:
\begin{itemize}
 \item $a=0$, $x\neq 0,\pm 1,\pm i$, $y=0$.
 \item $a\neq 0$. Then the two values of $y$ are different.
 Also
 $$
 a \neq \pm \frac{2x^{-1}}{x+x^{-1}}\, , a \neq \pm x^{-1}.
 $$
 \end{itemize}
For fixed $\lambda$, the case $a=0$ yields $2$ points, while $a\ne0$ gives $2$ values for $x$,
each of which gives two
values of $y$ and there are $5$ excluded values for $a$. So $e(S^7_\lambda)=2+4(q-5)=2(2q-9)$.

For variable $\lambda$, when $a=0$, we can take $x$ as a parameter, giving a contribution of $q-5$.
When $a\ne0$, we can take
$a$ and $x$ as parameters and we have $2$ values of $y$ for each $(a,x)$. So $e(S^7)=2(q-5)^2+q-5=(q-5)(2q-9)$.

This means that
\begin{eqnarray*}
  e(\overline{X}{}_{4,\lambda}^7)&=&2(q-1)(2q-9) = 4q^2 - 22q + 18, \\
 e(\overline{X}{}^7_4)&=&(q-1)(q-5)(2q-9)= 2 q^3 - 21 q^2 + 64 q-45.
\end{eqnarray*}

The action of $\ZZ_2$ on $S^7$ sends $x\mapsto \lambda x= x^{-1}$, $a\mapsto x^{2}a$ (so $A=ax$ is invariant)
and $y\mapsto  \frac{z}{c}= \lambda y + \frac{\lambda (\lambda-1)a}{a^2-\lambda} x$.
So
  $$
 y=0\mapsto \frac{\lambda (\lambda-1)a}{a^2-\lambda}x =
 \frac{\lambda^{-1} a (1-\lambda^{-1})}{\lambda^{-2}a^2-\lambda^{-1}} x^{-1}\, .
 $$
For $a=0$, we get simply $x\mapsto x^{-1}$, giving a contribution of $q-3$ to the quotient. For $a\ne0$,
the two lines corresponding to the two values $y\in \{0, \frac{a(1-\lambda)}{a^2 -\lambda} x\}$ are interchanged.
Thus we just have to compute the
E-polynomial of the variety with $y=0$. This is parameterised by $(x,A)$, subject to the conditions
$A\neq \pm \frac{2}{x+x^{-1}}$, $A\neq 0,\pm 1$. Then
$x\neq 0,\pm 1,\pm i$, $A\neq 0, \pm 1, \pm \frac2{x+x^{-1}}$.
As before, the contribution is $(q-5)^2$. So
\begin{eqnarray*}
 e(X^{7}_{4})&=&(q^3-q)\big( (q-5)^2+(q-3) \big)=(q^3 -q)(q^2-9q +22)\\
&=& q^5 -9 q^4 + 21 q^3+9q^2 -22 q\, .
\end{eqnarray*}

\subsubsection{$X^{8}_{4}$} \label{4.5.10}
It now just remains to take out from $X^{6}_{4}$ the points at infinity in each conic.
We consider the sets $X_4^8$, $\overline{X}{}_4^8$, $\overline{X}{}_{4,\lambda}^8$, $S^8$
corresponding to the points at infinity in each of the conics in the projective conic bundle.
For each value $(a,\lambda)$, these points are given by
 $$
 (\lambda-a^{2})y^2+a(1-\lambda)xy+\lambda x^2=0,
 $$
in projective coordinates $[x:y]$. Put $Y=y/x$, so the equation for $S^8$ is
 $$
 (\lambda-a^{2})Y^2+a(1-\lambda)Y+\lambda=0,
 $$
where the variables are $(a, \lambda, Y)$, $\lambda \neq 0,\pm 1$, $a\neq \pm \frac{2\lambda}{1+\lambda}$,
$a^2\neq \lambda$.

Note that we know that $a^{2}\neq\lambda$, so we may rewrite the eqation as
 $$
 Y^2+\frac{a(1-\lambda)}{\lambda-a^{2}}Y+\frac{\lambda}{\lambda-a^{2}}=0 \, .
 $$
Equivalently,
 \begin{equation}\label{eq:x4''-conic2}
 \left( Y+\frac12 \frac{a(1-\lambda)}{\lambda-a^{2}} \right)^{2}=\frac{D}{4(\lambda-a^2)^2} \, ,
 \end{equation}
where $D=(a(1+\lambda) + 2\lambda)(a(1+\lambda) - 2\lambda)$.
Write $e=\frac12 \frac{a(1-\lambda)}{\lambda-a^{2}}$. By the action of $\ZZ_{2}$,
$e\mapsto -e$. Now
 $$
 Y=\frac{y}{x}\mapsto \frac{z/c}{\lambda x}=  \frac{y}{x}+\frac{(\lambda-1)a}{a^2-\lambda} =Y+2e\, .
 $$
Introduce a new variable $\alpha =Y+e$. Then $\alpha$ remains invariant by the action of $\ZZ_{2}$.
The equation is now
  $$
  \alpha^{2}=\frac{D}{4(\lambda-a^2)^2}.
  $$

Consider the change of variables $(a,\alpha) \mapsto (A,\Omega)$ given by
\begin{eqnarray*}
A&:=&\frac{(1+\lambda)a}{\lambda}\\
\Omega&:=&\frac{2\alpha(\lambda-a^2)}{\lambda},
\end{eqnarray*}
possible because $a^2\neq \lambda$.
These are invariant under the action of $\ZZ_{2}$.
The equation (\ref{eq:x4''-conic2}) is now
 \begin{equation}\label{eqn:Omega+A}
 \Omega^{2}=A^{2}-4\, ,
 \end{equation}
and the conditions are $\lambda\neq 0,\pm 1$, $A\neq \pm 2$ and $A^2\ne\frac{(1+\lambda)^2}{\lambda}$.
This is our final equation for $S^8$. For fixed $\lambda$, we have a conic, from which we must remove the
two points at infinity and $6$ further points corresponding to the excluded values of $A$; so $e(S^8_\lambda)=q-7$, and
 $$
 e(\overline{X}{}^8_{4,\lambda})=(q-1)(q-7)=q^2 - 8 q +7.
 $$
To compute $e(S^8)$, we take the contribution of (\ref{eqn:Omega+A}) with $\lambda\neq 0,\pm1$, $A\neq \pm 2$,
which gives $(q-3)^2$, and subtract the contribution of $A^2=\frac{(1+\lambda)^2}{\lambda}$. Note
that then $\Omega^2=\frac{(1-\lambda)^2}{\lambda}$, so $A=\pm \frac{1+\lambda}{1-\lambda} \Omega$.
Writing $B=\frac{1+\lambda}{A}$, we have the equation $\lambda=B^2$, with $B\neq 0,\pm 1,\pm i$. Thus,
$e(S^8)=(q-3)^2-2(q-5)=q^2-8q+19$ and
 $$
  e(\overline{X}{}^8_4)=(q-1)(q^2-8q+19)=q^3  - 9 q^2 + 27 q-19.
  $$

Finally consider the action of $\ZZ_2$. As already observed, we find that this action leaves $A$, $\Omega$ invariant
and sends $\lambda\mapsto\lambda^{-1}$. We can therefore take $s=\lambda+\lambda^{-1}$ as a parameter on the
quotient. Ignoring the final condition $A^2\ne\frac{(1+\lambda)^2}{\lambda}$ for the moment, we have a conic with
$4$ points deleted and $s$ has $2$ excluded values, so we obtain a contribution of $(q-3)(q-2)$. {}From this we
have to delete the subvariety defined by $A^2=s+2$. This subvariety is a conic
$\Omega^{2}=A^{2}-4$ with the $2$ points at infinity and the $4$ points $(A,\Omega)=(0,\pm2i), (\pm2,0)$ removed;
this contributes $q-5$. Thus
 \begin{eqnarray*}
 e(X^{8}_{4})&=&(q^3-q) \big( (q-3)(q-2)-(q-5) \big)\\&=&(q^3-q)(q^2-6q+11) \\&=&q^5- 6 q^4+ 10 q^3+ 6 q^2-11 q.
 \end{eqnarray*}

\bigskip

We obtain finally
 \begin{eqnarray}
 \notag e(X_{4}) &=& e(X^{0}_{4})+e(X^{1}_{4})+e(X^{2}_{4})+e(X^{3}_{4})-e(X^{4}_{4})+
 e(X^{5}_{4})+e(X^{6}_{4})-e(X^{7}_{4})-e(X^{8}_{4}) \\ &=&(q^3-q)(q^3-2q^2-3q-2)
 = q^6 - 2 q^5 - 4 q^4 + 3 q^2 +2 q\, ,
 \label{eqn:e(X_4)}
 \end{eqnarray}
 and similarly
 \begin{equation}\label{eqn:x4lambda}
 e(\overline{X}_{4,\lambda})=(q-1)(q^2+4q+1)=q^3 + 3 q^2 - 3 q -1
 \end{equation}
 and
 \begin{equation}\label{eqn:e(x4-bar)}
 e(\overline{X}_4)= q^4-3q^3-6q^2+5q+3\,.
 \end{equation}

Note also that
 $$
 e(X)=e(X_0)+e(X_1)+e(X_2)+e(X_3)+e(X_4)=q^6- 2q^4 +q^2 =(q^3-q)^2,
 $$
which agrees with the known formula for $e(\SL(2,\CC)^2)$.

\section{$\SL(2,\CC)$-character varieties for $g=1$}\label{charvarg=1}

We now use the previous work to calculate the E-polynomials of the character varieties
 \begin{eqnarray*}
 \cM_\xi &=& \{ (A,B)\in SL(2,\CC) \, | \, [A,B] =\xi\}// \mbox{Stab}(\xi) \\
 &=& \{(A,B)\in SL(2,\CC) \, | \, [A,B] \in \cC\} // \PGL(2,\CC)\, , \\
 \end{eqnarray*}
where $\cC$ is the conjugacy class of $\xi\in \SL(2,\CC)$.

Corresponding to the five strata $X_i$ of section \ref{sl2}, we have five character varieties,
which we shall denote more logically by $\cM_{\Id}$, $\cM_{-\Id}$, $\cM_{J_+}$, $\cM_{J_-}$ and $\cM_\lambda$.

\subsection{Hodge polynomials}
\begin{itemize}
\item $\cM_{\Id}=X_0//\PGL(2,\CC)$.
This character variety contains reducibles (in fact, all elements are reducible), so we have to make a GIT quotient, identifying S-equivalence classes.
The non-diagonalisable orbits have limits in the diagonalisable ones, so every S-equivalence class contains an element of the form
 $$
 (A,B)=\left(\left(\begin{array}{cc}\lambda&0\\0&\lambda^{-1}\end{array}\right),
 \left(\begin{array}{cc}\mu&0\\0&\mu^{-1}\end{array}\right)\right),
 $$
unique up to
$(\lambda, \mu)\mapsto (\lambda^{-1},\mu^{-1})$. We have the fibration $F\to E=F\x B \to B$, where
$\lambda$ parametrizes $F=\CC\setminus\{0\}$ and $\mu$ parametrizes $B=\CC\setminus\{0\}$. So
$e(F)^+=q$, $e(F)^-=-1$, $e(B)^+=q$, $e(B)^-=-1$, and Proposition \ref{prop:fibration} gives
  $$
e(\cM_{\Id})= q^2+1 \, .
  $$

\item {$\cM_{-\Id}=X_1//\PGL(2,\CC)$}.
We have already seen that the character variety is just one point, so
 $$
 e(\cM_{-\Id})=1\, .
 $$

\item {$\cM_{J_+}=X_2//\PGL(2,\CC)$}. Here we put $J_+=\left(\begin{array}{cc} 1 &1 \\ 0 & 1 \end{array}\right)$.
The computations in subsection \ref{4.3} show that
 \begin{eqnarray*}
 \cM_{J_+} &\cong& \CC^* \x \CC^* \setminus \{(\pm 1,\pm 1)\}, \\
 e(\cM_{J_+}) &=& (q-1)^2-4= q^2-2q-3 \, .
 \end{eqnarray*}

\begin{rmk}\label{rmk:cMJ+-red}
Given the shape of the matrices $A,B$ in subsection \ref{4.3}, we see that all representations in
$X_2$ are reducible. Therefore there is a map $\cM_{J_+}\to \cM_{\Id}$ which consists of quotienting
by the involution
$(a,x)\mapsto (a^{-1},x^{-1})$.
\end{rmk}

\item {$\cM_{J_-}=X_3//\PGL(2,\CC)$}. Here we put $J_-=\left(\begin{array}{cc} -1 &1 \\ 0 & -1 \end{array}\right)$.
We can stratify as in subsection \ref{4.4} to show that
 $$
 e(\cM_{J_-})=2(q-1) + (q^2+q+2)= q^2 +3q \, .
 $$

\begin{rmk}\label{rmk:cMJ--red}
 Note that there are no reducibles in $X_3$. If $A,B$ form a reducible representation, then
 there is a common eigenvector $v\in\CC^2$. Then $[A,B](v)=v$, so the only eigenvalue of  $[A,B]$ is $1$.
 Therefore either $(A,B)\in X_0$, or $(A,B)\in X_2$.
\end{rmk}

\item {$\cM_\lambda=\cM_\xi = \overline{X}{}_{4,\lambda} // \CC^*$}.
Here we have set $\xi=\left(\begin{array}{cc} \lambda & 0 \\ 0 & \lambda^{-1} \end{array}\right)$,
$\lambda \neq 0,\pm 1$.
The slices $S^i_\lambda$ constructed in subsection \ref{4.5} give a stratification of $\cM_\lambda$. It follows at once that
 $$
 e(\cM_\lambda)=e(\overline{X}_{4,\lambda})/(q-1)=q^2+4q+1.
 $$
\end{itemize}

\begin{rmk} \label{rem:suggestion-Hausel}
 Note the equality $e(\cM_{J_-})+(q+1) e(\cM_{-\Id}) = e(\cM_{\lambda})$ holds.
 This has been suggested to us by T.\ Hausel, and it is predicted to hold for
 arbitrary genus.
\end{rmk}

\subsection{Poincar\'e polynomials and Hodge numbers}
The information obtained above for the E-polynomials is not sufficient to compute
Poincar\'e polynomials or the $h^{k,p,q}_c$ of the character varieties, since there can be cancellations in the
computation of the E-polynomial. However, using the explicit description and, in one case, the known
Poincar\'e polynomials, we can compute the Hodge numbers (except in the case of $\cM_{J_-}$).

\begin{itemize}
 \item For $\cM_{\Id}$, we have an explicit description.  We know that $h^{k,p,q}_c(\CC^*)=1$ for $(k,p,q)=(2,1,1)$ and
$(k,p,q)=(1,0,0)$ and is otherwise non-zero. So, for $\CC^*\times\CC^*$,  we have non-zero Hodge numbers as follows:
 $h^{4,2,2}_c=1, h^{3,1,1}_c=2, h^{2,0,0}_c=1$.
It is easily checked that, for the action of $\ZZ_2$ on the cohomology of $\CC^*\times \CC^*$, we have
 $(H^*_c)^+=H^2_c\oplus H^4_c\ \ \mbox{and }  (H^*_c)^-=H^3_c$.
It follows that the non-zero Hodge numbers of $\cM_{\Id}$ are
 $$
 h^{4,2,2}_c(\cM_{\Id})=h^{2,0,0}_c(\cM_{\Id})=1.
 $$
 It follows also that the Poincar\'e polynomial $P_t^c$ for compact cohomology is given by $P_t^c(\cM_{Id})=t^4+t^2$.
A similar computation using ordinary cohomology gives $P_t(\cM_{\Id})=1+t^2$ (or we can use
 Poincar\'e duality because $\cM_{\Id}$ is smooth). We can summarise these results in the statement
 $$
 H_c(\cM_{\Id})(q,t)=q^2t^4+t^2.
 $$

\item For $\cM_{-\Id}$, there is nothing to prove.

\item  For $\cM_{J_+}$, we again have an explicit description, from which it follows that the non-zero Hodge numbers are
 $$
 h^{4,2,2}_c(\cM_{J_+})=1,\ \ h^{3,1,1}_c(\cM_{J_+})=2,\ \
 h^{2,0,0}_c(\cM_{J_+})=1,\ \ h^{1,0,0}_c(\cM_{J_+})=4\, .
 $$
The Poincar\'e polynomials are
 $$
 P_t^c(\cM_{J_+})=t^4+2t^3+t^2+4t,\ \ P_t(\cM_{J_+})=4t^3+t^2+2t+1 \, .
 $$
Note that this space is smooth, so Poincar\'e duality
applies to relate the two polynomials. Any one of our formulae implies that
the Euler characteristic is $-4$. We can again summarise our results as
 $$
 H_c(\cM_{J_+})(q,t)=q^2t^4+2qt^3+t^2+4t\, .
 $$

\item
To deal with the case $\cM_{J_-}$, let us use the computations in subsection \ref{4.4} to describe explicitly
$\cM_{J_-}$. First, we have maps $\pi_1:\cM_{J_-} \to \CC$, $\pi_1(A,B)=\Tr A$, and $\pi_2:\cM_{J_-} \to \CC$, $\pi_2(A,B)=\Tr B$. This produces a map
$\pi=\pi_1\x \pi_2:\cM_{J_-} \to \CC\x\CC\setminus \{(0,0)\}$.
Note that $\overline{X}{}_3''$ is the part corresponding to
$\pi_2^{-1}(\CC\setminus \{0\})$. So this latter set is described by the equation
(\ref{eqn:adx}), which is a smooth surface.
By symmetry, $\pi_1^{-1}(\CC\setminus \{0\})$ is also smooth, so $\cM_{J_-}$ is smooth, being covered
by $\pi_1^{-1}(\CC\setminus \{0\})$ and $\pi_2^{-1}(\CC\setminus \{0\})$. It follows that $h^{k,p,q}_c=0$ whenever $p+q>k$.

The map $\pi_2:\cM_{J_-} \to \CC$ is a fibration over
$\CC$ with some singular fibres. The generic fibre is a hyperbola, isomorphic to
$\CC^*$; this conic (\ref{eqn:adx}) degenerates into two (parallel) lines
when $x=\pm 2$; the fibre over $0$ consists of two copies of $\CC^*$, since it corresponds to
$\overline{X}{}_3'$.
Thus the manifold $\cM_{J_-}$ is constructed as follows:
take the surface $\CC\x \PP^1$, blow it up at three points in the fibres over
$0,2,-2$, obtaining thus three nodal fibres, with three nodes, which we call $p_{-2},p_0,p_2$.
Now we remove a bi-section of this non-minimal ruled surface, which passes through $p_{-2},p_2$,
but not through $p_0$. This
bisection cannot be reducible (otherwise, we would not get the correct E-polynomial),
so it has to be a double cover
of $\CC$ ramified exactly at $p_{-2},p_2$. Hence it is isomorphic to $\CC^*$.
Finally, we also remove the point $p_0$. The result is $\cM_{J_-}$.

{}From this it follows that $b_0=1$, $b_1=1$, $b_2=5$ and $b_3=1$. The Euler characteristic is $4$.
The Poincar\'e polynomials are
 $$
 P_t^c(\cM_{J_-})=t^4+ t^3+ 5 t^2+t,\ \ P_t(\cM_{J_-})=t^3+5t^2+t+1 \, .
 $$

 Although the combination of $e(\cM_{J_-})$ and $P_t^c(\cM_{J_-})$ gives considerable information about the values
 of the Hodge numbers, this is not sufficient to determine them completely. In fact $h^{4,2,2}_c(\cM_{J_-})=1$
 and $h^{1,0,0}_c(\cM_{J_-})=1$, and the other non-zero Hodge numbers are either
 $$
 h^{3,1,1}_c(\cM_{J_-})=1,\ \ h^{2,1,1}_c(\cM_{J_-})=4,\ \ h^{2,0,0}_c(\cM_{J_-})=1
 $$
or
 $$
 h^{3,0,0}_c(\cM_{J_-})=1,\ \ h^{2,1,1}_c(\cM_{J_-})=3,\ \ h^{2,0,0}_c(\cM_{J_-})=2.$$

\item
Finally, let us consider
$\cM_\lambda$. Since this is smooth, $h^{k,p,q}_c=0$ whenever $p+q>k$. Moreover $h^4_c=1$ since
$\cM_\lambda$ is connected and $h^0_c=0$ since $\cM_\lambda$ is not compact.
The moduli space $\cM_\lambda$ is homeomorphic to the moduli
space of parabolic $\SL(2,\CC)$-Higgs bundles $\cH_\lambda$ on an elliptic curve $C$.
This is the space $\cN_\alpha^0$ considered in \cite{by} (with $g=n=1$),
for which the Poincar\'e polynomials are
 $$
 P_t(\cN^0_\alpha)=\frac{(1+t^3)^2+t^3(1+t)(t-2)}{(1-t^2)^2}+3t^2=5t^2+1, \ \ P_t^c(\cN^0_\alpha)=t^4+5t^2.
 $$
Comparing this with the formula for $e(\cM_\lambda)$ and noting that all Hodge numbers of this space must be of type $(p,p)$ (see Proposition \ref{prop:pure-type}), we find that the non-zero Hodge numbers are
 $$
 h^{4,2,2}_c(\cM_\lambda)=1,\ \ h^{2,1,1}_c(\cM_\lambda)=4,\ \ h^{2,0,0}_c(\cM_\lambda)=1.
 $$
 Hence
 $$
 H_c(\cM_\lambda)=q^2t^4+qt^2+t^2.
 $$
\end{itemize}

\begin{rmk}
The space $\cH_{\Id}$ is the moduli space of $\SL(2,\CC)$-Higgs bundles on $C$, which is known to be
isomorphic to $C\times\CC/\ZZ_2$, where
the action of $\ZZ_2$ is given by $(\lambda,a)\mapsto(\lambda^{-1},-a)$.
The non-zero Hodge numbers of $C\times\CC$ are
$h^{4,2,2}_c=h^{3,2,1}_c=h^{3,1,2}_c=h^{2,1,1}_c=1$. Now $\ZZ_2$
acts on $H^4$ and $H^2$ by $+1$ and on $H^3$ by $-1$. So
the non-zero Hodge numbers of the moduli space are  $h^{4,2,2}_c=h^{2,1,1}_c=1$
and the E-polynomial is
  $$
  e(\cH_{\Id})=q^2+q
  $$
(different from that of $\cM_{\Id}$). This result could also be obtained by noting that the moduli space maps to
$C/\ZZ_2\cong\PP^1$ and all the fibres are isomorphic to $\CC$ (although it is not a locally trivial fibration).
\end{rmk}

\subsection{Failure of Mirror Symmetry for $\cM_{J_\pm}$}
Mirror Symmetry predicts an equality of the E-polynomials of the $G$-character varieties
and $G^L$-character varieties, where $G^L$ is the Langlands dual of $G$, for \textit{semisimple
holonomy}.

Mirror Symmetry was proved for the moduli spaces of $G$-Higgs bundles of ranks $2$ and $3$
by Hausel and Thaddeus \cite{hausel-thaddeus:2003}, showing that the mirror partners have equal Hodge numbers.
It was conjectured and proved by Hausel and Rodriguez-Villegas \cite{hausel-rvillegas:2007} that the same
happens for the twisted $G$-character variety.

For $G=\SL(2,\CC)$, we have $G^L=\PGL(2,\CC)$.
We want to show now that the analogous statement to
the conjecture of \cite{hausel-rvillegas:2007} does not hold
for $G$-character varieties with non-semisimple holonomy $\xi=J_\pm$.
We check it for curves of genus $g=1$.
For this, let us compute the E-polynomial of
 $$
 \cM_\xi^L:= \cM_\xi (\PGL(2,\CC))\, ,
 $$
for $\xi=\Id, -\Id, J_+, J_-$ and $\left( \begin{array}{cc} \lambda & 0 \\ 0 & \lambda^{-1}\end{array}\right)$,
with $g=1$.

Note that $\cM_\xi^L$ appears as the quotient of $\cM_\xi$ with the action of $Z=\ZZ_2\x \ZZ_2$
given by $A\mapsto -A$, $B\mapsto -B$, for $(A,B)\in \cM_\xi$.

\begin{itemize}
\item We have $\cM_{\Id}=(\CC^* \x \CC^*)/\ZZ_2$, where $\ZZ_2$ acts as
$(\lambda,\mu) \mapsto (\lambda^{-1},\mu^{-1})$. We quotient now by $\lambda \mapsto -\lambda$ and
$\mu \mapsto -\mu$. As $\CC^*/ \pm 1 \cong \CC^*$, we get that
$\cM_{\Id}^L=\cM_{\Id}/Z  \cong (\CC^* \x \CC^*)/\ZZ_2$. Therefore
$e(\cM_{\Id}^L)=q^2+1$, as expected.

\item $\cM_{-\Id}$ is one point, so $\cM_{-\Id}^L$ is one point also, and $e(\cM_{-\Id}^L)=1$.

\item The discussion in subsection \ref{4.3} produced a description
$\cM_{J_+} = \CC^* \x \CC^* \setminus \{ (\pm 1, \pm 1)\}$.
The action of $Z$ on $(a,x)\in \cM_{J_+}$ is given by
 \begin{equation}\label{eqn:actioon}
 a\mapsto -a\, , \quad x \mapsto -x\, .
 \end{equation}
So $\cM_{J_+}^L = (\CC^* \x \CC^* \setminus \{ (\pm 1, \pm 1)\})/Z \cong \CC^* \x \CC^* \setminus \{ (1, 1)\}$.
Therefore $e(\cM_{J_+}^L)= q^2-2q \neq e(\cM_{J_+})$.

\item We have to work a little bit for $\cM_{J_-}^L$. We decompose $\cM_{J_-}$ in several pieces,
following the discussion in subsection \ref{4.4}. The first piece, corresponding to $\overline{X}{}'_3$,
yields a space $\CC^*\x \{\pm i\}$, and $Z$ acts on $(a,x)\in \CC^*\x \{\pm i\}$ by (\ref{eqn:actioon}).
This produces $(\CC^*\x \{\pm i\})/Z \cong \CC^*$ with E-polynomial $q-1$.

The second piece, corresponding to $S^a$, is described by $a=-d= \pm i$, $x\in \CC^*$, with the action
(\ref{eqn:actioon}). This contributes $q-1$ to the E-polynomial.

Finally, the remainder is given by equation (\ref{eqn:adx}), where
$(a,d,x)$ satisfies $a+d\neq 0$, $x\neq 0$, and subject
to the quotient by (\ref{eqn:actioon}). We do the change of variables $t=(a+d)/x$, where $t\sim -t$. Now
we introduce the variable $\alpha=t^2$, and have the equation $ad-\alpha=1$, modulo $(a,d) \sim (-a,-d)$.
Here $\alpha=ad-1$ is uniquely determined, so we can forget about it, only recalling that $ad-1 \neq 0$. The
resulting space is $\{(a,d) \, | \, a+d \neq 0, ad\neq 1\} / \ZZ_2$, whose E-polynomial is easily computed to
be $q^2-2q+2$.
Hence
 $$
 e(\cM_{J_-}^L)=q^2-2q+2+ 2(q-1)= q^2 \neq e(\cM_{J_-})\, .
 $$

\item The E-polynomial of $\cM_\lambda^L$ has not been computed yet.
The conjecture in \cite{hausel-rvillegas:2007} predicts that $e(\cM_\lambda^L)=q^2+4q+1$. We shall
undertake a full study of the case of $\PGL(2,\CC)$-character varieties in future work, including
this case and many others.
\end{itemize}

\section{The Hodge monodromy representation of $X_4$} \label{sec:representation}

Recall that
 $$
  \overline X_4=\left\{ (A,B,\lambda) \ | \ [A,B]= \left(\begin{array}{cc} \lambda & 0 \\ 0 & \lambda^{-1}
 \end{array}\right), \ \lambda\neq 0,\pm 1\right\}.
 $$
In this section we want to study in detail the fibration
$$
 \pi: \overline X_4 \too \CC\setminus \{0,\pm 1\}\, ,
$$
where $(A,B,\lambda)\mapsto \lambda$. The fibres of $\pi$ are the spaces $\overline{X}_{4,\lambda}$,
all of them diffeomorphic and of balanced type. In particular, all fibres have
the same E-polynomial and by Proposition \ref{prop:pure-type} $\overline X_4$ is also of balanced type.

Consider loops $\gamma_0$, $\gamma_1$, $\gamma_{-1}$
surrounding the points $0,1,-1$ respectively. These are free generators of $\Gamma=H_1(\CC\setminus \{0,\pm 1\})$.
We are interested in computing the Hodge monodromy representation of $\pi$,
  $$
  R(\overline X_4) \in R(\Gamma)[q]\, .
  $$

%
%
%

\begin{thm} \label{thm:R(overline)}
  The actions of $\gamma_{\pm 1}$ are trivial, and the
  action of $\gamma_0$ is an involution. Consider the associated
  quotient $\Gamma\twoheadrightarrow \ZZ_2$. Then
   $$
    R(\overline{X}_4) = (q^3-1) \,  T + (3q^2-3q) \, N \in R(\ZZ_2)[q]\, ,
   $$
   where $T$ is the trivial representation, and $N$ is the non-trivial one.
\end{thm}

\begin{proof}
According to subsection \ref{4.5}, we have a stratification $\overline X_4=\bigsqcup\overline{X}{}^i_{4}$,
which is compatible with the fibration. Therefore
  $$
  R(\overline{X}_4) =\sum R(\overline{X}{}^i_4) \, .
  $$
Note that we have isomorphisms $\overline{X}{}_{4,\lambda}^i \cong S^i_\lambda \times \CC^\ast$,
so $e(\overline{X}{}_{4,\lambda}^i)= e(S^i_\lambda)(q-1)$. Therefore
$R(\overline{X}{}_4^i)=(q-1) R(S^i)$. Let us compute each $R(S^i)$ individually. We
follow the notations of subsection \ref{4.5}.

 \begin{itemize}
  \item $S_\lambda^0$. The equations are $a^2=\lambda$, where $\lambda$ is fixed. So $S_\lambda^0$
  consists of two points. The actions
  of $\gamma_1$ and $\gamma_{-1}$ are trivial. The action of $\gamma_0$ is an involution on the fibre, as it
  swaps the two points.
  Therefore $R(S^0)=T+N$.

  \item $S_\lambda^{1}$. Exactly as in the previous case, $R(S^1)=T+N$.

  \item $S_{\lambda}^2$. Recall that there are two different (isomorphic) strata, given by
  either $b\neq 0, c=0$ or $b=0, c\neq 0$. The
  equations of the first stratum are $a^2=\lambda$, $y= (1-\lambda) ax/\lambda$, and $z= (\lambda x^2-1)/y$.
  Hence $a$ admits two values, and $x\neq 0$. The elements $\gamma_{\pm1}$ act trivially, and $\gamma_0$
  swaps the two values of $a$. Doubling the contribution, we get $R(S^2)=2(q-1)T+2(q-1)N$.

  \item $S_{\lambda}^3$. Exactly as in the previous case,
  $R(S^3)=2(q-1)T+2(q-1)N$.

  \item $S_{\lambda}^4$.
  Again we have two different (isomorphic) strata. The equations determining one of them are
  $a^2=\lambda$, $x=\pm \frac1a$, $y=\pm \frac{1-a^2}{a^2}$, so it consists of $4$ points.
  The elements $\gamma_{\pm1}$ act trivially, and $\gamma_0$
  swaps the two values of $a$. Doubling the contribution, we get $R(S^4)=4T+4N$.

  \item $S_{\lambda}^5$.
  Now $a=\pm \frac{2\lambda}{\lambda+1}$ has two values, giving two (isomorphic) strata. For each value,
  we have the equation:
  $$
  \left(x - \frac{\lambda-1}{\lambda+1} y \right)^2 =\frac1\lambda\, ,
  $$
  with the conditions $x\neq \pm \frac{1}{\sqrt{\lambda}}$. Hence we get two parallel lines with two points
  removed in each of them. The actions
 of $\gamma_1$ and $\gamma_{-1}$ are trivial, and the action of $\gamma_0$ interchanges the two lines.
 Doubling the contribution, we get
 $$R(S^5)=2(q-2)T+2(q-2)N.$$

 \item $S_{\lambda}^6$. For a fixed $\lambda$, $S_\lambda^6$ is a surface which consists of a family of projective conics
  $$
  \lambda^2 x + a(1-\lambda)xy + (\lambda -a^2)y^2 =u^2\, ,
  $$
  parameterised by $a\in B=\CC\setminus\{ \pm \sqrt{\lambda}, \pm \frac{2\lambda}{\lambda+1}\}$.
  The actions
  of $\gamma_1$ and $\gamma_{-1}$ in $S_\lambda^6$ are again trivial. The action of $\gamma_0$ interchanges two of
  the points in the base $B$ of the fibration. So $e(B)^{inv}=q-3$, and $e(B)-
  e(B)^{inv} = -1$. As $S_\lambda^6 \to B$ is a fibration whose fibres are conics, we have
  $$R(S^6)=(q+1) \big( (q-3)T - N \big)=(q^2-2q-3)T - (q+1)N.$$

\item $S_{\lambda}^7$. {}From $S_\lambda^6$ we have to remove some points. Firstly, from each fibre (conic)
we consider the points
  $x=\pm \frac{1}{\sqrt{\lambda}}$, $y=0$; and
  $x=\pm \frac{1}{\sqrt{\lambda}}$, $y= - a(1-\lambda) x/(\lambda -a^2)$, where
  $a\in \CC\setminus \{ \pm \sqrt\lambda, \pm \frac{2\lambda}{\lambda+1}\}$.
  Geometrically, this produces two disjoint components, each of which consists of
  $2$ punctured lines (punctured at  $4$ points), intersecting at one point (corresponding to the value $a=0$).
 This object has E-polynomial $2q-9$, and there are two copies of it.

 The action of $\gamma_{\pm 1}$ is trivial, and the action of $\gamma_0$ interchanges
 the components. So $R(S^7)=(2q-9) T+ (2q-9)N$.

\item $S_{\lambda}^8$. We also remove the points of the hyperbola $\Omega^2=A^2-4$
  minus $(A,\Omega)=(\pm 2,0), (\pm \frac{1+\lambda}{\sqrt{\lambda}},\pm \frac{1-\lambda}{\sqrt{\lambda}})$
 The action of $\gamma_{\pm 1}$ is again trivial, and the action of $\gamma_0$ leaves the first two points
 fixed and interchanges the other four in two pairs.
  So  $R(S^8)=(q-5) T-2N$.
  \end{itemize}

Altogether, we get
 \begin{eqnarray*}
 R(\overline{X}_{4}) &=& R(\overline{X}{}^{0}_{4})+R(\overline{X}{}^{1}_{4})+
 R(\overline{X}{}^{2}_{4})+\\&&+R(\overline{X}{}^{3}_{4})-R(\overline{X}{}^{4}_{4})+
 R(\overline{X}{}^{5}_{4})+R(\overline{X}{}^{6}_{4})-R(\overline{X}{}^{7}_{4})-
 R(\overline{X}{}^{8}_{4}) \\ &=&  (q-1)((q^2 +q+1)T+ 3q N) \\ &=& (q^3-1)T + (3q^2-3q)N  \, .
 \end{eqnarray*}
\end{proof}

A lot of information can be extracted from Theorem \ref{thm:R(overline)}. For instance,
the E-polynomial of the fibre is obtained by
 $$
 aT+bN \mapsto a+b
 $$
(i.e., evaluating the characters $\chi \mapsto \chi(1)$). So
 $$
  e(\overline{X}{}_{4,\lambda}) = (q^3-1) + (3q^2-3q) = q^3+3q^2-3q-1 \, .
 $$

The E-polynomial of the total space is evaluated using Proposition \ref{prop:e-total-space},
 $$
  aT+bN \mapsto (q-1) a - 2(a+b) = (q-3)a-2b\, .
 $$
Using Theorem \ref{thm:R(overline)}, we get
 $$
 e(\overline{X}_4)= (q-3)(q^3-1)-2(3q^2-3q)= q^4-3q^3-6q^2+5q+3\, ,
 $$
which agrees with (\ref{eqn:e(x4-bar)}).

\section{The quotient of $\overline{X}_4$ by $\ZZ_2$} \label{sec:representation2}

Now we recall that there is an action of $\ZZ_2$ on $\overline{X}_4$, given by
$$(A,B,\lambda) \mapsto (P_0AP_0^{-1},P_0BP_0^{-1},\lambda^{-1}),$$ where
$P_0=\left(\begin{array}{cc} 0 & 1 \\ 1 & 0 \end{array}\right)$. Then there is a
diagram
  $$
\xymatrix{
\overline{X}_4 \ar[r] \ar[d] &\overline{X}_4/\ZZ_2 \ar[d]\\
\CC\setminus\{0,\pm 1\}\ar[r]& \CC\setminus\{\pm 2 \},
}
$$
where the bottom map is $\lambda \mapsto s=\lambda+\lambda^{-1}$.
The fibration $\overline{X}_4/\ZZ_2 \to \CC\setminus \{\pm 2\}$
is given by $(A,B) \mapsto \tr [A,B]=s=\lambda + \lambda^{-1}$.

\subsection{The Hodge monodromy representation of $\overline{X}_4/\ZZ_2$}
We want to calculate the Hodge monodromy representation
  $$
  R(\overline{X}_4/\ZZ_2) \in R(\Gamma')[q] \, ,
  $$
where $\Gamma'=H_1(\CC\setminus \{\pm 2\})$.

Write $B=\CC\setminus\{0,\pm 1 \}$ and $B'=\CC\setminus\{\pm 2 \}$, where
$B'=B/\ZZ_2$. Let $\eta_2,\eta_{-2}$ denote small loops in $B'$ around $2$ and $-2$,
respectively. The group $\Gamma'$ is freely generated by $\eta_2,\eta_{-2}$.
There is an exact sequence
  $$
  \Gamma \to \Gamma' \to \ZZ_2 \to 0\, ,
  $$
where $\gamma_1\mapsto 2 \eta_2$, $\gamma_{-1}\mapsto 2 \eta_{-2}$
and $\gamma_0\mapsto \eta_{-2}+\eta_2$. By Theorem \ref{thm:R(overline)}, the
action of $\gamma_{\pm 1}$ on $\overline{X}_4$ is trivial, and the
action of $2\gamma_0$ is also trivial.
This implies that the actions of $2\eta_2$ and $2\eta_{-2}$ are trivial. Hence,
the Hodge monodromy representation of $\overline{X}_4/\ZZ_2$
descends to a representation of the quotient group
  $$
  \Gamma' = H_1(B') \twoheadrightarrow \ZZ_2 \times \ZZ_2 \, .
  $$
The conclusion is that
   $$
    R(\overline{X}_4/\ZZ_2) \in R(\ZZ_2\times \ZZ_2)[q] \, .
   $$

We shall denote by $T,S_2,S_{-2}, S_0$ the four irreducible representations of $\ZZ_2\times \ZZ_2$.
The trivial representation is $T$, $S_2$ is characterised by $\rho(\eta_2)=\Id$, $S_{-2}$ is
characterised by $\rho(\eta_{-2})=\Id$, and $S_0$ is characterised by $\rho(\gamma_0)=\Id$, where
$\gamma_0=\eta_2+\eta_{-2}$.

\begin{thm}\label{thm:R(overline/Z2)}
  For $\overline{X}_4/\ZZ_2$, we have
    $$
    R(\overline{X}_4/\ZZ_2) = q^3 \, T -3q \, S_2 + 3q^2 \, S_{-2} -S_0
     \in R(\ZZ_2\times \ZZ_2)[q]\, .
    $$
\end{thm}

\begin{proof}
As in the proof of Theorem \ref{thm:R(overline)},
we shall follow the stratification described in subsection \ref{4.5}.
Note that $\ZZ_2$ acts on each of the strata $\overline{X}{}_4^i$.

We need to check the monodromy when we describe small loops around $-2$ and $2$
in the $s$-plane $\CC\setminus \{\pm 2\}$. 
The composition of the two loops is homotopic to a big loop encircling both $\pm 2$.
To understand the monodromy action, we have to lift the loops to paths in $B=\CC\setminus \{0,\pm 1\}$.
Due to the monodromy around $0$, we cannot trivialize the local system on $B$. So
we cut $B$ along the axis $L=\{-y \, i \ | \ y>0\}$, and take a determination of $\sqrt{\lambda}$ there.
In this way, we can identify all fibres over $B\setminus L$.
The image of $L$ under $h:\lambda\mapsto s=\lambda+\lambda^{-1}$ is the whole imaginary axis.
We shall fix the fibre over $b_0=1+\epsilon$, where $\epsilon>0$ is a small real number, and
identify all fibres to it.

\begin{enumerate}
 \item Going around $\eta_2$ consists (in the $\lambda$-plane) of a path from $1+\epsilon$ to $1/(1+\epsilon)$,
followed by the $\ZZ_2$-action.

\item Going around $\eta_{-2}$ consists of a path from $-1+\epsilon$ to $1/(-1+\epsilon)$ followed by
the $\ZZ_2$-action. Here we have to identify the fiber over $-1+\epsilon$ with the fiber over $b_0=1+
\epsilon$ via a path $\sigma$ in the $\lambda$-plane. Returning from  $-1+\epsilon$ to $b_0$ is
the $\ZZ_2$-transform of going from $1/(-1+\epsilon)$ to $1/(1+\epsilon)$ via the
transform of $\sigma^{-1}$. We can take $\sigma$ not to cross $L$, and then this new path does
cross $L$ once.

\item Going around $\gamma_0=\eta_2+ \eta_{-2}$ consists of doing both loops and joining them. This
crosses the vertical axis. Taking the preimage under $h$, it gives a loop that encircles $0$ and
crosses $L$ once.
\end{enumerate}

Note that the action of $\gamma_0$ is the composition of the actions of $\eta_2$ and $\eta_{-2}$. We shall
use below this fact in the form: the action of $\eta_{-2}$ is the composition of the actions of $\eta_2$
and $\gamma_0$. The action of $\gamma_0$ is basically known thanks to the computations in the proof
of Theorem \ref{thm:R(overline)}.

To abbreviate, we shall write $F^i=
\overline{X}{}_{4,\lambda}^{i} \cong S^i_\lambda \times \CC^*$ for the fiber of $\overline{X}{}_{4,\lambda}^i\too B'$.
For a loop $\alpha =\gamma_0, \eta_2, \eta_{-2}$, we shall use the notation
$e(F^i)^\alpha$ to mean the E-polynomial of $H^*_c(F^i)^\alpha$, i.e., the invariant part of the
cohomology $H^*_c(F^i)$ under the action of $\alpha$.
Also $e(F)^{inv}$ will stand for the E-polynomial of $H^*_c(F)^{\G'}$, as usual.
We have the following:

 \begin{itemize}
  \item $\overline{X}{}_4^0/\ZZ_2$. The fibre consists of $2$ copies of $\CC^*$. We already know
  that $\gamma_0$ acts by permuting the two copies of $\CC^*$. Note that the parameter
  for $\CC^*$ is $y\neq 0$, and that the $\ZZ_2$-action sends $y\mapsto z=-1/y$. The path
corresponding to $\eta_2$ does not permute the copies of $\CC^*$, and turns $\CC^*$ inside out
  (sending $0$ to $\infty$ and vice-versa). Therefore, $\eta_{-2}$ swaps the two copies of
  $\CC^*$ (and sends $0$ to $\infty$ and vice-versa).
  Thus
  \begin{eqnarray*}
    e(F^0)             &=& 2(q-1), \\
    e(F^0)^{inv}       &=& q,\\
    e(F^0)^{\eta_2}   &=& 2q, \\
    e(F^0)^{\eta_{-2}} &=& q-1, \\
    e(F^0)^{\gamma_0} &=& q-1.
  \end{eqnarray*}
  The conclusion is that $R(\overline{X}{}_4^0/\ZZ_2)=a T+ bS_2+cS_{-2} +d S_0$, where
  $a=e(F)^{inv}$, $a+b=e(F)^{\eta_2}$, $a+c=e(F)^{\eta_{-2}}$ and $a+d=e(F)^{\gamma_0}$.
  Thus
   $$
   R(\overline{X}{}_4^0/\ZZ_2)= q T+ q S_2 - S_{-2} - S_0\, .
   $$
  Note that $e(F^0)=a+b+c+d=2q-2$.

  \item $\overline{X}{}_4^1/\ZZ_2$. Exactly as above, $R(\overline{X}{}_4^1/\ZZ_2)= q T+ q S_2 - S_{-2} - S_0$.

  \item $\overline{X}{}_4^2/\ZZ_2$. By our previous analysis, we have a $\CC^*$-bundle
  (the slice fixing condition) over a collection of two sets (given by either $b\neq 0, c=0$ or $b=0, c\neq 0$),
each of which consists of
two copies of $\CC^*$ (given by the two solutions of $a^2=\lambda$).
  The involution $\gamma_0$ interchanges the two values of $a$.
  On the other hand, the $\ZZ_2$-action interchanges the two sets. So
  \begin{eqnarray*}
    e(F^2)            &=&4(q-1)^2,\\
    e(F^2)^{inv}      &=& (q-1)^2,\\
    e(F^2)^{\eta_2}  &=&  2(q-1)^2, \\
    e(F^2)^{\eta_{-2}}&=& 2(q-1)^2, \\
    e(F^2)^{\gamma_0}&=&  2(q-1)^2.
  \end{eqnarray*}
  Thus
      $$
   R(\overline{X}{}_4^2/\ZZ_2)= (q-1)^2 T+ (q-1)^2 S_2 +(q-1)^2S_{-2} +(q-1)^2 S_0\, .
   $$

  \item  $\overline{X}{}_4^3/\ZZ_2$. As before,
  $R(\overline{X}{}_4^3/\ZZ_2)= (q-1)^2 T+ (q-1)^2 S_2 +(q-1)^2S_{-2} +(q-1)^2 S_0$.

  \item $\overline{X}{}_4^4/\ZZ_2$. Now we have $8$ copies of $\CC^*$ ($2$ solutions
  to $a^2=\lambda$; a choice of sign in $x=\pm \frac1a$; and the choice of whether $b=z=0$ or $c=y=0$).
  The action of $\gamma_0$ interchanges the copies of $\CC^*$
  in pairs (corresponding to the solutions to $a^2=\lambda$). The action of $\eta_2$ (or $\eta_{-2}$) interchanges
  the strata $b=z=0$ and $c=y=0$. So finally,
  \begin{eqnarray*}
    e(F^4)            &=&8(q-1),\\
    e(F^4)^{inv}      &=& 2(q-1),\\
    e(F^4)^{\eta_2}  &=&  4(q-1), \\
    e(F^4)^{\eta_{-2}}&=& 4(q-1), \\
    e(F^4)^{\gamma_0}&=&  4(q-1).
  \end{eqnarray*}
  Thus
      $$
   R(\overline{X}{}_4^4/\ZZ_2)= 2(q-1) T+ 2(q-1) S_2 +2(q-1)S_{-2} +2(q-1) S_0\, .
   $$

  \item $\overline{X}{}_4^5/\ZZ_2$.
  Now there are two possibilites $a=\pm \frac{2\lambda}{\lambda+1}$. Consider
one of the possibilities, say  $a=\frac{2\lambda}{\lambda+1}$ (the other one is
analogous). Then the fibre $F^5$ is a $\CC^*$-bundle over $S^5_\lambda$, which consists of two lines
(given by equation (\ref{eq:x4''-conic0})),
parameterised by a square root of $\lambda$,
each with two points removed. The action of $\gamma_0$ interchanges the two lines.

To find the monodromy of $\eta_2$, we have to consider how the $\ZZ_2$-action behaves in
a small neighbourhood of $\lambda=1$. There, we have a well-defined square root
$\sqrt{\lambda}$. Hence the two lines have equations
 \begin{equation} \label{eqn:two-lines}
 \sqrt{\lambda}\, x - \sqrt{\lambda} \frac{\lambda-1}{\lambda+1} y = \pm 1\, .
 \end{equation}
 The $\ZZ_2$-action sends $\sqrt{\lambda} \mapsto \sqrt{\lambda}{}^{-1}$,
$\sqrt{\lambda}\, x  \mapsto \sqrt{\lambda} \, x$,
and
 \begin{eqnarray*}
 \sqrt{\lambda} \frac{\lambda-1}{\lambda+1} y &\mapsto &
 \sqrt{\lambda}{}^{-1} \frac{\lambda^{-1}-1}{\lambda^{-1}
 +1} \left( \lambda y + \frac{\lambda(\lambda -1)a}{a^2-\lambda} x \right) \\
& & =-\sqrt{\lambda} \frac{\lambda-1}{\lambda+1} y + 2 \sqrt{\lambda} x\, .
 \end{eqnarray*}
Therefore the two lines (\ref{eqn:two-lines}) get interchanged.

  By composition, the action of $\eta_{-2}$ preserves each line. With the notation
of subsection \ref{4.5.7}, each line may be described as $F=\CC\setminus \{\pm 1\}$,
and the $\ZZ_2$-action swaps $\pm 1$. As we have a $\CC^*$-bundle over
  $F$, we compute using $e(\CC^*)^+=q$,  $e(\CC^*)^-=-1$,
  $e(F)^+=q-1$, $e(F)^-=-1$, to obtain the Hodge-Deligne polynomial $q(q-1)+1$. Doubling all polynomials to account for the other value $a=-\frac{2\lambda}{\lambda+1}$, we summarise this as
   \begin{eqnarray*}
    e(F^5)             &=&  4(q-1)(q-2), \\
    e(F^5)^{inv}      &=&  2(q^2-q+1),\\
    e(F^5)^{\eta_2} &=&  2(q-1)(q-2),\\
    e(F^5)^{\eta_{-2}}&=&  4(q^2-q+1),\\
    e(F^5)^{\gamma_0}&=&  2(q-1)(q-2).
  \end{eqnarray*}
  Thus
      $$
   R(\overline{X}{}_4^5/\ZZ_2)= 2(q^2-q+1) T+ 2(-2q+1) S_2 + 2(q^2-q+1)S_{-2} +2(-2q+1) S_0\, .
   $$

 \item $\overline{X}{}_4^6/\ZZ_2$. We have a $\CC^*$-bundle over a family of complete
   smooth conics over the space of $a\in \CC \setminus \{\pm \sqrt{\lambda},\pm \frac{2\lambda}{\lambda+1}\}$.
  The action of $\gamma_0$ interchanges the first two of
  the points in the line parameterised by $a$. The action of $\eta_{-2}$ does the same, but
  turns $\CC^*$ inside-out. The action of $\eta_2$ fixes the line and turns $\CC^*$ inside-out.
 Multiplying all polynomials by $(q+1)$ since we are dealing with conic bundles, we obtain
 \begin{eqnarray*}
     e(F^6)            &=& (q-1)(q-4)(q+1), \\
     e(F^6)^{inv}      &=& q(q-3)(q+1),\\
     e(F^6)^{\eta_2}  &=&  q(q-4)(q+1),\\
     e(F^6)^{\eta_{-2}}&=& (q(q-3)+1)(q+1), \\
     e(F^6)^{\gamma_0}&=& (q-1)(q-3)(q+1).
  \end{eqnarray*}
  Thus
   $$
   R(\overline{X}{}_4^6/\ZZ_2)= (q+1) \big( (q^2-3q) T-q S_2+ S_{-2} -(q-3) S_0\big) .
   $$

\item $\overline{X}{}_4^7/\ZZ_2$. By our previous discussion, the fibre consists of a
 $\CC^*$-bundle over two copies of the variety consisting of the nodal curve $N$ which is
 two lines $\CC\setminus \{4 \, \text{points}\}$ joined by the origin. Note that $e(N)=2q-9$.
 The action of $\gamma_0$ interchanges the two nodal curves in pairs. Clearly,
 $\eta_{-2}$ does the same, and $\eta_2$ fixes each nodal curve. The nodal curve
 is parameterised by $y \in \{0,  a(1-\lambda) x/(a^2-\lambda)\}$, where
  $a\in \CC\setminus \{ \pm \sqrt\lambda, \pm \frac{2\lambda}{\lambda+1}\}$.
 We know that the $\ZZ_2$-action sends $y=0 \mapsto y= a(1-\lambda) x/(a^2-\lambda)$,
 so the quotient $N/\ZZ_2$ is a punctured line $\CC\setminus \{4 \, \text{points}\}$,
hence  $e(N)^+=q-4$ and $e(N)^-=q-5$. As $e(\CC^*)^+=q$, $e(\CC^*)^-=-1$, we
 get the polynomial $(q-4)q - (q-5)=q^2-5q+5$. Hence
 \begin{eqnarray*}
        e(F^7)            &=& 2(q-1)(2q-9), \\
        e(F^7)^{inv}      &=& q^2-5q+5, \\
        e(F^7)^{\eta_2}  &=& 2(q^2-5q+5), \\
        e(F^7)^{\eta_{-2}}&=& (q-1)(2q-9), \\
        e(F^7)^{\gamma_0}&=& (q-1)(2q-9).
  \end{eqnarray*}
  Thus
   $$
   R(\overline{X}{}_4^7/\ZZ_2)= (q^2-5q+5) T+  (q^2-5q+5) S_2 +(q^2-6q+4)
    S_{-2} +(q^2-6q+4) S_0\, .
   $$

\item $\overline{X}{}_4^8/\ZZ_2$. We follow the notations of subsection
 \ref{4.5.10}. The fiber $F^8$ is a $\CC^*$-bundle over the hyperbola
  $\Omega^2=A^2 -4$ minus six points: $(A,\Omega)=(\pm 2, 0)$ and the four points
  $(\pm\frac{1+\lambda}{\sqrt{\lambda}}, \pm \frac{1-\lambda}{\sqrt{\lambda}})$ (all signs allowed).

  The monodromies around $\gamma_0$, $\eta_{\pm 2}$ fix the points at infinity and the first two points.
  Let us see what happens with the remaining four points: they are parameterised by
  $\Omega^2=A^2 -4$ and $A^2 =s+2$, where $s=\lambda+\lambda^{-1}$.
  Clearly, $\gamma_0$  changes both signs simultaneously $(r_1,r_2) \mapsto (-r_1,-r_2)$,
  $(r_1,r_2)\in \{(\pm\frac{1+\lambda}{\sqrt{\lambda}}, \pm \frac{1-\lambda}{\sqrt{\lambda}})\}$.
 The map $(A,\Omega) \mapsto s=A^2-2$ ramifies at $A=0$, $s=-2$, so
$\eta_{-2}$ has the effect $(r_1,r_2) \mapsto (-r_1,r_2)$.
Similarly,
$(A,\Omega) \mapsto s=\Omega^2+2$ ramifies at $\Omega=0$, $s=2$, so
$\eta_{2}$ has the effect $(r_1,r_2) \mapsto (r_1,-r_2)$.
 Recalling that $\eta_{\pm 2}$ also turns $\CC^*$ inside-out, we obtain
  \begin{eqnarray*}
      e(F^8)            &=& (q-1)(q-7), \\
      e(F^8)^{inv}      &=& a,\\
      e(F^8)^{\eta_2}  &=& q(q-5)+2, \\
      e(F^8)^{\eta_{-2}}&=& q(q-5)+2, \\
      e(F^8)^{\gamma_0}&=& (q-1)(q-5),
  \end{eqnarray*}
 for some number $a$. Let us determine the value of $a$ :
   $$
   R(\overline{X}{}_4^8/\ZZ_2)= a T+  (q^2-5q+2-a )S_2 +(q^2-5q+2-a ) S_{-2} +(q^2-6q+5-a ) S_0\, .
   $$

Considering the E-polynomial of the fibre, we get $3q^2-16q+9-2a=q^2-8q+7$, and hence $a=q^2-4q+1$.
This gives
   $$
   R(\overline{X}{}_4^8/\ZZ_2)= (q^2-4q+1) T+  (-q+1)S_2 + (-q+1)S_{-2} +(-2q+4) S_0\, .
   $$
  \end{itemize}

\medskip

Adding up all contributions together, we get
   $$
   R(\overline{X}{}_4/\ZZ_2)= q^3 T -3q S_2 + 3q^2 S_{-2} -S_0\,.
   $$
\end{proof}

\begin{rmk}
 In the proof of Theorem \ref{thm:R(overline/Z2)}, we have chosen a base point $b_0$.
This produces an asymmetry in some of the calculations. If we chose a different
base point (or a different cut-locus for the plane $B$), some of the intermediate
calculations would change, but the final answer would be the same. The asymmetry
between the coefficients of $S_2$ and $S_{-2}$ in $R(\overline{X}_4/\ZZ_2)$ is inherent
to the geometry of the situation. 
\end{rmk}

It is clear that if
   $$
    R(\overline{X}_4/\ZZ_2) =a T+ bS_2+cS_{-2}+ d S_0
    $$
then
   $$
    R(\overline{X}_4) =(a +d)T+ (b+c)N = (q^3-1)T + (3q^2-3q)N\, ,
    $$
agreeing with Theorem \ref{thm:R(overline)}. For the fibre,
   $$
   e(\overline{X}_{4,\lambda})=a+b+c+d= q^3+3q^2-3q -1\, ,
   $$
agreeing with (\ref{eqn:x4lambda}).

We also recover the E-polynomial of the total space $e(\overline{X}_4/\ZZ_2)$
using Proposition \ref{prop:pure-type} and Proposition \ref{prop:e-total-space},
  \begin{eqnarray}\label{eqn:e(X4/Z2)}
  e(\overline{X}_4/\ZZ_2) &=& (q-1)a - (a+b+c+d) \nonumber\\
   &=&  (q-2)a-(b+c+d) \nonumber\\
   &=& (q-2) q^3 - (3q^2-3q-1) \nonumber \\
 &=& q^4-2q^3-3q^2+3q+1.
  \end{eqnarray}

\begin{rmk} \label{rmk:the-e}
More generally, suppose that we have a fibration
  $$
  \overline{X} \too B=\CC \setminus \{0,\pm1 \} ,
  $$
with a $\ZZ_2$-action compatible with the action $\lambda\mapsto \lambda^{-1}$ in
the base. Then we have a fibration
  $$
  \overline{X}/\ZZ_2 \too B'=\CC\setminus \{\pm 2\}\, .
  $$
Suppose that the Hodge monodromy representation is of the form
  $$
  R(\overline{X}/\ZZ_2)=a T + b S_2+c S_{-2} + d S_0\, ,
 $$
then we have
\begin{eqnarray*}
 e(\overline{X}) &=& (q-3)(a+d) -2(b+c), \\
 e(\overline{X})^+ &=&(q-2)a - (b+c+d), \\
 e(\overline{X})^- &=& (q-2)d- (a+b+c).
\end{eqnarray*}
\end{rmk}

\subsection{Application: Recovering the E-polynomial of $X_4$} \label{subsec:recovering}
Consider the set
 $$
 \widetilde{X}_4=\{ (A,B,\ell) \ | \ \xi=[A,B],\,  \tr(\xi)=\lambda+ \lambda^{-1} , \lambda\neq 0,\pm 1, \ell
 \text{ eigenspace of }\xi\}\,.
 $$
There are fibrations
 \begin{equation}\label{eqn:fibrations}
 \xymatrix{
 \CC^{\ast} \ar[r] \ar@{=}[d]& \PGL(2,\CC) \times \overline{X}_{4} \ar[r]\ar[d] & \widetilde{X}_{4}\ar[d]^{2:1}\\
 \CC^{\ast} \ar[r] & (\PGL(2,\CC)\times \overline{X}_{4})/\ZZ_{2} \ar[r] & X_{4}}
 \end{equation}
where the top map is
$(P,A,B,\lambda) \mapsto \big(PAP^{-1},PBP^{-1}, P \left(\begin{array}{c} 1\\
0\end{array}\right) \big)$, which is a principal bundle with fibre $D/\CC^*  \cong \CC^*$, given
by $\mu \sim \left(\begin{array}{cc} \mu & 0\\ 0& 1\end{array}\right)$.
The action by $\ZZ_2$ in the middle is given by $(P,A,B,\lambda) \mapsto
(PP_0^{-1},P_0AP_0^{-1},P_0BP_0^{-1},\lambda^{-1})$, $P_0=
\left(\begin{array}{cc} 0& 1\\ 1& 0\end{array}\right)$. On the right hand column, we have
the $2:1$-map $(A,B,\ell) \mapsto (A,B)$. Note that the $\ZZ_2$-action on the fibre is
 $$
 \mu \sim \left(\begin{array}{cc} \mu & 0\\ 0& 1\end{array}\right)  \mapsto
 \left(\begin{array}{cc} 1 & 0\\ 0& \mu\end{array}\right) \sim
 \left(\begin{array}{cc} \mu^{-1} & 0\\ 0& 1\end{array}\right) \sim \mu^{-1}\, .
 $$

{}From (\ref{eqn:fibrations}) and using Remark \ref{rem:new}, we get first that
 $$
 e(\widetilde{X}_{4})=\frac{e(\PGL(2,\CC)) e(\overline{X}_{4})}{e(\CC^{\ast})} =
 q(q+1)  e(\overline{X}_{4}) .
 $$
Now Proposition \ref{prop:fibration} and equation \eqref{eqcc} yield
 \begin{eqnarray*}
 e((\PGL(2,\CC)\times \overline{X}_{4})/\ZZ_{2})&=& e(\widetilde{X}_{4})^+
  e(\CC^{\ast})^{+}+ e(\widetilde{X}_{4})^- e(\CC^{\ast})^-\\
 &=& e(X_{4}) q - (e(\widetilde{X}_{4})-e(X_{4})) \\
 &=&  (q+1) e(X_4) - q(q+1)  e(\overline{X}_{4}).
\end{eqnarray*}
On the other hand, using Proposition \ref{prop:rem:extra},
\begin{eqnarray*}
 e((\PGL(2,\CC)\times \overline{X}_{4})/\ZZ_2) &=&e((\PGL(2,\CC)\times \overline{X}_{4})^+\\&=& e(\PGL(2,\CC))^+ e(\overline{X}_4)^+ +
 e(\PGL(2,\CC))^- e(\overline{X}_4)^- \\
  &=& e(\PGL(2,\CC)) e(\overline{X}_4/\ZZ_2) \\
  &=& q(q^2-1)e(\overline{X}_4/\ZZ_2) \, .
\end{eqnarray*}

Equating these two formulae, we get
  $$
  q(q^2-1)e(\overline{X}_4/\ZZ_2) =(q+1) e(X_4) - q(q+1)  e(\overline{X}_{4})
  $$
and hence
 $$
 e(X_4)= (q^2-q)e(\overline{X}_4/\ZZ_2) +q \, e(\overline{X}_{4}).
 $$
Now use (\ref{eqn:e(X4/Z2)}) and (\ref{eqn:e(x4-bar)})
to get
\begin{eqnarray*}
 e(X_4) &=& (q^2-q)(q^4-2q^3-3q^2+3q+1) +q(q^4-3q^3-6q^2+5q+3) \\
 &=& q^6-2q^5-4q^4+3q^2+2q\, ,
\end{eqnarray*}
agreeing with our previous formula (\ref{eqn:e(X_4)}).

\begin{rmk}\label{rmk:the-e-continuation}
  Suppose that we are in the situation of Remark \ref{rmk:the-e}, and that we have a diagram
 $$
 \xymatrix{
 \CC^{\ast} \ar[r] \ar@{=}[d]& \PGL(2,\CC) \times \overline{X}\ar[r]\ar[d] & \widetilde{X}\ar[d]^{2:1}\\
 \CC^{\ast} \ar[r] & (\PGL(2,\CC)\times \overline{X})/\ZZ_{2} \ar[r] & X}
 $$
in which the action of $\ZZ_2$ on the fiber $\CC^*$ takes $\mu\mapsto \mu^{-1}$. Then
 \begin{eqnarray}
 e(X) &=& (q^2-q)e(\overline{X})^+ \, +q \, e(\overline{X}) \notag \\
 &=& q(q-1)\big( (q-2)a -(b+c+d) \big) + q \big( (q-3) (a+d) - 2(b+c) \big) \label{eqn:main-formula} \\
 \notag &=& q(q^2-2q-1) a- q(q+1) (b+c) -2 q d \, ,
 \end{eqnarray}
where $R(\overline{X}/\ZZ_2)=aT +b S_2+cS_{-2}+d S_0$.
\end{rmk}


\section{E-polynomial of the $SL(2,\CC)$-character variety for genus $2$}\label{sec:g=2}

We consider now the map
 \begin{eqnarray*}
 F:\SL(2,\CC)^4  &\longrightarrow& \SL(2,\CC), \\
   (A,B,C,D) &\mapsto & [A,B][C,D],
 \end{eqnarray*}
where $\PGL(2,\CC)$ acts on both spaces by conjugation and $F$ is equivariant.
The character variety to consider is then
 $$
 \cM_{\Id}=Y//\PGL(2,\CC),
 $$
where
 $$
 Y=F^{-1}(\Id)=\{(A,B,C,D)\in \SL(2,\CC)^4\, |\, [A,B]=[D,C]\}.
 $$
(Here we have to take the GIT quotient because there are reducibles.)

In order to compute its E-polynomial, we stratify $Y$ as follows.
\begin{itemize}
\item $Y_0:=\{(A,B,C,D) |\,[A,B]=[D,C]=\Id\}$,
\item $Y_1:=\{(A,B,C,D) |\, [A,B]=[D,C]=-\Id\}$,
\item $Y_2:=\left\{(A,B,C,D) |\,[A,B]=[D,C]\sim\small{\left(
    \begin{array}{cc}
      1 & 1\\
      0 & 1\\
    \end{array}
  \right) } \right\}$,
  \item $Y_3:=\left\{(A,B,C,D) |\,[A,B]=[D,C]\sim\small{\left(
    \begin{array}{cc}
      -1 & 1\\
      0 & -1\\
    \end{array}
  \right)} \right\}$,
\item $Y_4:=\left\{(A,B,C,D) |\,[A,B]=[D,C]\sim\small{\left(
    \begin{array}{cc}
      \lambda & 0\\
      0 & \lambda^{-1}\\
    \end{array}
  \right)}, \ \lambda\neq 0,\pm 1\right\}$.
\end{itemize}

Therefore
 \begin{equation}\label{eqn:yyy}
 Y=\bigsqcup_{i=0}^{4} Y_{i}.
 \end{equation}

We also introduce the sets
\begin{itemize}
\item $\overline{Y}_0:=\{(A,B,C,D) |\,[A,B]=[D,C]=\Id\}=Y_0$,
\item $\overline{Y}_1:=\{(A,B,C,D) |\, [A,B]=[D,C]=-\Id\}=Y_1$,
\item $\overline{Y}_2:=\left\{(A,B,C,D) |\,[A,B]=[D,C]=\small{\left(
    \begin{array}{cc}
      1 & 1\\
      0 & 1\\
    \end{array}
  \right) } \right\}$,
  \item $\overline{Y}_3:=\left\{(A,B,C,D) |\,[A,B]=[D,C]=\small{\left(
    \begin{array}{cc}
      -1 & 1\\
      0 & -1\\
    \end{array}
  \right)} \right\}$,
\item $\overline{Y}_4:=\left\{(A,B,C,D) |\,[A,B]=[D,C]=\small{\left(
    \begin{array}{cc}
      \lambda & 0\\
      0 & \lambda^{-1}\\
    \end{array}
  \right)}, \ \lambda\neq 0,\pm 1\right\}$,
\item
 $
  \overline{Y}_{4,\lambda} :=\left\{(A,B,C,D) |\,[A,B]=[D,C]=\small{\left(
    \begin{array}{cc}
      \lambda & 0\\
      0 & \lambda^{-1}\\
    \end{array}
  \right)} \right\}
  $,
for $\lambda\neq 0,\pm 1$.
\end{itemize}

\subsection{E-polynomial of $Y$}
We use the stratification (\ref{eqn:yyy}).
\begin{itemize}
\item E-polynomial of $Y_0$. Clearly $Y_0=X_0 \times X_0$, hence
 $$
 e(Y_0 )=e(X_0)^2=(q^{4}+4q^{3}-q^{2}-4q)^{2}=
 q^8  + 8 q^7 + 14 q^6 - 16 q^5 - 31 q^4 + 8 q^3  +16 q^2\, .
 $$

\item E-polynomial of $Y_1$.
Similarly,  $Y_1 =X_1 \times X_1$, so
 $$
 e(Y_1) =e(X_1)^2=(q^3 - q) ^2= q^6- 2 q^4 + q^2 \, .
 $$

\item E-polynomial of $Y_2$.
Clearly,
 $$
 \overline{Y}_2=\overline{X}_2\times \overline{X}{}_2\, ,
 $$
so $e(\overline{Y}_2)=e(\overline{X}_2)^2$.
The stabiliser of $\xi=\left(
    \begin{array}{cc}
      1 & 1\\
      0 & 1\\
    \end{array}
  \right)$
is $U=\CC^{\ast}\times \CC$. So we have a fibration
 $$
 U \longrightarrow \GL(2,\CC)\times \overline{Y}_{2} \longrightarrow Y_{2}\, .
 $$
Hence
 \begin{eqnarray*}
 e(Y_{2}) &=& e(\GL(2,\CC)/U) e(\overline{X}_{2})^2=(q^2-1)(q^3-2q^2-3q)^2 \\
  &=& q^8 - 4 q^7 - 3 q^6 + 16 q^5+ 11 q^4 - 12 q^3  -9 q^2.
 \end{eqnarray*}

\item E-polynomial of $Y_3$.
Again as above, $\overline{Y}_3=\overline{X}_3\times \overline{X}{}_3$, and
 \begin{eqnarray*}
 e(Y_{3}) &=& e(\GL(2,\CC)/U) e(\overline{X}_{3})^2=(q^2-1)(q^3+3q^2)^2 \\ &=& q^8 + 6 q^7+ 8 q^6 - 6 q^5 -9 q^4.
 \end{eqnarray*}

\item E-polynomial of $Y_4$. It is clear that
 \begin{equation}\label{eqn:Y4}
 \overline{Y}_{4,\lambda}=\overline{X}_{4,\lambda}\times \overline{X}{}_{4,\lambda}\, .
 \end{equation}
There is a fibration
  $$
 \overline{Y}_4 \too B=\CC\setminus\{0,\pm 1\},
  $$
with a $\ZZ_2$-action that covers the action $\lambda\mapsto \lambda^{-1}$ on the base. The
Hodge monodromy representation of
  $$
 \overline{Y}_4/\ZZ_2 \too B'=\CC\setminus\{\pm 2\}
  $$
is given by
 $$
 R(\overline{Y}_4/\ZZ_2) = R(\overline{X}_4/\ZZ_2)\otimes R(\overline{X}_4/\ZZ_2).
 $$
Write $R(\overline{X}_4/\ZZ_2)=aT +b S_2+cS_{-2}+d S_0$, where
 $$
 a=q^3,\ b=-3q,  \ c=3q^2,\ d=-1.
 $$
Now, using the fact that $T\otimes S_i=S_i$, $i=0,\pm2$, $S_0\otimes S_2=S_{-2}$, $S_0\otimes S_{-2}=S_{2}$
and $S_2\otimes S_{-2}=S_{0}$, we have
\begin{eqnarray}\label{eqn:R(Y4/Z2)}
 R(\overline{Y}_4/\ZZ_2) &=& 
 (a^2+b^2 +c^2 + d^2 ) T+2(ab +cd )S_2 +2(ac+ bd) S_{-2} + 2(ad+ bc ) S_0 \notag\\
 &=& (q^6 + 9q^4 + 9q^2 + 1)T - 6(q^4 + q^2)S_2 + 6(q^5 + q)S_{-2} -20q^3 S_0\, .
\end{eqnarray}

We introduce the set
 $$
 \widetilde{Y}_4=\{(A,B,C,D, \ell) |\, [A,B]=[D,C]=\xi, \, \tr \xi= \lambda+
 \lambda^{-1}, \, \lambda\neq 0,\pm 1, \, \ell \text{ eigenspace of } \xi\}.
 $$
Then we have a diagram
 $$
 \xymatrix{
 \CC^{\ast} \ar @{=}[d]\ar[r] & \PGL(2,\CC) \times \overline{Y}_4 \ar[r]\ar[d] & \widetilde{Y}_4\ar[d]^{2:1}\\
 \CC^{\ast} \ar[r] & (\PGL(2,\CC) \times \overline{Y}_4)/ \ZZ_2 \ar[r]& Y_4
 }
$$
Using Remark \ref{rmk:the-e-continuation}, we have
 \begin{eqnarray*}
 e(Y_4) &=& q(q-1) e(\overline{Y}_4/\ZZ_2) + q\,  e(\overline{Y}_4) \\
 &=& q(q^2-2q-1) a'- q(q+1) (b'+c') -2 q d' \\
 &=& q^9 - 2q^8 + 2q^7 - 18q^6 + 6q^5 + 28q^4 - 8q^3 - 8q^2 - q,
 \end{eqnarray*}
where $R(\overline{Y}_4/\ZZ_2)=a'T +b' S_2+c'S_{-2}+d' S_0$ is given by (\ref{eqn:R(Y4/Z2)}).

\end{itemize}

Adding up all contributions above, we get
 $$
 e(Y)=q^9 + q^8 + 12q^7 + 2q^6 -3q^4 -12q^3 -q.
 $$

\subsection{The E-polynomial of $\cM_{\Id}$}
We want to calculate the E-polynomial of
 $$
 \cM_{\Id}=Y//\PGL(2,\CC)\, .
 $$
It is clear that there are no reducibles in $Y_1$ and $Y_4$. Also there are no
reducibles in $Y_3$ by Remark \ref{rmk:cMJ--red}. All elements in $Y_2$ are
reducible by Remark \ref{rmk:cMJ+-red}. Finally, $Y_0$ contains both reducibles
and irreducibles.

\subsection*{Contribution from reducibles} A reducible $(A,B,C,D)$ is
$S$-equivalent to a direct sum of two representations in $\CC^*$, given
as
  $$
   \left( \left(\begin{array}{cc} \lambda_1 & 0 \\ 0 &\lambda_1^{-1}  \end{array}\right),
   \left(\begin{array}{cc} \lambda_2 & 0 \\ 0 & \lambda_2^{-1} \end{array}\right),
   \left(\begin{array}{cc} \lambda_3 & 0 \\ 0 & \lambda_3^{-1} \end{array}\right),
   \left(\begin{array}{cc} \lambda_4 & 0 \\ 0 & \lambda_4^{-1} \end{array}\right) \right).
   $$

Therefore the space of $S$-equivalence classes of reducible representations is
${\mathcal R}=(\CC^*)^4/\ZZ_2$, where the $\ZZ_2$-action is
   $$
   (\lambda_1,\lambda_2,\lambda_3,\lambda_4)\mapsto
   (\lambda_1^{-1},\lambda_2^{-1},\lambda_3^{-1},\lambda_4^{-1}).
   $$
Using the fact that $e(\CC^*)^+=q$ and $e(\CC^*)=-1$, the above variety has E-polynomial
 \begin{eqnarray*}
  e({\mathcal R})&= & (e(\CC^*)^+)^4 + \binom{4}{2}  (e(\CC^*)^+)^2  (e(\CC^*)^-)^2+ (e(\CC^*)^-)^4 \\
  &=& q^4 + 6 q^2 + 1 \, .
 \end{eqnarray*}

\subsection*{Contribution from irreducibles in $Y_0$}
We consider the space
$${\mathcal I}:=\{\mbox{irreducible elements }
(A,B,C,D)| AB=BA, CD=DC\}/\PGL(2,\CC).$$ The irreducibility means that
there the four matrices have no common eigenvector.

Consider the trace map
 $$
  (A,B,C,D) \mapsto (t_1,t_2,t_3,t_4)=(\Tr A,\Tr B, \Tr C, \Tr D)\, .
 $$
We stratify ${\mathcal I}$ according to various cases:
\begin{itemize}
 \item ${\mathcal I}_1$ given by the conditions
 $(t_1,t_2)\neq (\pm 2,\pm 2)$ and $(t_3,t_4)\neq (\pm 2,\pm 2)$.
 This space is described with the techniques of \cite{vicente-knots}.

 The matrices $A,B$ are diagonalizable and have two well-defined common eigenvectors $e_1,e_2$
 (up to multiplication by scalars). Let $\lambda_1,\lambda_2$ be the eigenvalues of the eigenvector $e_1$
 for $A,B$, respectively. Also $C,D$ have two eigenvectors $f_1,f_2$, and let $\lambda_3,\lambda_4$ be
 the eigenvalues of $f_1$ for $C,D$, respectively.
 The four elements $[e_1],[e_2],[f_1],[f_2]\in \PP^1$ are different, so they have a well-defined cross-ratio
  $$
  r=([e_1],[e_2],[f_1],[f_2])\in \CC \setminus\{0,1\}.
  $$
There is an action of $\ZZ_2\x \ZZ_2$ given by $e_1,e_2 \mapsto e_2,e_1$ and $f_1,f_2\mapsto f_2,f_1$. This
yields
 \begin{eqnarray*}
  & &(\lambda_1,\lambda_2,\lambda_3,\lambda_4,r ) \sim
  (\lambda_1^{-1},\lambda_2^{-1},\lambda_3,\lambda_4,1-r ) \\
  && \sim (\lambda_1,\lambda_2,\lambda_3^{-1},\lambda_4^{-1},1-r ) \sim
  (\lambda_1^{-1},\lambda_2^{-1},\lambda_3^{-1},\lambda_4^{-1},r )\, .
 \end{eqnarray*}
We write $$U_1:=(\CC^*\x \CC^*)\setminus \{(\pm 1,\pm 1)\},\
U_2:=(\CC^*\x \CC^*)\setminus \{(\pm 1,\pm 1)\},\ U_3=\CC\setminus \{0, 1\}$$
and define
\begin{eqnarray*}
 \sigma_1: U_1\to U_1,  & & \sigma_1(\lambda_1,\lambda_2)=(\lambda_1^{-1},\lambda_2^{-1}), \\
 \sigma_2: U_2\to U_2, && \sigma_2(\lambda_3,\lambda_4)=(\lambda_3^{-1},\lambda_4^{-1}), \\
 \tau: U_3\to U_3, && \tau(r)=1-r\, .
\end{eqnarray*}
Then
 $$
 {\mathcal I}_1= (U_1\x U_2 \x U_3)/\langle \sigma_1\x \tau, \sigma_2\x \tau\rangle\, .
 $$

Thus $e(U_1)^+=e(U_2)^+=q^2-3$, $e(U_1)^-=e(U_2)^-= -2q$, $e(U_3)^+=q-1$ and $e(U_3)^-=-1$.
We have
\begin{eqnarray*}
 e({\mathcal I}_1) &=& e\left( (U_1\x U_2 \x U_3)^{\sigma_1\x \tau,\sigma_2\x \tau} \right) \\ &=&
  e(U_1)^+ e(U_2)^+ e(U_3)^+ + e(U_1)^- e(U_2)^- e(U_3)^- \\
  &=& (q^2-3)^2(q-1) - 4q^2 = q^5-q^4-6 q^3+2 q^2+9 q-9 \, .
\end{eqnarray*}

 \item ${\mathcal I}_2'$ corresponding to
 $(t_1,t_2)\neq (\pm 2,\pm 2)$, $(t_3,t_4)= (\pm 2,\pm 2)$ and $C=\pm \Id$, or $D=\pm \Id$.
 Note that we cannot have simultaneously $C,D=\pm \Id$, since in this case the element $(A,B,C,D)$
 is reducible.

We suppose $C=\Id$ and $D \sim J_+$ (the other cases are similar, and they account for $8$ cases).
The matrices $A,B$ define two eigenvectors $e_1,e_2$. The matrix $D$ has
 a unique eigenvector $f$, and $f \neq e_1,e_2$. Use the basis $\{f,e_2\}$. By multiplying the vectors by
 suitable scalars, we can suppose $D=J_+$. Then
 $$
  A=\left(\begin{array}{cc} \lambda_1 & 0 \\ x_1 & \lambda_1^{-1} \end{array}\right), \
  B= \left(\begin{array}{cc} \lambda_2 & 0 \\ x_2 & \lambda_2^{-1} \end{array}\right).
  $$
 Note that $(x_1,x_2)\neq (0,0)$, since $e_1 \neq f$.
The commutation relation $AB=BA$ translates to $x_1(\lambda_2-\lambda_2^{-1})=
x_2(\lambda_1-\lambda_1^{-1})$. Therefore there is one parameter $x_1\in \CC^*$ or $x_2\in \CC^*$
(depending on the case).
Note that $x_1=\tr (AD)- \tr A$, $x_2=\tr (BD)- \tr B$.

We have made a choice of basis. If we choose $\{f,e_1\}$ instead, then we get the change:
$(\lambda_1,\lambda_2,x_1) \sim   (\lambda_1^{-1},\lambda_2^{-1},x_1)$.
We deduce that
 \begin{eqnarray*}
  e({\mathcal I}_2') &=&8 (q-1) e(((\CC^* \x \CC^*)\setminus \{(\pm 1,\pm 1)\})/\ZZ_2) \\
  &=& 8(q-1)(q^2-3) = 8 q^3-8  q^2 -24 q +24 \, .
 \end{eqnarray*}

 \item ${\mathcal I}_2''$ corresponding to
 $(t_1,t_2)\neq (\pm 2,\pm 2)$, $(t_3,t_4)= (\pm 2,\pm 2)$ and $C,D$ of Jordan form.
 Suppose that $C,D\sim J_+$ (the other cases are similar).
As $C,D$ commute, they share an eigenvector $f$. Also $A,B$ have two common eigenvectors
$e_1,e_2$ (and $f\neq e_1,e_2$). Choose the basis $\{f,e_2\}$, and arrange that $C=J_+$. Then
 $$
  A=\left(\begin{array}{cc} \lambda_1 & 0 \\ x_1 & \lambda_1^{-1} \end{array}\right), \
  B= \left(\begin{array}{cc} \lambda_2 & 0 \\ x_2 & \lambda_2^{-1} \end{array}\right), \
  C= \left(\begin{array}{cc} 1 & 1 \\ 0 &1  \end{array}\right), \
  D= \left(\begin{array}{cc} 1 & x_4\\ 0 &1 \end{array}\right),
  $$
where $x_1,x_2,x_4\neq 0$. Again $AB=BA$ fixes a relation between $x_1$ and $x_2$.
Now $x_1=\tr(AC)-\tr A$ and $x_1x_4 = \tr(AD)-\tr A$, so $x_1,x_4$ are well-determined.
However, $(\lambda_1,\lambda_2,x_1,x_4) \sim   (\lambda_1^{-1},\lambda_2^{-1},x_1,x_4)$.
We get
 $$
  e({\mathcal I}_2'')=4(q-1)^2 (q^2-3) = 4 q^4-8 q^3 -8 q^2 +24 q -12\, .
  $$

 \item ${\mathcal I}_3={\mathcal I}_3'\cup {\mathcal I}_3''$ corresponding to
 $(t_1,t_2)= (\pm 2,\pm 2)$, $(t_3,t_4)\neq (\pm 2,\pm 2)$. This is similar to the
 previous two cases, ${\mathcal I}_2={\mathcal I}_2'\cup {\mathcal I}_2''$,
 $$
  e({\mathcal I}_3)=e({\mathcal I}_2)=e({\mathcal I}_2')+e({\mathcal I}_2'')=
  4 q^4-16 q^2+12\, .
  $$

\item ${\mathcal I}_4$ given by
$(t_1,t_2)=(\pm 2,\pm 2)$ and $(t_3,t_4)= (\pm 2,\pm 2)$, where one of
$A,B$ is $\pm \Id$ and the other is of Jordan form, and one of
$C,D$ is $\pm \Id$ and the other is of Jordan form (there are $64$ cases).

Suppose $A=\Id$, $B\sim J_+$, $C=\Id$, $D\sim J_+$. We take the eigenvector
$e$ of $B$ and the eigenvector $f$ of $D$. We arrange that $D=J_+$. So
 $B= \left(\begin{array}{cc} 1 & 0 \\ x_2 & 1 \end{array}\right)$, and $x_2\neq 0$.
 Thus $e({\mathcal I}_4)=64(q-1)=64q -64$.

\item ${\mathcal I}_5$ given by
$(t_1,t_2)=(\pm 2,\pm 2)$ and $(t_3,t_4)= (\pm 2,\pm 2)$, where both
$A,B$ are of Jordan form, and one of
either $C,D$ is $\pm \Id$ and the other is of Jordan form
(or both $C,D$ are of Jordan form, and one of
either $A,B$ is $\pm \Id$ and the other is of Jordan form). There are $64$ cases.

Suppose $A,B\sim J_+$, $C=\Id$, $D\sim J_+$. We can arrange that $D=J_+$ and
 $$
  A=\left(\begin{array}{cc} 1 & 0 \\ x_1 & 1 \end{array}\right), \
  B= \left(\begin{array}{cc} 1 & 0 \\ x_2 & 1 \end{array}\right),
 $$
with $x_1,x_2\neq 0$. So $e({\mathcal I}_5)=64(q-1)^2= 64q^2-128q + 64$.

\item ${\mathcal I}_6$ given by
$(t_1,t_2)=(\pm 2,\pm 2)$ and $(t_3,t_4)= (\pm 2,\pm 2)$, where all
$A,B,C,D$ are of Jordan form. We can assume $A,B,C,D\sim J_+$ (there are $16$ cases).
Working as before, we can arrange that
 $$
  A=\left(\begin{array}{cc} 1 & 0 \\ x_1 & 1 \end{array}\right), \
  B= \left(\begin{array}{cc} 1 & 0 \\ x_2 & 1 \end{array}\right), \
  C= \left(\begin{array}{cc} 1 & 1 \\ 0 &1  \end{array}\right), \
  D= \left(\begin{array}{cc} 1 & x_4\\ 0 &1 \end{array}\right),
  $$
where $x_1,x_2,x_4\neq 0$. Thus $e({\mathcal I}_6)=16(q-1)^3=16q^3-48q^2+48q-16$.
\end{itemize}

Therefore
 \begin{eqnarray*}
 e({\mathcal I})  &=& \sum_{j=1}^6 e({\mathcal I}_j) =
 q^5 + 7 q^4 + 10 q^3 -14 q^2 -7q -1\, .
 \end{eqnarray*}

\subsection*{Other contributions}
 The remaining contributions are
 $$
  e({\mathcal J}) = (e(Y_1)+e(Y_3)+e(Y_4))/e(\PGL(2,\CC)) =q^6-q^5+9q^4-10q^3+9q^2+7q+1\, .
 $$

\medskip

Adding up the previous E-polynomials, we get
 \begin{eqnarray*}
 e(\cM_{\Id}) &=& e({\mathcal R})+ e({\mathcal I}) + e({\mathcal J})
  = \, q^6+17q^4 + q^2+1\, .
 \end{eqnarray*}

\section{E-polynomial of the twisted $\SL(2,\CC)$-character variety for genus $g=2$} \label{sec:twisted}

In this section, we consider the {\em twisted\/} $\SL(2,\CC)$-character variety, i.e.,
 $$
 \cM_{-\Id}=W/\PGL(2,\CC),
 $$
where $W=F^{-1}(-\Id)$, with the definition of $F$ in Section \ref{sec:g=2}. Note that
now the GIT quotient is a geometric quotient, since there are no reducibles.

We have
 $$
 W:=\{(A,B,C,D)\in \SL(2,\CC)^{4}\, | \, [A,B]=-[D,C]\}
 $$
We stratify 
\begin{itemize}
\item $W_0:=\{(A,B,C,D)\, |\, [A,B]=-[D,C]=\xi=\Id\}$,
\item $W_1:=\{(A,B,C,D)\,|\, [A,B]=-[D,C]=\xi=-\Id\}$, 
\item $W_2:=\left\{(A,B,C,D)\,|\, [A,B]=-[D,C]=\xi\sim\small{\left(
    \begin{array}{cc}
      1 & 1\\
      0 & 1\\
    \end{array}
  \right)}\right\}$, 
  \item $W_3:=\left\{(A,B,C,D)\, |\, [A,B]=-[D,C]=\xi\sim\small{\left(
    \begin{array}{cc}
      -1 & 1\\
      0 & -1\\
    \end{array}
  \right)}\right\}$, 
\item $W_4:=\left\{(A,B,C,D) \,|\, [A,B]=-[D,C]=\xi\sim\small{\left(
    \begin{array}{cc}
      \lambda & 0\\
      0 & \lambda^{-1}\\
    \end{array}
  \right)},\, \lambda\neq 0,\pm 1\right\}$.
\end{itemize}

We also introduce the sets
\begin{itemize}
\item $\overline{W}_2:=\left\{(A,B,C,D)\,|\, [A,B]=-[D,C]=\xi=\small{\left(
    \begin{array}{cc}
      1 & 1\\
      0 & 1\\
    \end{array}
  \right)}\right\}$, 
  \item $\overline{W}_3:=\left\{(A,B,C,D) \,|\, [A,B]=-[D,C]=\xi=\small{\left(
    \begin{array}{cc}
      -1 & 1\\
      0 & -1\\
    \end{array}
  \right)}\right\}$, 
\item $\overline{W}_4:=\left\{(A,B,C,D)\, |\, [A,B]=-[D,C]=\xi=\small{\left(
    \begin{array}{cc}
      \lambda & 0\\
      0 & \lambda^{-1}\\
    \end{array}
  \right)},\, \lambda\neq 0,\pm 1\right\}$,
\item $
 \overline{W}_{4,\lambda}:=\left\{(A,B,C,D) \, |\, [A,B]=-[D,C]=\small{\left(
    \begin{array}{cc}
      \lambda & 0\\
      0 & \lambda^{-1}\\
    \end{array}
  \right)} \right\}
  $,
for $\lambda\neq 0,\pm 1$.
\end{itemize}

It is clear that $W_0=X_0\times X_1$, $W_1=X_1\times X_0$, so
 $$
 e(W_0)=e(W_1)=e(X_0)e(X_1)=  q^7 + 4 q^6 - 2 q^5 - 8 q^4 + q^3  + 4 q^2 \, .
 $$
Also $\overline{W}_2=\overline{X}_2 \times \overline{X}{}_3$ and
$\overline{W}_3=\overline{X}_3 \times \overline{X}{}_2$. So
 $$
 e(W_2)=e(W_3)=e(\GL(2,\CC)/U) e(\overline{X}_2) e(\overline{X}_3) =
 q^8 + q^7 - 10 q^6  - 10 q^5 + 9 q^4 + 9 q^3\, .
 $$

For the last stratum, note that
 \begin{equation}\label{38bis}
\overline{W}_{4,\lambda}=\overline{X}_{4,\lambda}\times
 \overline{X}{}_{4,-\lambda}\, .
 \end{equation}
Consider the involution $\tau:B\to B$, $\tau(\lambda)=-\lambda$. Then we have the
fibration $\tau^*\overline{X}_4 \to B$. Since $\tau$ commutes with the $\ZZ_2$-action, we
have a fibration
  $$
  \tau^*\overline{X}_4/\ZZ_2 \to B'=\CC\setminus \{\pm 2\}\, ,
  $$
whose fibre over $\lambda$ is $\overline{X}{}_{4,-\lambda}$.
Writing $R(\overline{X}_4/\ZZ_2)=aT +b S_2+cS_{-2}+d S_0$, where
 $$
 a=q^3,\ b=-3q, \ c=3q^2,\ d=-1\, ,
 $$
we have that
  $$
  R(\tau^*\overline{X}_4/\ZZ_2)=\tau^* R(\overline{X}_4/\ZZ_2)=aT +c S_2+bS_{-2}+d S_0\, ,
 $$
since $\tau^*T=T$, $\tau^*S_i=S_{-i}$, $i=0,\pm2$.

Now
\begin{eqnarray} \label{39bis}
 R(\overline{W}_4/\ZZ_2)&=&R(\overline{X}_4/\ZZ_2)\otimes R(\tau^*\overline{X}_4/\ZZ_2) \nonumber \\
 &=& (a^2+2 bc +d^2)T +(ac+ab+bd+cd )S_2+ ( ab+ac+bd+cd ) S_{-2} \nonumber \\
 && + (2ad+b^2+c^2 )S_0 \nonumber \\
 &=& (q^6-18q^3+1) T + (3q^5-3q^4-3q^2+3q) S_{2} \\
 && + (3q^5-3q^4-3q^2+3q) S_{-2} + (9q^4-2q^3+9q^2) S_{0} \nonumber \, .
\end{eqnarray}
Using Remark \ref{rmk:the-e-continuation}, we get
 $$
e(W_4)= q^9- 2 q^8- 7 q^7- 18 q^6 + 24 q^5+ 28 q^4 - 17 q^3 - 8 q^2-q  .
 $$

Adding up all contributions,
 $$
 e(W)= q^9- 3 q^7 - 30 q^6+ 30 q^4+ 3 q^3 -q  .
 $$
Also
 $$
 e(\cM_{-\Id})= e(W)/e(\PGL(2,\CC)) =  q^6 - 2 q^4 - 30 q^3 - 2 q^2 +1.
 $$
This agrees with the result of Mereb \cite[Theorem 4.6]{mereb:2010}.

The moduli space $\cM_{-\Id}$ is homeomorphic to the moduli space of rank $2$,
fixed determinant and odd degree Higgs bundles on a curve $C$ of genus $g=2$.
By \cite[p.\ 99]{hitchin:1987}, the Poincar\'e polynomial of this space is
 $$
 P_t(\cM_{-\Id})=2 t^6+34 t^5+2 t^4+4 t^3+t^2+1 \, .
 $$
This space is smooth and of complex dimension $6$ (it is homotopy equivalent to
the nilpotent cone, which is Lagrangian and of half the dimension). By
Poincar\'e duality, we have
 $$
 P_t^c(\cM_{-\Id})= t^{12} + t^{10} + 4 t^9 + 2 t^8 + 34 t^7 + 2 t^6\, .
 $$
The Euler-Poincar\'e characteristic of $\cM_{-\Id}$ is thus $-32$.
The variety $\cM_{-\Id}$ is smooth and of balanced type. So $h_c^{k,p,q}=0$ if $p\neq q$ or
$p+q>k$. However, this is not enough information to get the Hodge numbers of $\cM_{-\Id}$.

The moduli space of rank $2$ odd degree Higgs bundle with fixed determinant, $\cH_{-\Id}$,
over a curve of genus $2$ has E-polynomial
  $$
  e(\cH_{-\Id})=u^6v^6+u^5v^5-2u^4v^5-2u^5v^4+2u^4v^4 - 17 u^3v^4-17u^4v^3+2 u^3v^3\, ,
  $$
which is not of balanced type. This polynomial basically follows from the arguments
in \cite{hitchin:1987}.

\section{E-polynomial of $\SL(2,\CC)$-character varieties for genus $g=2$ and
diagonalizable holonomy} \label{sec:xi0}

Let $\xi_0=\left(
    \begin{array}{cc}
      \lambda_0 & 0\\
      0 & \lambda_0^{-1}\\
    \end{array}
  \right)$, where $\lambda_0\neq 0,\pm 1$, and consider
 $$
 \cM_{\lambda_0}=Z/\CC^* ,
 $$
where $Z=F^{-1}(\xi_0)$. Thus
 $$
 Z:=\{(A,B,C,D)\in \SL(2,\CC)^{4}\, | \, [A,B]=\xi_0 \, [D,C]\} \, .
 $$

Let
 $$
 \eta=[D,C]\, , \ \ \delta=[A,B]\, ,
 $$
and write $t_1:=\tr \eta$, $t_2:=\tr \delta$.
We put
 $$
 \eta=\left(\begin{array}{cc} a & b\\ c & d \end{array}\right)\, ,
 $$
 and
 $$
 \delta =  \xi_0 \eta =\left(\begin{array}{cc} \lambda_0 a & \lambda_0 b\\
 \lambda_0^{-1} c & \lambda_0^{-1}d \end{array}\right)\, .
 $$

We have the equations:
 \begin{equation}\label{eqn:t1t2}
 \left\{ \begin{array}{l} t_1=a+d \\ t_2 = \lambda_0 a + \lambda_0^{-1} d \end{array} \right.
 \qquad \Longrightarrow \qquad
 \left\{ \begin{array}{l} a={\displaystyle \frac{-\lambda_0^{-1} t_1+t_2}{\lambda_0-\lambda_0^{-1}}}\\
 d={\displaystyle \frac{\lambda_0 t_1-t_2}{\lambda_0-\lambda_0^{-1}}} \end{array} \right.
 \end{equation}

We shall pay special attention to the equation $ad=1$. This becomes
$$(-\lambda_0^{-1} t_1+t_2)(\lambda_0 t_1-t_2)=(\lambda_0-\lambda_0^{-1})^2,$$ or equivalently:
 \begin{equation} \label{eqn:t1t2-hyperbola}
 -t_1^2-t_2^2+(\lambda_0+\lambda_0^{-1})t_1t_2=(\lambda_0-\lambda_0^{-1})^2\, .
 \end{equation}
This is a hyperbola, that we shall denote by $H$ (see Figure \ref{fig:conic37}).

We stratify $Z$ according to the preimages of different sets over the plane $(t_1,t_2)$, where
we consider the map $Z\to \CC^2$, $(A,B,C,D)\mapsto (\tr [D,C],\tr [A,B])$. We distinguish the lines
$t_1=\pm 2, \, t_2=\pm 2$ and the curve $ad=1$, and the points of intersection of these curves.

\begin{figure}
  \centering
    \includegraphics[scale=0.6]{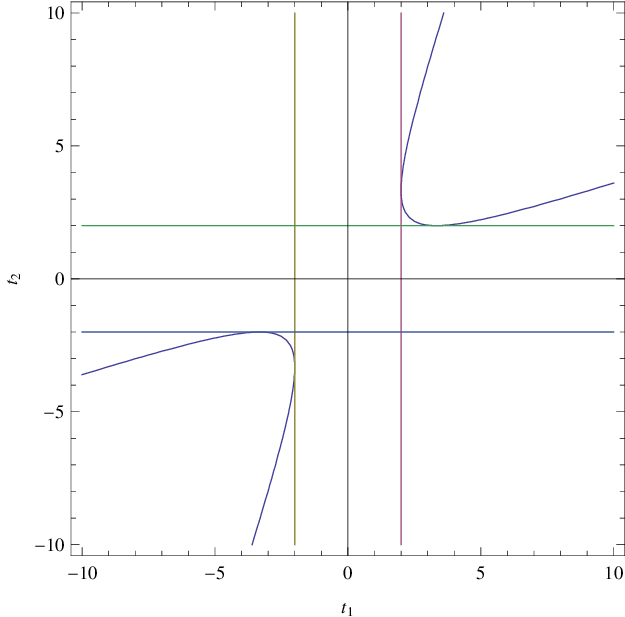}
  \caption{Distinguished lines in the base of the fibration for $Z$ when $\lambda_0=1/3$.}
  \label{fig:conic37}
\end{figure}

We consider the following sets:

\begin{itemize}
\item {$Z_1$},
corresponding to the cases $t_1=\pm 2$, $t_2=\pm 2$. First observe that (\ref{eqn:t1t2}) gives us
the values of $a,d$.
In this situation $ad\neq 1$ (if $ad=1$, then $\pm2=a+d$ would imply that $a=d=\pm 1$, and $\pm2=\lambda_0 a+\lambda_0^{-1} d$
would imply that $\lambda_0 a=\lambda_0^{-1}d=\pm 1$, giving $\lambda_0=\pm1$, which is not the case).

The condition $ad-bc=1$ then gives
$bc\neq 0$, therefore $b\in \CC^*$ and $c$ is uniquely determined. Moreover, for each choice
of $b$, we have that the matrices $\eta$ and $\delta$ are fixed, and both are of Jordan type
$\left(\begin{array}{cc}\pm 1 &1 \\
0&\pm 1\end{array}\right)$. Taking into account the four possibilities, this yields the polynomial
 \begin{eqnarray*}
 e(Z_1) &=&  (q-1) e(\overline{X}_2)^2 +(q-1) e(\overline{X}_3)^2 + 2 (q-1) e(\overline{X}_2) e(\overline{X}_3)
  \\ &=&(q-1) (e(\overline{X}_2)+e(\overline{X}_3))^2= 4 q^7 - 15 q^5 + 5 q^4 + 15 q^3 -9 q^2 \, .
 \end{eqnarray*}

\item {$Z_2$}, corresponding to the cases $t_1=2$, $t_2=\lambda_0+\lambda_0^{-1}$,
and $t_2=2$, $t_1=\lambda_0+\lambda_0^{-1}$. We shall do the first case (the second is analogous), and
then multiply the result by $2$. So consider the point $t_1=2$, $t_2=\lambda_0+\lambda_0^{-1}$,
which is the intersection of the line $t_1=2$ with $ad=1$. Thus $bc=0$. There are three possibilities, giving
rise to three sub-strata:
 \begin{enumerate}
 \item $Z^{1}_{2}$, given by $b=c=0$. Then $\eta=\Id$ and $\delta=\xi_0$. So
 $(D,C)\in X_0$, $(A,B) \in \overline{X}_{4,\lambda_0}$. Therefore $Z_2^1=X_0\times \overline{X}_{4,\lambda_0}$, and
 $$
 e(Z_2^1)= e(X_0)e(\overline{X}_{4,\lambda_0}).
 $$
 \item $Z^{2}_{2}$, given by $b\neq 0$, $c=0$.
 Then $b\in \CC^*$, and $\eta$ is of Jordan type $J_+$, and $\delta$ is diagonalizable conjugate to $\xi_0$.
 So
 $$
 e(Z_2^2)= (q-1) e(\overline{X}_2) e(\overline{X}_{4,\lambda_0}).
 $$
 \item $Z^{3}_{2}$, given by $b=0$, $c\neq 0$. It gives a similar contribution $e(Z_2^3)=e(Z_2^2)$.
 \end{enumerate}
Altogether, and doubling the contribution as we mentioned above, we obtain
 \begin{eqnarray*}
 e(Z_2) &= & 2(e(X_0)+2(q-1)e(\overline{X}_2)) e(\overline{X}_{4,\lambda}) \\ &=&  6q^7+14q^6-36q^5-8q^4+34q^3-6q^2-4q\, .
 \end{eqnarray*}

\item {$Z_3$}, corresponding to the cases $t_1=-2$, $t_2=-\lambda_0-\lambda_0^{-1}$,
and $t_2=-2$, $t_1=-\lambda_0-\lambda_0^{-1}$. Again consider the first case, $t_1=-2$, $t_2=-\lambda_0-\lambda_0^{-1}$.
The analysis is analogous to that for $Z_2$, just changing the sign of the diagonal entries of
$\eta$ to $-1$. So this produces
 \begin{eqnarray*}
 e(Z_3) & =& 2(e(X_1)+2(q-1)e(\overline{X}_3)) e(\overline{X}_{4,\lambda}) \\
 &=& 4 q^7 + 22 q^6 + 6 q^5 - 72 q^4 + 20 q^3 + 18 q^2+ 2 q\, .
 \end{eqnarray*}

\item {$Z_4$}, corresponding to the cases $t_1=2$, $t_2\neq \pm 2, \lambda_0+\lambda_0^{-1}$,
and $t_2=2$, $t_1\neq \pm 2, \lambda_0+\lambda_0^{-1}$. We focus on the first case, the second one being analogous.
So $(t_1,t_2)$ move in a (three-punctured)
line $L=\{(t_1,t_2)\, | \, t_1=2, t_2\neq \pm 2, \lambda_0+\lambda_0^{-1}\}$.

First, we cannot have $\eta=\Id$, since in that case $\delta=\xi_0$ and $t_2=\lambda_0+\lambda_0^{-1}$.
So $\eta$ is of Jordan form $J_+$. Also $ad\neq 1$, since $(t_1,t_2)\notin H$.
So the numbers $a,b,c,d$ are determined, up to a choice $b\in \CC^*$.

We fix a slice by setting $b=1$. Then
$\eta=\left(\begin{array}{cc} a & 1 \\ ad-1 & d \end{array}\right)$, where
  $$
  a=\frac{-2\lambda_0^{-1}+t_2}{\lambda_0-\lambda_0^{-1}}, \
  d=\frac{2\lambda_0 - t_2}{\lambda_0-\lambda_0^{-1}}. \,
  $$
The matrix $\eta$ has eigenvector $(1,1-a)$. Conjugating by
$\left(\begin{array}{cc} 0 & 1 \\ 1 & 1-a \end{array}\right)$, we transform
$\eta$ into $J_+$, and the set of matrices $\{(C,D) \, | \, [D,C] =\eta\}$ into $\overline{X}_2$.
This gives a trivial family over the line $L$. On the other hand,
  $$
  \delta= \left(\begin{array}{cc} \frac{-2+ \lambda_0 t_2}{\lambda_0-\lambda_0^{-1}} & \lambda_0 \\
  * & \frac{2- \lambda_0^{-1} t_2}{\lambda_0-\lambda_0^{-1}}  \end{array} \right) \, .
  $$
To understand the family of matrices $(A,B)$ with $[A,B]=\delta$ for $t_2\in L$, we need to take
a double cover given by $t_2=\mu+\mu^{-1}$, $\mu\in \CC\setminus \{0,\pm 1,\lambda_0,\lambda_0^{-1}\}$.
Then we can conjugate the matrix $\delta$ by
 $$
 P_\mu:=\left(\begin{array}{cc} \lambda_0 & \lambda_0 \\ \mu + \frac{2-\lambda_0t_2}{\lambda_0-\lambda_0^{-1}} &
\mu^{-1} + \frac{2-\lambda_0t_2}{\lambda_0-\lambda_0^{-1}} \end{array}\right)
 $$
to take it into diagonal form $\left( \begin{array}{cc} \mu & 0\\ 0 &\mu^{-1}\end{array}\right)$. This makes
the space isomorphic to $\overline{X}_4$. To recover the original space, we have to quotient by $\mu\mapsto \mu^{-1}$.
Note that $P_{\mu^{-1}}=P_\mu \left( \begin{array}{cc} 0 & 1\\ 1 &0 \end{array}\right)$, so the
corresponding action on $\overline{X}_4$ is conjugation by $\left( \begin{array}{cc} 0 & 1\\ 1 &0 \end{array}\right)$.
To summarize, the resulting space is $(\overline{X}_4/\ZZ_2) \setminus \overline{X}_{4,\lambda_0}$.
This gives the polynomial
 \begin{eqnarray*}
 e(Z_4)&=& 2(q-1) e(\overline{X}_2) (e(\overline{X}_4/\ZZ_2)- e(\overline{X}_{4,\lambda_0}))\\
  &=&  2 q^8 - 12 q^7 + 4 q^6 + 60 q^5 - 38 q^4 - 60 q^3 + 32 q^2  +12 q\, .
 \end{eqnarray*}

\item {$Z_5$}, corresponding to the cases $t_1=-2$, $t_2\neq \pm 2, -\lambda_0-\lambda_0^{-1}$,
and $t_2=-2$, $t_1\neq \pm 2, -\lambda_0-\lambda_0^{-1}$. Analogously to the previous case, we obtain
 \begin{eqnarray*}
 e(Z_5) &=& 2(q-1) e(\overline{X}_3) (e(\overline{X}_4/\ZZ_2)- e(\overline{X}_{4,-\lambda_0})) \\
 &=&  2 q^8 - 2 q^7 - 30 q^6 + 6 q^5 + 64 q^4 - 28 q^3 -12 q^2\, .
 \end{eqnarray*}

\item {$Z_6$}, corresponding to the hyperbola $H$, i.e., those $(t_1,t_2)$ 
satisfying (\ref{eqn:t1t2-hyperbola}), but $t_1,t_2\neq \pm 2$. This
set is parameterized by $\CC^*\setminus \{\pm 1,\pm \lambda_0^{-1}\}$ under
 $$
 \mu \mapsto (\mu+\mu^{-1}, \lambda_0\mu+\lambda^{-1}_0\mu^{-1}).
 $$
The elements on $H$ satisfy $ad=1$, so $bc=0$. For each value of $\mu$, there are three possibilities:
 \begin{enumerate}
 \item $Z^{1}_{6}$, for $b=c=0$;
 \item $Z^{2}_{6}$, for $b\in \CC^*$, $c=0$;
 \item $Z^{3}_{6}$, for $b=0$, $c\in \CC^*$.
 \end{enumerate}

This gives a factor $2q-1$, that
multiplies the E-polynomial of the space which is a fibration over
$B=\CC^*\setminus \{\pm 1,\pm \lambda_0^{-1}\}$, and whose fibre over $\mu$ is
$\overline{X}_{4,\mu}\times \overline{X}_{4,\lambda_0\mu}$. Call this fibration $\overline{Z}_6 \to
B$, where $\overline{Z}_6$ is of balanced type by Proposition \ref{prop:pure-type}.

Consider the map $\tau:\CC^*\to \CC^*$, $\tau(\mu)=\lambda_0\mu$. Then $\overline{Z}_6 = \overline{X}_4 \x_B
\tau^* \overline{X}_4$. Using Theorem \ref{thm:R(overline)}, this fibration has monodromy
representation:
 \begin{eqnarray*}
 R(\overline{Z}_6) &=& R(\overline{X}_4)\otimes \tau^*R(\overline{X}_4) =
 R(\overline{X}_4)\otimes R(\overline{X}_4) \\ &=& ((q^3-1)T + (3q^2-3q)N)^2\\
 &=& ((q^3-1)^2+(3q^2-3q)^2)T + 2 (q^3-1)(3q^2-3q) N \\
 &=&(q^6 + 9 q^4  - 20 q^3 + 9 q^2 +1)T+(6 q^5 - 6 q^4 - 6 q^2 +6 q ) N\, .
 \end{eqnarray*}
Now we use Proposition \ref{prop:e-total-space} to calculate the E-polynomial. Write
$R(\overline{Z}_6)=a T + b N$. Then
 \begin{eqnarray*}
 e(Z_6) & =& (2q-1) e(\overline{Z}_6) \\ & =& (2q-1) ((q-5) a -4 b) \\ &=&
 (2q-1)(q^7 -5q^6-15q^5-41q^4+109 q^3-21 q^2-23 q- 5) \\ &=&
 2 q^8 - 11 q^7- 25 q^6 - 67 q^5  + 259 q^4 - 151 q^3 - 25 q^2 + 13 q +5\, .
 \end{eqnarray*}

\item $Z_7$, corresponding to the open part of the plane $(t_1,t_2)$, that is,
$(t_1,t_2) \notin H$ and $t_1,t_2\neq \pm 2$.
The equations (\ref{eqn:t1t2}) determine $a,b,c,d$ up to $b\in \CC^*$. This gives a space whose fibre
over $(t_1,t_2)$ is isomorphic to
 \begin{equation}\label{eqn:lll}
  \overline{X}_{4,\mu_1}\times \overline{X}_{4,\mu_2},
 \end{equation}
where $t_i=\mu_i+\mu_i^{-1}$, $i=1,2$. Ignoring the condition $ad=1$, this contributes
$e(\overline{X}_4/\ZZ_2)^2$.
The term to be subtracted is the E-polynomial of the fibration with fibres
(\ref{eqn:lll}) over $\{(t_1,t_2)\in H \ | \ t_1\ne\pm 2,t_2\ne \pm 2\}$.
This is the same as $e(\overline{Z}_6)$ above. So
 \begin{eqnarray*}
 e(Z_7)&=& (q-1)\big( e(\overline{X}_4/\ZZ_2)^2 - e(\overline{Z}_6) \big)\\
 &=&q^9 - 6 q^8 + 8 q^7 + 30 q^6 + 7 q^5 - 171 q^4 + 155 q^3 - 5 q^2 - 23 q -6\, .
 \end{eqnarray*}

\end{itemize}

Adding up all contributions, we get
 \begin{eqnarray*}
e(Z) 
&=& q^9 - 3 q^7 + 15 q^6 - 39 q^5 + 39 q^4 - 15 q^3 + 3 q^2 -1\, .
 \end{eqnarray*}

Dividing by the stabilizer, $\stab(\xi)=\CC^* \subset \PGL(2,\CC)$, we get
 $$
 e(\cM_{\lambda_{0}})=q^8  + q^7 - 2 q^6 + 13 q^5 - 26 q^4 + 13 q^3 - 2 q^2 + q + 1 \, .
 $$

Note that this polynomial is palindromic, as suggested to us by T.\ Hausel.

The space $\cM_{\lambda_0}$ is homeomorphic to the moduli space of parabolic Higgs bundles
$\cH_{\lambda_0}$. By \cite{by}, the Poincar\'e polynomial is
 \begin{eqnarray*}
 P_t(\cM_{\lambda_0}) &=& \frac{(1+t^3)^4+ t^7(1+t)^4(3t-4)}{(1-t^2)^2} + 15t^6(1+t)^2 \\
  &=& 19t^8 + 38 t^7 + 25 t^6 + 8 t^5+ 3 t^4  +4t^3+   2t^2+1\, .
 \end{eqnarray*}
As $\cM_{\lambda_0}$ is smooth and of complex dimension $8$, Poincar\'e duality gives
 $$
 P_t^c(\cM_{\lambda_0})= t^{16} + 2t^{14} +4 t^{13} +3 t^{12}+ 8 t^{11}+25t^{10}+ 38t^9+19 t^8
 \, . 
 $$
Note that the Euler-Poincar\'e characteristic of $\cM_{\lambda_0}$ is $0$.

%
%
%
%

\section{E-polynomial of $\SL(2,\CC)$-character varieties for genus $g=2$ and Jordan form $J_+$} \label{sec:J+}

Let $J_+=\left(
    \begin{array}{cc}
      1 & 1\\
      0 & 1 \\
    \end{array}
  \right)$. Recall that $U\subset \GL(2,\CC)$ is the stabilizer of $J_+$, so that
 $$
 \cM_{J_+}=Z/ (U/\CC^*) \,  ,
 $$
where 
 $$
 Z:=F^{-1}(J_+)= \{(A,B,C,D)\in \SL(2,\CC)^{4}\ | \ [A,B]=J_+ \, [D,C]\} \, .
 $$
(The notation $Z$ should not be confused with the one used in section \ref{sec:xi0}.)

Let
 $$
 \eta=[D,C]=\left(\begin{array}{cc} a & b\\ c & d \end{array}\right)\, ,
 $$
 and
 $$
 \delta =[A,B]=  J_+ \eta =\left(\begin{array}{cc} a+c & b+d\\
 c & d \end{array}\right)\, .
 $$

Set $t_1=\tr \eta=a+d$, and $t_2=\tr \delta=a+c+d$, so $c=t_2-t_1$.

\begin{figure}
  \centering
    \includegraphics[scale=0.5]{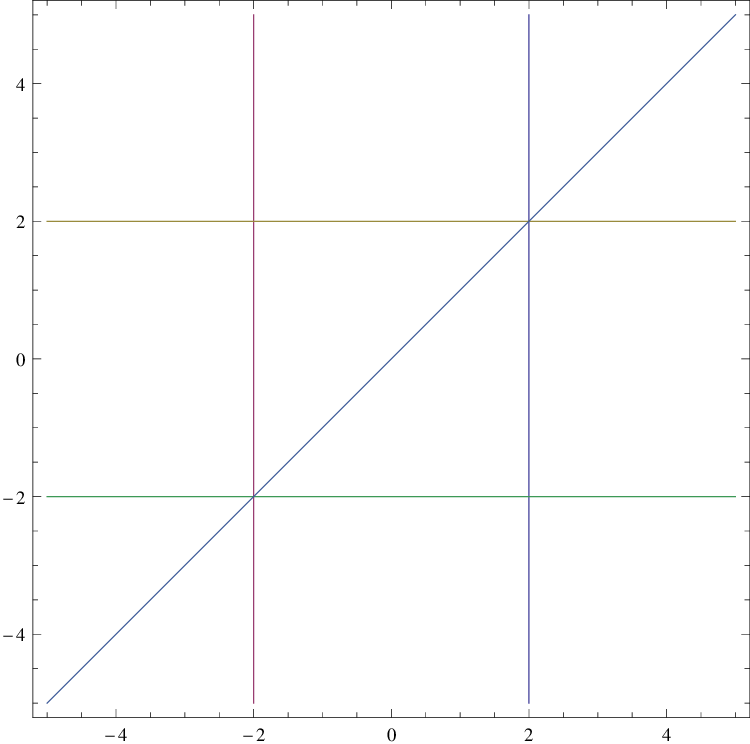}
  \caption{Distinguished lines in the base of the fibration for $Z$.}
  \label{fig:lines}
\end{figure}

We stratify $Z$ into the following sets corresponding to the various possibilities
for $(t_1,t_2)$ (see Figure \ref{fig:lines}).

\begin{itemize}
 \item {$Z_1$} is the subset of $Z$ lying over the point $(t_1,t_2)=(2,2)$. Observe that in this case
$c=0$. So  $a=d=1$.
There is still a free parameter $b\in \CC$.
For $b\neq 0,-1$, we have that both $\eta,\delta$ are of Jordan form $J_+$. This contributes a summand
to the  E-polynomial equal to $(q-2) e(\overline{X}_2)^2$. For $b= 0,-1$, one of $\eta,\delta$
is of Jordan form and the other
is diagonal (i.e., equal to $\Id$), thus contributing $2 e(\overline{X}_2) e(\overline{X}_0)$. All together,
we have
 \begin{eqnarray*}
 e(Z_1) &=& (q-2) e(\overline{X}_2)^2+ 2 e(\overline{X}_2) e(\overline{X}_0) \\
  &=& 
  3q^7-2q^6-18q^5-12q^4+7q^3+6q^2 \, .
 \end{eqnarray*}

 \item {$Z_2$}, the stratum corresponding to $t_1=t_2= -2$. This is completely analogous to the previous case,
giving the E-polynomial
 \begin{eqnarray*}
 e(Z_2) &=& (q-2) e(\overline{X}_3)^2+ 2 e(\overline{X}_3) e(\overline{X}_1) \\
  &=& 
  q^7+6q^6+3q^5-20q^4-6q^3 \, .
 \end{eqnarray*}

 \item {$Z_3$}, corresponding to $t_1=-t_2= \pm 2$. These are two points, so we
consider the first case $(t_1,t_2)=(2,-2)$, and double the final contribution.
Recall that $c=t_2-t_1$, so $c\neq 0$. Then conjugating by a matrix $\left(
    \begin{array}{cc}
      1 & \lambda \\
      0 & 1 \\
    \end{array}
  \right)$,
with $\lambda=d/c$, we can arrange that $d=0$. Therefore $b=-1/c$,
and $a=t_1$. This fixes $a,b,c,d$. Note now that $\eta$ is of Jordan type $J_+$ and
$\delta$ is of Jordan type $J_-$. So
 $$
  e(Z_3)=2q \, e(\overline{X}_2) e(\overline{X}_3) = 
 2q^7+2q^6-18q^5-18q^4 \, .
  $$
Here the factor $q$ is given by the parameter $\lambda \in \CC$.

 \item {$Z_4$}, corresponding to four strata given by the cases $t_1=2$, $t_2\neq \pm 2$;
$t_1=-2$, $t_2\neq \pm 2$; $t_2=2$, $t_1\neq \pm 2$; and $t_2=-2$, $t_1\neq \pm 2$. Let us
focus on the line $L=\{(t_1,t_2)\ | \ t_1=2, t_2\neq \pm 2\}$.
Then $c=t_2-2 \neq 0$, so $\eta$ is of Jordan form $J_+$.
Conjugating as before, we can assume $d=0$, $b=-1/c$ and $a=t_1=2$.
Therefore
  $$
   \eta= \left( \begin{array}{cc} 2 & \frac1{2-t_2} \\ t_2-2 & 0 \end{array} \right), \
   \delta= \left( \begin{array}{cc} t_2 & \frac1{2-t_2} \\ t_2-2 & 0 \end{array} \right).
  $$

Now, conjugating by $\left( \begin{array}{cc} 1 & 1\\ t_2-2 & 0 \end{array} \right)$, we put
$\eta$ into Jordan form $J_+$. Hence the fibration given by the matrices $(C,D)$ over $L$
has fibre $\overline{X}_2$ and is trivial.

To study the fibration given by the matrices $(A,B)$ over $L$, we introduce the
variable $t_2=\mu+\mu^{-1}$, $\mu \in \CC\setminus \{0,\pm 1\}$. This gives a double
cover of the space we are interested in. The fibre over $\mu$ is $\overline{X}_{4,\mu}$,
and consists of pairs of matrices $(A,B)$ satisfying
 $$
 [A,B]=\delta = \left( \begin{array}{cc} \mu+\mu^{-1} & \frac1{2-\mu-\mu^{-1}} \\ \mu+\mu^{-1}-2
 & 0 \end{array} \right).
 $$
We conjugate by the matrix $P_\mu:=\left( \begin{array}{cc} \mu+\mu^{-1} -2& \mu+\mu^{-1} -2 \\
-\mu^{-1} & -\mu \end{array} \right)$, to put $\delta$ into standard form
$\left( \begin{array}{cc} \mu& 0 \\ 0 & \mu^{-1} \end{array} \right)$.
The quotient by $\mu\mapsto \mu^{-1}$ corresponds to conjugation by $\left( \begin{array}{cc} 0 & 1 \\
1&0 \end{array} \right)$, i.e., to the standard $\ZZ_2$-action on $\overline{X}_4$.

So, the substratum of $Z_4$ that we are studying is isomorphic to
$\CC \x \overline{X}_2 \x (\overline{X}_4/\ZZ_2)$. Analogously, the substratum corresponding to
$t_1=-2$, $t_2\neq \pm 2$ is isomorphic $\CC \x \overline{X}_3 \x (\overline{X}_4/\ZZ_2)$.
The remaining two substrata are copies of the previous two. Finally, we get
 \begin{eqnarray*}
 e(Z_4)&=&2q\,
 \left( e(\overline{X}_2)+e(\overline{X}_3) \right) e(\overline{X}_{4}/\ZZ_2) \\
 &=&4q^8-6q^7-22q^6+18q^5+28q^4-16q^3-6q^2\, .
 \end{eqnarray*}

 \item {$Z_{5}$}, corresponding to $t_1=t_2\neq \pm 2$. First, we have
$c=t_2-t_1=0$, so $d=a^{-1}$ and $b\in\CC$. Thus 
 $$
  \eta=\left(\begin{array}{cc}
      a & b\\
      0 & a^{-1}
    \end{array}\right), \ \delta=\left(\begin{array}{cc}
      a & b+a^{-1}\\
      0 & a^{-1}
    \end{array}\right),
    $$
where $t_1=t_2=a+a^{-1}$. The fibre over $(t_1,t_1)$ consists of two disjoint
sets, one is isomorphic to $\overline{X}_{4,a}\x \overline{X}_{4,a}$, the
other isomorphic to $\overline{X}_{4,a^{-1}}\x \overline{X}_{4,a^{-1}}$. Taking the
double cover $a \mapsto t_1=a+a^{-1}$, we get a fibration over the line
$B=\CC\setminus \{0,\pm 1\}$ whose fibres are $\overline{X}_{4,a}\x \overline{X}_{4,a}$.
We call $\overline{Z}_5 \to B$ this fibration, and note that by Proposition \ref{prop:pure-type} $\overline{Z}_5$ is of balanced type. Its monodromy representation is
 \begin{eqnarray*}
  R(\overline{Z}_5) &=& R(\overline{X}_4)\otimes R(\overline{X}_4) \\
  &=& (aT+bN) \otimes (aT+bN) = (a^2+b^2) T + 2ab N \\
  &=& (q^6+9 q^4 -20q^3 + 9 q^2 +1) T + (6q^5 - 6 q^4 -6q^2 + 6q ) N\, ,
 \end{eqnarray*}
where $R(\overline{X}_4)=aT+bN$, $a= q^3-1$, $b=3q^2-3q$ (by Theorem \ref{thm:R(overline)}).
Therefore, Proposition \ref{prop:e-total-space} says that the E-polynomial is
 \begin{eqnarray*}
 e(\overline{Z}_5) &=& (q-3)(q^6+9 q^4 -20q^3 + 9 q^2 +1) - 2(6q^5 - 6 q^4 -6q^2 + 6q ) \\
 &=&  q^7-3q^6-3q^5-35q^4+69 q^3-15 q^2-11q -3 \, .
 \end{eqnarray*}
and $e(Z_5)=q \, e(\overline{Z}_5) = q^8-3q^7-3q^6-35q^5+69 q^4-15 q^3-11q^2 -3q$.

 \item {$Z_6$}, corresponding to the open stratum, characterised by the equations
$t_1,t_2\neq \pm 2$, and $t_1\neq t_2$.
Then $c=t_2-t_1 \neq 0$. Conjugating by a matrix $\left(
    \begin{array}{cc}
      1 & \lambda \\
      0 & 1 \\
    \end{array}
  \right)$,
with $\lambda=d/c$, we can arrange that $d=0$. Therefore $b=-1/c$,
and $a=t_1$. So the fibre over
$(t_1,t_2)$ has
  $$
   \eta=\left(\begin{array}{cc}
      t_1 & -1/(t_2-t_1)\\
      t_2-t_1 & 0
    \end{array}\right), \ \delta=\left(\begin{array}{cc}
      t_2 & -1/(t_2-t_1)\\
      t_2-t_1 & 0
    \end{array}\right).
    $$

Each fibre is isomorphic to $\CC\times \overline{X}_{4,\lambda_1}\times \overline{X}_{4,\lambda_2}$,
where $t_i=\lambda_i+\lambda_i^{-1}$. Working as before, we see that, ignoring the condition
$t_1\neq t_2$, the total
space is isomorphic to $\CC\times (\overline{X}_{4}/\ZZ_2) \times (\overline{X}_{4}/\ZZ_2)$.
We have to subtract the contribution corresponding to the space $Z_6'$ parameterized by
$(t_1,t_1)$, $t_1\neq \pm2$, and with fibers $\CC\times \overline{X}_{4,\lambda_1}\times
\overline{X}_{4,\lambda_1}$. So $Z_6'= \CC \x (\overline{Y}_4/\ZZ_2)$, where $\overline{Y}_4$
is described in section \ref{sec:g=2} (see equation (\ref{eqn:Y4})). The
Hodge monodromy representation $R(\overline{Y}_4/\ZZ_2)$ is given by equation (\ref{eqn:R(Y4/Z2)}). Note that by Proposition \ref{prop:pure-type} $\overline{Y}_4/\ZZ_2$ is of balanced type. Hence using Proposition \ref{prop:e-total-space}, we have that, 
 $$
 e(\overline{Y}_4/\ZZ_2)=q^7-2 q^6+3 q^5-12 q^4+29 q^3-12 q^2-5 q-2\, ,
 $$
and then
\begin{eqnarray*}
 e(Z_6) &=& q(e(\overline{X}_4/\ZZ_2)^2- e(\overline{Y}_4/\ZZ_2)) \\
 &=&q^9-5 q^8+15 q^6+11 q^5-51 q^4+15 q^3+11 q^2+3 q \, .
\end{eqnarray*}
\end{itemize}

Adding all contributions, we get
 $$
 e(Z)= q^9  -3 q^7 -4 q^6 - 39 q^5-4 q^4 -15 q^3 \, ,
 $$
and dividing by the stabilizer, $\stab (J_+)=U/\CC^* \subset \PGL(2,\CC)$,
 $$
 e(\cM_{J_+}) = q^8  -3 q^6 -4 q^5 - 39 q^4-4 q^3 -15 q^2 \, .
 $$

\begin{rmk}
 Note that $U/\CC^*$ acts freely on $Z$. Certainly, if $(A,B,C,D)\in \SL(2,\CC)^4$ is acted on
trivially by $U$, then all matrices $A,B,C,D \in U$, and hence $[A,B]=[C,D]=\Id$.
\end{rmk}

\section{E-polynomial of $\SL(2,\CC)$-character varieties for genus $g=2$ and Jordan form $J_-$} \label{sec:J-}

Let $J_-=\left(
    \begin{array}{cc}
      -1 & 1\\
      0 & -1 \\
    \end{array}
  \right)$, and consider
 $$
 \cM_{J_-}=Z/ (U/\CC^*) ,
 $$
where
 $$
 Z:=F^{-1}(J_-)= \{(A,B,C,D)\in \SL(2,\CC)^{4}\ | \ [A,B]=J_- \, [D,C]\} \, .
 $$
(The notation $Z$ should not be confused with the ones used in section \ref{sec:xi0} or
section \ref{sec:J+}.)

Let
 $$
 \eta=[D,C]=\left(\begin{array}{cc} a & b\\ c & d \end{array}\right)\, ,
 $$
and
 $$
 \delta =[A,B]=  J_- \, \eta =\left(\begin{array}{cc} -a+c & -b+d\\
 -c & -d \end{array}\right)\, .
 $$
Then $t_1=\tr \eta=a+d$ and $t_2=\tr \delta=-a+c-d$, so $c=t_2+t_1$.

\begin{figure}
  \centering
    \includegraphics[scale=0.5]{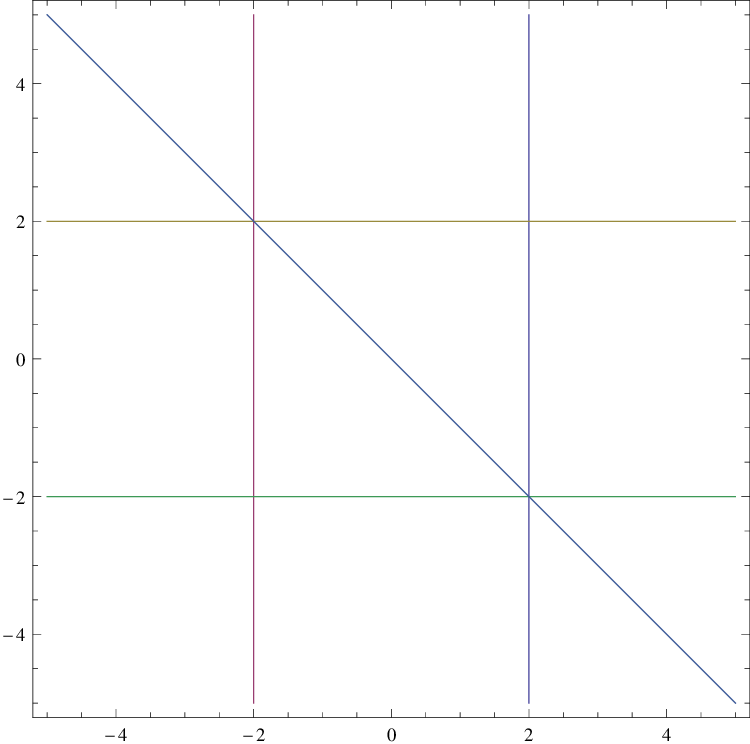}
  \caption{Distinguished lines in the base of the fibration for $Z$.}
  \label{fig:rectasm}
\end{figure}

We stratify $Z$ in the following sets corresponding to the various possibilities
for $(t_1,t_2)$ (see Figure \ref{fig:rectasm}).

\begin{itemize}
 \item {$Z_1$} is the stratum corresponding to $(t_1,t_2)=(2,-2)$.
Then $c=t_1+t_2=0$. So $a=d=1$.
There is a parameter $b\in \CC$. For $b\neq 0, 1$, we have that both $\eta,\delta$ are of Jordan form
(of types $J_+$ and $J_-$, resp.). So this
contributes $(q-2) e(\overline{X}_2)e(\overline{X}_3)$. For $b= 0,1$, one of $\eta,\delta$ is of Jordan form
and the other is diagonal, contributing $e(\overline{X}_2)
e(\overline{X}_1)+e(\overline{X}_3) e(\overline{X}_0)$. Altogether, we have
 \begin{eqnarray*}
 e(Z_1)&=& (q-2)  e(\overline{X}_2)e(\overline{X}_3)+ e(\overline{X}_2) e(\overline{X}_1)+e(\overline{X}_3)
  e(\overline{X}_0) \\
 &=&  2 q^7 + 7 q^6 - 2 q^5- 2 q^4 + 8 q^3 +3 q^2 \, .
 \end{eqnarray*}

\item {$Z_2$}, corresponding to $(t_1,t_2)=(-2,2)$. It is analogous to the first stratum.
We get
 \begin{eqnarray*}
 e(Z_2)&=& (q-2)  e(\overline{X}_2)e(\overline{X}_3)+ e(\overline{X}_2) e(\overline{X}_1)+
 e(\overline{X}_3) e(\overline{X}_0) \\
 &=& 2 q^7 + 7 q^6 - 2 q^5- 2 q^4 + 8 q^3 +3 q^2 \, .
  \end{eqnarray*}

\item {$Z_3$}, corresponding to $t_1=t_2= \pm 2$. Consider first the case $(t_1,t_2)=(2,2)$.
We have that $c=t_1+t_2=4$. Conjugating by the matrix $\left(
    \begin{array}{cc}
      1 & \lambda \\
      0 & 1 \\
    \end{array}
  \right)$,
with $\lambda=d/c$, we can arrange that $d=0$. Therefore $b=-1/c$,
and $a=t_1$. The matrices $\eta,\delta$ are not diagonal, so they are both of Jordan type
$J_+$. Therefore the contribution is $q\, e(\overline{X}_2)^2$. Analogously, the
case $(t_1,t_2)=(-2,-2)$ contributes $q\, e(\overline{X}_3)^2$. Hence
 \begin{eqnarray*}
 e(Z_3)&=&q ( e(\overline{X}_2)^2 + e(\overline{X}_3)^2) \\ 
 &=& 2q^7+2q^6+7q^5+12q^4+9q^3 \, .
 \end{eqnarray*}

\item {$Z_4$}, corresponding to four strata given by the cases $t_1=2$, $t_2\neq \pm 2$;
$t_1=-2$, $t_2\neq \pm 2$; $t_2=2$, $t_1\neq \pm 2$; and $t_2=-2$, $t_1\neq \pm 2$. Let us
focus on the line $L=\{(t_1,t_2)\ | \ t_1=2, t_2\neq \pm 2\}$.

Then $c=t_2+2 \neq 0$, so $\eta$ is of Jordan form $J_+$.
Conjugating as before, we can assume $d=0$, $b=-1/c$ and $a=t_1=2$.
Therefore
  $$
   \eta= \left( \begin{array}{cc} 2 & -\frac1{t_2+2} \\ t_2+2 & 0 \end{array} \right), \
   \delta= \left( \begin{array}{cc} t_2 & \frac1{2+t_2} \\ -t_2-2 & 0 \end{array} \right).
  $$
Working as in the case of the stratum $Z_4$ of section \ref{sec:J+}, we can put the matrix
$\eta$ into Jordan form $J_+$ for all $t_2$ simultaneously, so that the family
parametrizing $(C,D)$ is a trivial family with fiber $\overline{X}_2$ over $L$.
Also, we can make the change of variable
$t_2=\mu+\mu^{-1}$, $\mu \in \CC\setminus \{0,\pm 1\}$, to put $\delta$ into diagonal
form $\left( \begin{array}{cc} \mu& 0 \\ 0 & \mu^{-1} \end{array} \right)$, so that
we see that the family parametrizing $(A,B)$ over $L$ is isomorphic to
$\overline{X}_4/\ZZ_2 \to \CC\setminus \{\pm 2\}$.
So, the substratum of $Z_4$ that we are dealing with is isomorphic to
$\CC \x \overline{X}_2 \x (\overline{X}_4/\ZZ_2)$. The substratum corresponding to
$t_1=-2$, $t_2\neq \pm 2$ is isomorphic $\CC \x \overline{X}_3 \x (\overline{X}_4/\ZZ_2)$.
The remaining two substrata are copies of the previous two. Finally, we get
 \begin{eqnarray*}
 e(Z_4)&=&2q\,
 (e(\overline{X}_2)+e(\overline{X}_3)) e(\overline{X}_{4}/\ZZ_2) \\
 &=&4q^8-6q^7-22q^6+18q^5+28q^4-16q^3-6q^2\, .
 \end{eqnarray*}

\item {$Z_{5}$}, corresponding to $t_1=-t_2\neq \pm 2$. First, we have
$c=t_2+t_1=0$, so $d=a^{-1}$ and $b\in\CC$. So 
 $$
  \eta=\left(\begin{array}{cc}
      a & b\\
      0 & a^{-1}
    \end{array}\right), \ \delta=\left(\begin{array}{cc}
      -a & -b+a^{-1} \\
      0 & -a^{-1}
    \end{array}\right),
    $$
where $t_1=-t_2=a+a^{-1}$. The fibre over $(t_1,-t_1)$ consists of two disjoint
sets, one is isomorphic to $\overline{X}_{4,a}\x \overline{X}_{4,-a}$, the
other isomorphic to $\overline{X}_{4,a^{-1}}\x \overline{X}_{4,-a^{-1}}$. Taking the
double cover $a \mapsto t_1=a+a^{-1}$, we get a fibration over the plane
$B=\CC\setminus \{0,\pm 1\}$ whose fibers are $\overline{X}_{4,a}\x \overline{X}_{4,-a}$.
We call this fibration $\overline{Z}_5 \to B$. Consider $\tau:B\to B$, $\tau(a)=-a$.
The monodromy representation of $\overline{Z}_5$ is
 \begin{eqnarray*}
  R(\overline{Z}_5) &=& R(\overline{X}_4)\otimes \tau^*R(\overline{X}_4)
  = R(\overline{X}_4)\otimes R(\overline{X}_4) \\
  &=& (aT+bN) \otimes (aT+bN) = (a^2+b^2) T + 2ab N \\
  &=& (q^6+9 q^4 -20q^3 + 9 q^2 +1) T + (6q^5 - 6 q^4 -6q^2 + 6q ) N\, ,
 \end{eqnarray*}
where $R(\overline{X}_4)=aT+bN$, $a= q^3-1$, $b=3q^2-3q$ (by Theorem \ref{thm:R(overline)}).
Therefore, Proposition \ref{prop:e-total-space} says that the E-polynomial is
 \begin{eqnarray*}
 e(Z_5) &=& q\,  e(\overline{Z}_5) \\ 
 &=&  q^8-3q^7-3q^6-35q^5+69 q^4-15 q^3-11q^2 -3 q \, .
 \end{eqnarray*}

\item {$Z_6$}, corresponding to the open stratum, characterised by the equations
$t_1,t_2\neq \pm 2$, and $t_1+ t_2\neq 0$.
Then $c \neq 0$. Conjugating by a matrix $\left(
    \begin{array}{cc}
      1 & \lambda \\
      0 & 1 \\
    \end{array}
  \right)$,
with $\lambda=d/c$, we can arrange that $d=0$. Therefore $b=-1/c$,
and $a=t_1$. So the fibre over
$(t_1,t_2)$ has
  $$
   \eta=\left(\begin{array}{cc}
      t_1 & -1/(t_1+t_2)\\
      t_1+t_2 & 0
    \end{array}\right), \ \delta=\left(\begin{array}{cc}
      t_2 &  1/(t_1+t_2)\\
      -t_1-t_2 & 0
    \end{array}\right).
    $$
So each fibre is isomorphic to $\CC\times \overline{X}_{4,\lambda_1}\times \overline{X}_{4,\lambda_2}$,
where $t_i=\lambda_i+\lambda_i^{-1}$, $i=1,2$.
Working as before, we see that, ignoring the condition
$t_1 + t_2\neq 0$, the total
space is isomorphic to $\CC\times (\overline{X}_{4}/\ZZ_2) \times (\overline{X}_{4}/\ZZ_2)$,
so contributing $q\, e(\overline{X}_4/\ZZ_2)^2$. 

We have to subtract the contribution corresponding to the space $Z_6'$ parametrized by
$(t_1,-t_1)$, $t_1\neq \pm2$, and with fibres $\CC\times \overline{X}_{4,\lambda_1}\times
\overline{X}_{4,-\lambda_1}$. Therefore $Z_6'= \CC\times (\overline{W}_4/\ZZ_2)$, where
$\overline{W}_4$ is the space of Section \ref{sec:twisted} (see equation (\ref{38bis})).
Therefore the Hodge monodromy representation $R(\overline{W}_4/\ZZ_2)$ is given by
(\ref{39bis}), and
 \begin{eqnarray*}
  e(\overline{W}_4/\ZZ_2) &=& (q-2)(q^6-18q^3+1) - (6q^5 +3 q^4-2q^3+3q^2+6q) \\
   &=& q^7-2q^6-6q^5-21q^4+38q^3-3q^2-5q-2\, .
 \end{eqnarray*}
 So
 \begin{eqnarray*}
 e(Z_6) &=& q \,( e(\overline{X}_4/\ZZ_2)^2- e(\overline{W}_4/\ZZ_2)) \\
 &=&  q^9 - 5 q^8 + 24 q^6+20q^5-60 q^4+6 q^3+11 q^2+3q\, .
 \end{eqnarray*}
\end{itemize}

Adding up all contributions,
 $$
  e(Z)= q^9-3 q^7+15 q^6+6 q^5+45 q^4 \, . 
 $$
Dividing by the stabilizer, $\stab (J_-)=U/\CC^* \subset \PGL(2,\CC)$,
 $$
 e(\cM_{J_-}) = q^8-3 q^6+15 q^5+6 q^4+45 q^3 \, .  
 $$

\begin{rmk}
 As in Remark \ref{rem:suggestion-Hausel}, we again see that
 $$
 e(\cM_{J_-})+(q+1) e(\cM_{-\Id}) = q^8+q^7-2 q^6+13 q^5-26 q^4+13 q^3-2 q^2+q+1 =
  e(\cM_{\lambda})
  $$
 for genus $g=2$. Such equality is predicted to hold for arbitrary genus $g\geq 1$.
This amusing fact deserves an explanation.
\end{rmk}

\end{document}